\documentclass[11pt,dvipdfmx]{amsart}
\usepackage{amsmath,amssymb,amsthm,stmaryrd,latexsym,amsfonts,mathtools,tikz} 
\usepackage[top=30truemm,bottom=30truemm,left=25truemm,right=25truemm]{geometry}

\newtheorem{thm}{Theorem}[section]
\newtheorem{lemm}[thm]{Lemma}

\newtheorem{prop}[thm]{Proposition}

\newtheorem{rem}[thm]{Remark}

\theoremstyle{plain}

\makeatletter
\@addtoreset{equation}{section}

\makeatother

\newcommand{\1}{\mbox{1}\hspace{-0.25em}\mbox{l}}


\begin{document}

\title[Gibbs measures associated to $P(\Phi)$ QFT model]{ Construction of the Gibbs measures associated with Euclidean quantum field theory with various polynomial interactions in the Wick renormalizable regime }

\author{Hirotatsu Nagoji}
\address{Graduate School of Science, Kyoto University, Kitashirakawa-Oiwakecho, Sakyo-ku, Kyoto 606-8502, Japan}
\email{nagoji.hirotatsu.63x@st.kyoto-u.ac.jp}
\date{}
\thanks{Statements and Declarations: This work was supported by JST SPRING, Grant Number JPMJSP2110. The author has no conflicts of interest directly relevant to the content of this article.}

\begin{abstract}
In this paper, we consider the Gibbs measures associated with Euclidean quantum field theory with polynomial-type of interactions on the torus. We observe the (non-)normalizability of the multivariate version of $P(\Phi)_2$ models by the variational method introduced in \cite{bara}. Moreover, we extend this setting to the ``fractional model'' in the ``Wick renormalizable regime'': the base Gaussian measure has the inverse of the fractional power of the Laplacian as its covariance with the restriction that the Wick products that appear during the renormalization procedure are all well-defined. After working on the issue in a general setting, we also examine some specific examples as an application. 
\end{abstract}

\maketitle


\section{Introduction}
Let $\mu_{\alpha}$ be the Gaussian measure on $\mathcal{D}'(\mathbb{T}^d)$ with the covariance $(1-\Delta)^{-\alpha}$ and we consider the probability measure on $\mathcal{D}'(\mathbb{T}^d)^{\otimes n}$ formally written by
\begin{equation}\label{meas}
`` \mu^F_\alpha (d\Phi) = \frac{1}{Z}e^{ \int_{\mathbb{T}^d} F(\Phi(x))dx} \mu_{\alpha}^{\otimes n}(d\Phi)"
\end{equation}
where $Z$ is the normalizing constant, $F$ is an $n$-variate polynomial and $\mathbb{T}^d= (\mathbb{R}/\mathbb{Z})^d$ is the $d$-dimensional torus. The motivation for the construction of this kind of measures is mainly from quantum field theory. The most important and well-studied model is $\Phi^4_d$ model i.e. when $\alpha =1,\ n=1$, and $F(\Phi) = -\Phi^4$ in the above setting, see \cite{ak,bara,daprato,gh} and references therein, for example. It is straight forward to extend these results to the $n$-variate version of $\Phi^4$ model i.e. when $F(\Phi) = - |\Phi|^4$ where $|\cdot|$ denotes the Euclidian norm on $\mathbb{R}^n$, as commented in \cite[Remark 1.2]{gh}. When $\alpha =1$, the base Gaussian measure $\mu_{\alpha}$ coincides with Nelson's (massive) Gaussian free field. In this case, the models of Euclidian quantum field theory with polynomial-type of interaction like \eqref{meas} are called $P(\Phi)_d$ model.  

Depending on the value of $\alpha$, the expected nature of the measure $\mu_\alpha^F$ in \eqref{meas} can be classified as follows: Let $k$ be the degree of the polynomial $F$.
\begin{itemize}
\item
When $\alpha>\frac{d}{2}$, the density function $e^{ \int_{\mathbb{T}^d} F(\Phi(x))dx}$ is well-defined (in the usual sense) $\mu_{\alpha}^{\otimes n}$-almost-surely because $\mu_{\alpha}\left(L^\infty (\mathbb{T}^d)\right) = 1$. For example, $\Phi^p_1$ model is categorized here.
\item
When $\frac{k-1}{2k}d< \alpha \le \frac{d}{2}$, the density function $e^{ \int_{\mathbb{T}^d} F(\Phi(x))dx}$ is ill-defined because \\
$\mu_{\alpha}\left(H^{\alpha - \frac{d}{2}-\epsilon} (\mathbb{T}^d)\backslash H^{\alpha - \frac{d}{2}} (\mathbb{T}^d)\right) = 1$ for any $\epsilon>0$. However, it is known that we can give natural meaning to the density function by way of the Wick renormalization. To ensure the existence of the Wick powers appearing in the renormalization procedure, we need the restriction $\alpha> \frac{k-1}{2k}d$, see Proposition \ref{gono}. $\Phi^k_2$ model is an example of this case. In this case, the measure $\mu^F_\alpha$ is expected to be absolutely continuous with respect to the base measure $\mu_\alpha$.
\item
When $\alpha \le \frac{k-1}{2k}d$, the density function seems ill-defined even with the Wick renormalization, but there is still hope that the probability measure can be constructed by introducing further renormalization if needed as long as the associated (parabolic) dynamics 
\begin{equation}\label{papapa}
\partial_t \boldsymbol{u} + (1 - \Delta)^\alpha \boldsymbol{u} = (\nabla F)(\boldsymbol{u}) + \boldsymbol{\xi}, 
\end{equation}
where $\boldsymbol{\xi}= (\xi_1,\cdots,\xi_n)$ is an $n$-tuple of mutually independent space-time Gaussian white noises $\{ \xi_i\}_{1\le i \le n}$, is ``subcritical'' i.e. the equation can be solved via renormalization at least locally in time. The equation \eqref{papapa} is usually called (parabolic) stochastic quantization equation in this context. $\Phi^4_3$ model and $\Phi^3_3$ model, which is studied in \cite{phi3}, can be categolized in this case. The work \cite{gh} is mainly concerned with $\Phi^4_3$ model but they also considered the "fractional $\Phi^4$ model" , the case that the covariance of the base Gaussian measure $\mu_\alpha$ is $(1-\Delta)^{-\alpha}$ with $\alpha <1$. 
In this case, it is expected that $\mu^F_{\alpha}$ will be singular with respect to the base measure $\mu_\alpha$ differently from the case $\frac{k-1}{2k}d< \alpha \le \frac{d}{2}$ and it is actually proven to be true in some sense, see \cite{HKN}.
\end{itemize}
In this paper, we consider the ``Wick renormalizable regime'' i.e. the case $\frac{k-1}{2k}d< \alpha \le \frac{d}{2}$. Moreover, we will not consider the dynamical problem \eqref{dyn}, but focus on the construction of the measure $\eqref{meas}$. Our strategy is based on the variational method, which is introduced in \cite{bara} and applied in the works such as \cite{hartree, phi3, heat}.

\subsection{Setting}\label{jai}
In the following, we write polynomial $F$ on $\mathbb{R}^n$ as
\[ F(\boldsymbol{x}) = \sum_{\beta\in A}c_\beta \boldsymbol{x}^\beta = \sum_{\beta=(\beta_1,\cdots, \beta_n)\in A}c_\beta \prod_{i=1}^n x_i^{\beta_i}, \ \ \ \ \ \  \boldsymbol{x}=(x_1, \cdots, x_n)\in \mathbb{R}^n\]
with $c_\beta \in \mathbb{R}\backslash \{0\}$ and $A\subset \mathbb{N}^n$. As mentioned above, we need to introduce the Wick renormalization in order to give meaning to the density function. We define $P_N$ by
\begin{equation}
P_N f \coloneqq  \sum_{l\in\mathbb{Z}^d,\ |l|\le N}\hat{f}(l)e_l, \ \ \ \mbox{for}\ f\in\mathcal{D}'(\mathbb{T}^d)
\end{equation}
where $e_l(x) \coloneqq e^{\sqrt{-1}l\cdot x}$ and $\hat{f}(l)\coloneqq \int_{\mathbb{T}^d} f(x)e_l(x)dx$ denotes the Fourier coefficient of $f$. Then, we define 
\begin{equation}\label{rndef}
R_N^F(\Phi) = R_N(\Phi)\coloneqq \int_{\mathbb{T}^d} :F(P_N^{\otimes n} \Phi): dx \coloneqq \int_{\mathbb{T}^d} \sum_{\beta\in A} c_\beta H_\beta \left( P_N^{\otimes n}\Phi ; \sigma_{\alpha, N} \right) dx,\ \ \ \mbox{for}\ \Phi\in \mathcal{D}'(\mathbb{T}^d)^{\otimes n}
\end{equation}
where 
\begin{gather}\label{dyn}
 \sigma_{\alpha, N} \coloneqq  \sum_{l\in\mathbb{Z}^d,\ |l|\le N}\frac{1}{(1+|l|^2)^\alpha} \asymp
\begin{cases}
 \log N  \ \ \ \ \ \ \ \mbox{when}\ \alpha = \frac{d}{2} \\
 N^{d-2\alpha} \ \ \ \ \ \mbox{otherwise} 
\end{cases}
\end{gather}
is the renormalization constant and $H_\beta$ is the $n$-variate version of the Hermite polynomial, see Section \ref{pre2} for the definition. Thanks to the restriction $\alpha > \frac{k-1}{2k}d$, one can prove that $R_N$ converges $\mu_{\alpha}^{\otimes n}$-almost surely as $N \rightarrow \infty$ and we write its limit as
\begin{equation}
R^F(\Phi) = R(\Phi) = \int_{\mathbb{T}^d} :F(\Phi):(x) dx,
\end{equation}
see Proposition \ref{gono}.
If there holds
\begin{equation}\label{papa}
\int e^{ \int_{\mathbb{T}^d} :F(\Phi):(x)dx} \mu_{\alpha}^{\otimes n}(d\Phi) < \infty,
\end{equation}
we can construct the probability measure $\mu_\alpha^F$ by setting
\begin{equation}\label{pro}
 \mu^F_\alpha (d\Phi) = \frac{1}{Z^F}e^{ \int_{\mathbb{T}^d} :F(\Phi):(x)dx} \mu_{\alpha}^{\otimes n}(d\Phi)
\end{equation}
where $Z^F$ is the normalizing constant defined by
\begin{equation}\label{pro2}
Z^F \coloneqq \int e^{ \int_{\mathbb{T}^d} :F(\Phi):(x)dx} \mu_{\alpha}^{\otimes n}(d\Phi).
\end{equation}
In this case, we say that the measure $\mu_{\alpha}^F$ is normalizable. If \eqref{papa} does not hold, we cannot construct such a probability measure and we say that $\mu_\alpha^F$ is non-normalizable.
In this paper, we consider under what conditions on $F$ $\mu_\alpha^F$ becomes (non-)normalizable. 

\subsection{Main results part 1}\label{jai2}
Before stating our main results, we briefly recall some known results. When $\alpha > \frac{d}{2}$, the density function in \eqref{meas} is well-defined so there is no need to introduce renormalization as mentioned ealier. Therefore, if we assume 
\begin{equation}\label{yuukai}
\sup_{\boldsymbol{x}\in \mathbb{R}^n} F (\boldsymbol{x}) < \infty ,
\end{equation}
the measure $\mu_\alpha^F$ in \eqref{meas} can be realized as a probability measure by choosing proper normalizing constant $Z$. Next, we consider the case $\alpha = \frac{d}{2}$. In this case, we need to introduce Wick renormalization in order to give rigorous meaning to the density function as explained in Section \ref{jai}. Then, even if we assume \eqref{yuukai}, it is not clear whether \eqref{papa} holds because
the condition \eqref{yuukai} does not ensure 
\begin{equation} 
\sup_{\Phi} \int_{\mathbb{T}^d} :F(\Phi):(x)dx < \infty
\end{equation}
due to the involvement of renormalization. However, when $F$ is a $1$-variate polynomial of degree 3 or larger, it is known that \eqref{yuukai} is actually necessary and sufficient condition for the normalizability i.e. \eqref{papa} holds if and only if $F$ is written as
\[ F(x) = \sum_{i=1}^k c_i x^i \] 
with an even integer $k$ and $c_k <0$, see also Section \ref{ex1}. Then, can this result be extended to the case where $F$ is an $n$-variate polynomial? Unfortunately, it turns out that \eqref{yuukai} does not necessarily guarantee the normalizability: Let us consider the case of $F(x_1, x_2)= - x_1^2 x_2^2$. Then, this example satisfies \eqref{yuukai}, but it turns out that there holds
\[ \int e^{ \int_{\mathbb{T}^d} :F(\Phi):(x)dx} \mu_{\frac{d}{2}}^{\otimes n}(d\Phi) = \infty , \]
see Section \ref{ex2}. As seen from this example, the situation in the multivariate case seems a little bit different from that of the 1-variate case. To find out the sufficient condition and necessary condition for \eqref{papa} to hold is the main concern of this paper.

Now, we state our main results.
We write for multi-indices $\alpha=(\alpha_1,\cdots, \alpha_n), \beta=(\beta_1,\cdots,\beta_n)\in\mathbb{N}^n$, 
\[ |\alpha| \coloneqq \alpha_1 + \cdots + \alpha_n, \] 
\[ \alpha < \beta \iff \alpha_i \le \beta_i\ \ \ \mbox{for any}\ 1\le i \le n \ \mbox{and}\ \alpha \neq \beta \] and define 
\begin{equation}
 A^- \coloneqq \{ \gamma \in \mathbb{N}^n \ ;\ \gamma < \beta \ \mbox{for some}\ \beta\in  A \}.
\end{equation}
We denote the convex hull of $S\subset \mathbb{R}^n$ by $\mbox{Conv}(S)$, see \eqref{conv}.
Moreover, given $q:A^- \rightarrow \mathbb{R}_+$, we define the map
$p^{\alpha, F, q}:  A^- \rightarrow \mathbb{R}_+$ by
\begin{equation}\label{ac1}
p^{\alpha, F, q}(\beta) \coloneqq  1 \vee \max_{\xi\in  A,\ \xi >\beta}\frac{2\alpha -(d -2\alpha)|\xi - \beta|}{2\alpha -(d-2\alpha)|\xi -\beta|b(\beta)} 
\end{equation}
with
\begin{equation}\label{ac2}
b(\beta) \coloneqq \max_{1\le i \le n} \min_{(\epsilon,\tilde \epsilon)\in R^i(\beta)}\left( \frac{1+\sum_{j=1}^n \epsilon_j}{2}+ \sum_{j=1}^n \frac{\tilde \epsilon_j}{ \kappa_j(q)} \right),
\end{equation}
\begin{equation}\label{kawadou}
\kappa_j(q)\coloneqq \max \left\{ d\ ;\ d\in \mathbb{R}_+,\ d e_j \in \mbox{\textup{Conv}}\left( \bigcup_{\beta\in A^-}\{ q(\beta)\beta\} \cup\{ \boldsymbol{0}\} \right) \right\},
\end{equation}
and
\begin{align} 
&\displaystyle R^i(\beta) \coloneqq \Bigg\{ (\epsilon,\tilde \epsilon) \in (\mathbb{N}^n)^2\ ;\ (\epsilon_j , \tilde \epsilon_j) = (0,0)\  \mbox{if}\ \beta_j =0,\ \frac{d-2\alpha}{2d}\sum^n_{j=1}\epsilon_j + \sum_{j=1}^n \frac{\tilde \epsilon_j}{\kappa_j(q)} \le \frac{1}{2}  \notag \\
&\ \ \ \ \ \ \ \ \ \ \ \ \ \ \ \ \ \ \ \ \ \ \ \ \ \ \ \ \ \ \ \ \ \ \ \ \ \ \ \ \ \ \ \ \ \ \ \mbox{and}\ \epsilon_j + \tilde \epsilon_j =
\left\{ \begin{array}{l}
\beta_j \ \ \ \ \ \  \mbox{for}\ j \neq i \\
\beta_j -1 \  \mbox{for}\ j = i 
\end{array}\right.\ \mbox{if}\ \beta_j \neq 0 \Bigg\}  \label{ac3} 
\end{align}
if $\beta_i \neq 0$ and $R^i(\beta)\coloneqq\{ 0\}\subset (\mathbb{N}^n)^2$ if $\beta_i = 0$.

\begin{thm}\label{norA}
Let $\frac{k-1}{2k}d < \alpha \le \frac{d}{2}$. 
If there exist some $q: A^- \rightarrow \mathbb{R}_+$ satisfying $q > p^{\alpha, F , q}$ and $m < \frac{1}{2}$ such that $F$ satisfies 
\begin{equation}\label{lalala}
\sup_{\boldsymbol{x}\in \mathbb{R}^n} \left(  F (\boldsymbol{x}) + \sum_{\beta\in A^-}|\boldsymbol{x}^\beta|^{q(\beta)} -m|\boldsymbol{x}|^2 \right)<  \infty,
\end{equation}
there holds
\[ \int e^{ \int_{\mathbb{T}^d} :F(\Phi):(x)dx} \mu_{\alpha}^{\otimes n}(d\Phi) < \infty . \]
\end{thm}

\begin{rem}
The definition of the function $p^{\alpha,F,q}:  A^- \rightarrow \mathbb{R}_+$ appearing in Theorem \ref{norA} is a bit complicated, but  when $\alpha = \frac{d}{2}$, it takes quite a simple form: $p^{\alpha,F,q} \equiv 1$, see also Theorem \ref{nornon} below. 
\end{rem}

\begin{rem}
The condition on $m$ comes from the covariance $(1-\Delta)^{-\alpha}$ of the Gaussian field $\mu_\alpha$. If we consider $(m_0^2-\Delta)^{-\alpha}$ with $m_0^2>0$ instead of $(1-\Delta)^{-\alpha}$ as the covariance of $\mu_\alpha$, the condition $m<\frac{1}{2}$ in Theorem \ref{norA} is replaced with $m<\frac{m_0^2}{2}$, see also Remark \ref{2jou}.
\end{rem}

\begin{rem}
It is easy to see that the assumption on $F$ in Theorem \ref{norA} is equivalent to the following condition:
\begin{equation}
 F (\boldsymbol{x})  - m|\boldsymbol{x}|^2 \le -\sum_{\beta \in  A^-} |\boldsymbol{x}^\beta|^{q(\beta)} + C
\end{equation}
uniformly in $\boldsymbol{x}\in \mathbb{R}^n$ for some $q > p^{\alpha,F, q}$ and $C>0$. See Section \ref{exel} for other characterizations of this condition in some specific cases.
\end{rem}

Next, we give a criterion for the non-normalizability of the measure $\mu^F_\alpha$.

\begin{thm}\label{nonA}
Let $\frac{k-1}{2k}d < \alpha < \frac{d}{2}$. If there exist some $q: A^- \rightarrow \mathbb{R}_+$, $\boldsymbol{a}=(a_1,\cdots , a_n)\in \mathbb{R}^n$, $\boldsymbol{r}= (r_1, \cdots, r_n)\in \mathbb{N}^n$ satisfying the inequality \eqref{paka} below and $m>\frac{1}{2}$ such that 
\begin{equation}
\sup_{x\in\mathbb{R}_+}\left( F (\boldsymbol{a}_{\boldsymbol{r}}(x)) + \sum_{\beta \in  A^-} |\boldsymbol{a}_{\boldsymbol{r}}(x)^\beta|^{q(\beta)} -m|\boldsymbol{a}_{\boldsymbol{r}}(x)|^2  \right) = \infty,
\end{equation}
there holds
\[ \int e^{ \int_{\mathbb{T}^d} :F(\Phi):(x)dx} \mu_{\alpha}^{\otimes n}(d\Phi) = \infty, \]
where $\boldsymbol{a}_{\boldsymbol{r}}(x)  \coloneqq (a_1x^{r_1}, \cdots, a_n x^{r_n})\in\mathbb{R}^n$ and 
\begin{equation}\label{paka}
\min_{(\eta\in A^-,\ \boldsymbol{a}^\eta \neq 0)} \max_{(\beta\in A^-,\ \xi\in  A,\ \xi > \beta,\ \boldsymbol{a}^{\beta}\neq 0 )}  \left\{ \frac{|\xi - \beta|}{2}(d-2\alpha) + \frac{\beta\cdot \boldsymbol{r}}{q (\eta)(\eta \cdot \boldsymbol{r})}d  \right\}> d.
\end{equation}
\end{thm}
\begin{rem}
For example, $q\equiv 1$ satisfies \eqref{paka} when $\alpha < \frac{d}{2}$ if $\{ \eta\in A^-;\ \boldsymbol{a}^\eta \neq 0  \}\neq \emptyset$ and $\{ (\beta, \xi)\in  A^- \times A;\ \xi > \beta,\ \boldsymbol{a}^{\beta}\neq 0 \} \neq \emptyset$.
\end{rem}

Note that the case $\alpha = \frac{d}{2}$ is excluded in Theorem \ref{nonA}. Actually, the statement of this theorem is also true for the case $\alpha = \frac{d}{2}$, but there holds a slightly stronger claim in this case.
\begin{thm}\label{nornon}
Let $\alpha = \frac{d}{2}$.
\begin{enumerate}
\item 
If there exist some $q: A^- \rightarrow \mathbb{R}_+$ with $q > 1$ and $m <\frac{1}{2}$ such that $F$ satisfies 
\begin{equation}\label{nornon8}
\sup_{\boldsymbol{x}\in \mathbb{R}^n} \left(  F (\boldsymbol{x}) + \sum_{\beta\in A^-}|\boldsymbol{x}^\beta|^{q(\beta)} -m|\boldsymbol{x}|^2 \right)<  \infty,
\end{equation}
there holds
\[ \int e^{ \int_{\mathbb{T}^d} :F(\Phi):(x)dx} \mu_{\alpha}^{\otimes n}(d\Phi) < \infty. \]
\item
If there exist some  $\boldsymbol{a}=(a_1,\cdots , a_n)\in \mathbb{R}^n$, $\boldsymbol{r}= (r_1, \cdots, r_n)\in \mathbb{N}^n$, $m>\frac{1}{2}$ and $C>0$ such that 
\begin{equation}\label{nornon9}
\sup_{x\in\mathbb{R}_+}\left( F (\boldsymbol{a}_{\boldsymbol{r}}(x)) + C\sum_{\beta \in A^-} |\boldsymbol{a}_{\boldsymbol{r}}(x)^\beta| -m|\boldsymbol{a}_{\boldsymbol{r}}(x)|^2   \right) = \infty,
\end{equation}
then, there holds
\[ \int e^{ \int_{\mathbb{T}^d} :F(\Phi):(x)dx} \mu_{\alpha}^{\otimes n}(d\Phi) = \infty. \]
\end{enumerate}
\end{thm}

\begin{rem}
Theorem \ref{nornon} (i) is just the restatement of Theorem \ref{norA} in the case $\alpha = \frac{d}{2}$.
\end{rem}

\begin{rem}
As you can see, there is still a room for improvement in these results i.e. there is a gap between the sufficinet condition for the normalizability and the one for the non-normaliability. But still we can determine whether $\mu^F_\alpha$ can be realized as a probability measure for a large class of polynomial $F$ from our results, see Section \ref{exel}. Especially when $\alpha = \frac{d}{2}$, it turns out that the assumptions on $F$ in Theorem \ref{nornon} are actually necessary and sufficient condition for (non-)normalizability in specific cases such as the case where $F$ is a homogenous polynomial, see Section \ref{ex3}.
\end{rem}

\subsection{Main results part 2: Focusing Gibbs measures and the grand canonical form}
As explained in the beginning of the previous section, when $F$ is a $1$-variate polynomial 
\[ F(x) = \sum_{i=1}^k c_i x^i \] 
of degree $3$ or larger, \eqref{papa} can hold only in the defocusing case i.e. when $ k \ \mbox{is even and}\ c_k <0 .$
And in the focusing case i.e. when $ k \ \mbox{is odd or}\ c_k >0$, \eqref{papa} does not hold. However, it is proved in \cite{nons,heat} that $\Phi^3_2$-measure (or more generally, the log-correlated Gibbs measure with cubic interaction) becomes normalizable after introducing further modifications of the density function called (the Wick-ordered) $L^2$-cutoff or by considering the ``grand canonical form'' of the measure. See \cite{hartree, phi3, heat, nons} and references therein for more about the focusing Gibbs measures. Also in our setting, we can construct Gibbs measures with more varied types of interaction than in the previous section by allowing these modifications:   
\begin{itemize}
\item
(The Wick-ordered $L^2$-cutoff)
\begin{equation}\label{cut}
\nu_\alpha^F (d\Phi) = \frac{1}{Z} \1_{\left\{ \left|\int_{\mathbb{T}^d}:|\Phi|^2:(x)dx\right| \le K\right\}}e^{ \int_{\mathbb{T}^d} :F(\Phi):(x)dx} \mu_{\alpha}^{\otimes n }(d\Phi) 
\end{equation}
with $K>0$, where 
\begin{equation}
\int_{\mathbb{T}^d}:|\Phi|^2:(x)dx \coloneqq  \sum_{i=1}^n\int_{\mathbb{T}^d}:\Phi_i^2:(x)dx\ \ \ \ \ \ \ \mbox{for}\ \Phi = (\Phi_1,\cdots, \Phi_n)\in\mathcal{D}'(\mathbb{T}^d)^{\otimes n}.
\end{equation} 
\item
(The grand canonical form)
\begin{equation}\label{grand}
 \nu_\alpha^F (d\Phi) = \frac{1}{Z} e^{ \int_{\mathbb{T}^d} :F(\Phi):(x)dx -K \left|\int_{\mathbb{T}^d} :| \Phi |^2:(x)dx \right|^b} \mu_{\alpha}^{\otimes n}(d\Phi) 
\end{equation}
with $K, b>0$.
\end{itemize}
In the following, we only state the results relevant to the grand canonical form \eqref{grand} but the same kind of results are expected to hold in terms of the one with $L^2$-cutoff \eqref{cut} too, see \cite{heat} for the relation between two measures \eqref{cut} and \eqref{grand}. 

Now, we state the main results in a similar manner to the previous section. 

\begin{thm}\label{nornon1}
Let $\frac{k-1}{2k}d < \alpha < \frac{d}{2}$. 
\begin{enumerate}
\item 
If there exist some $q: A^- \rightarrow \mathbb{R}_+$ satisfying $q > \tilde p^{\alpha,F,q}$, $\kappa < \frac{2d}{d-\alpha}$ and $C>0$ such that
\begin{equation}\label{lalalag}
\sup_{\boldsymbol{x}\in \mathbb{R}^n} \left(  F (\boldsymbol{x}) + \sum_{\beta\in A^-}|\boldsymbol{x}^\beta|^{q(\beta)} - C|\boldsymbol{x}|^\kappa \right)<  \infty,   
\end{equation}
there holds
\[ \int e^{ \int_{\mathbb{T}^d} :F(\Phi):(x)dx -K \left|\int_{\mathbb{T}^d} :| \Phi |^2:(x)dx \right|^{b}} \mu_{\alpha}^{\otimes n}(d\Phi) < \infty  \]
for some $K>0$ and $0<b<\frac{d}{d-2\alpha}$, where $\tilde p^{\alpha,F,q}$ is defined through \eqref{ac1}, \eqref{ac2} and \eqref{ac3} similarly to $p^{\alpha, F, q} (\beta)$ in Theorem \ref{norA} but under the replacement of $\kappa_i(q)$ with 
\begin{equation}\label{kawadou2}
\tilde \kappa_i (q) \coloneqq \max \left\{ d\ ;\ d\in \mathbb{R}_+,\ d e_j \in \mbox{\textup{Conv}}\left( \bigcup_{\beta\in A^-}\{ q(\beta)\beta\} \cup \bigcup_{1\le i \le n} \{ \kappa e_i \}\cup\{ \boldsymbol{0}\}\right) \right\} .
\end{equation}
\item
If there exist some $q:A^- \rightarrow \mathbb{R}_+$, $\boldsymbol{a}=(a_1,\cdots , a_n)\in \mathbb{R}^n$ and $\boldsymbol{r}= (r_1, \cdots, r_n)\in \mathbb{N}^n$ satisfying the inequality \eqref{paka2} below such that 
\[\sup_{x\in\mathbb{R}_+}\left( F (\boldsymbol{a}_{\boldsymbol{r}}(x)) + \sum_{\beta \in  A^-} |\boldsymbol{a}_{\boldsymbol{r}}(x)^\beta|^{q(\beta)} -C |\boldsymbol{a}_{\boldsymbol{r}}(x)|^{\frac{2d}{d-\alpha}}  \right) = \infty   \]
for any $C>0$, there holds
\[ \int e^{ \int_{\mathbb{T}^d} :F(\Phi):(x)dx -K \left|\int_{\mathbb{T}^d} :| \Phi |^2:(x)dx \right|^b} \mu_{\alpha}^{\otimes n}(d\Phi) = \infty   \]
for any $K>0$ and $0<b\le \frac{d}{d-2\alpha}$, where $\boldsymbol{a}_{\boldsymbol{r}}(x) = (a_1 x^{r_1}, \cdots, a_n x^{r_n})$ and
\begin{equation} \label{paka2}
\min_{(\eta\in A^-,\ \boldsymbol{a}^\eta \neq 0)} \max_{(\beta\in A^-,\ \xi\in  A,\ \xi > \beta,\ \boldsymbol{a}^{\beta}\neq 0 )}  \left\{ \frac{|\xi - \beta|}{2}(d-2\alpha) + \frac{2(\beta\cdot\boldsymbol{r})}{q(\eta)(\eta\cdot\boldsymbol{r})}d -d \right\} > d.
\end{equation}
\end{enumerate}
\end{thm}
\begin{rem}
From Theorem \ref{nornon1} (ii), we can only say about the non-normalizability for $0<b\le \frac{d}{d-2\alpha}$. The author expects that the statement can be extended to the one with $0<b<\infty$, but has not been able to do it because of a technical difficulty in the proof. 
\end{rem}
We excluded the case $\alpha = \frac{d}{2}$ in Theorem \ref{nornon1} because there holds a slightly stronger assertion in that case similarly to the previous section. 

\begin{thm}\label{nono}
Let $\alpha = \frac{d}{2}$.
\begin{enumerate}
\item 
If there exist some $q: A^- \rightarrow \mathbb{R}_+$ with $q > 1$, $\kappa < 4$ and $C>0$ such that
\[ \sup_{\boldsymbol{x}\in \mathbb{R}^n} \left(  F  (\boldsymbol{x}) + \sum_{\beta\in A^-}|\boldsymbol{x}^\beta|^{q(\beta)} -C|\boldsymbol{x}|^\kappa \right)<  \infty \]
there holds
\[ \int e^{ \int_{\mathbb{T}^d} :F(\Phi):(x)dx -K \left|\int_{\mathbb{T}^d} :| \Phi |^2:(x)dx \right|^{b}} \mu_{\alpha}^{\otimes n}(d\Phi) < \infty  \]
for some $K>0$ and $b>0$.
\item
If there exist some $\boldsymbol{a}=(a_1,\cdots , a_n)\in \mathbb{R}^n$, $\boldsymbol{r}= (r_1, \cdots, r_n)\in \mathbb{N}^n$ and $C_1>0$ such that 
\[ \sup_{x\in\mathbb{R}_+}\left( F (\boldsymbol{a}_{\boldsymbol{r}}(x)) + C_1\sum_{\beta\in A^-}|\boldsymbol{a}_{\boldsymbol{r}}(x)^\beta| -C_2 |\boldsymbol{a}_{\boldsymbol{r}}(x)|^{\kappa} \right) = \infty\]
for any $\kappa <4$ and $C_2>0$, there holds
\[ \int e^{ \int_{\mathbb{T}^d} :F(\Phi):(x)dx -K \left|\int_{\mathbb{T}^d} :| \Phi |^2:(x)dx \right|^b} \mu_{\alpha}^{\otimes n}(d\Phi) = \infty   \]
for any $K>0$ and $b>0$. 
\end{enumerate}
\end{thm}
\begin{rem}
Like the results in the previous section, there is a gap between the assumptions of Theorem \ref{nornon1} (i) and (ii) i.e. there are examples of polynomial $F$ not satisfying both of these assumption. We will make a few observations on such examples, see Section \ref{ex4}.
\end{rem}

\subsection{The organization of the paper} In Section \ref{exel}, we examine some examples as an application of our main theorems. In Section \ref{katana}, we briefly recall the basics of Sobolev spaces and  Gaussian random variables, and the Bou\'e-Dupuis formula, which is the key to our method. In Sections \ref{norma1}, \ref{norma2}, \ref{norma3} and \ref{norma4}, we prove the main results: In Section \ref{norma1}, we prove Theorem \ref{norA}. In Section \ref{norma2}, we prove Theorems \ref{nornon} (ii) and \ref{nonA}. In Section \ref{norma3}, we prove Theorem \ref{nornon1} (i). In Section \ref{norma4}, we prove Theorems \ref{nono} (ii) and \ref{nornon1} (ii). Section \ref{criti} is devoted to the proof of Theorem \ref{corocoro}, which is introduced in Section \ref{exel} and states the existence of the example that shows critical behavior. As an appendix, we give the proof of the convergence of Wick powers in Section \ref{appen}.

\subsection{Notations}
\begin{itemize}
\item
$\mathbb{N} \coloneqq \{ 0,1,2,\cdots \},\ \mathbb{R}_+ \coloneqq \{ x\in \mathbb{R}\ ;\ x\geq 0 \}$.
\item
For $\boldsymbol{x}=(x_1,\cdots,x_n),\ \boldsymbol{y}=(y_1,\cdots,y_n)\in\mathbb{R}^n$,
\[ \boldsymbol{x}\leq \boldsymbol{y} \iff x_i \leq y_i \ \mbox{for any}\ 1 \leq i \leq n \]
and 
\[ \boldsymbol{x} < \boldsymbol{y} \iff \boldsymbol{x} \neq \boldsymbol{y} \ \mbox{and}\   x_i \leq y_i \ \mbox{for any}\ 1 \leq i \leq n.\]
\item 
$\mathcal{D}'(\mathbb{T}^d)$ denotes the topological dual of $C^\infty(\mathbb{T}^d)$.
\item
We write $a\lesssim b$ if there holds $a\le C b$ for some constant $C>0$ independent of the variables under consideration. We also write $a \simeq b$ if $a \lesssim b$ and $b \lesssim a$.
\item
For two sequences $a_N$ and $b_N$, we write
\begin{align}
a_N \asymp b_N \iff 0< \liminf_{N\rightarrow \infty} \frac{a_N}{b_N} \le \limsup_{N\rightarrow \infty} \frac{a_N}{b_N} < \infty.
\end{align}
\item
For $S \subset \mathbb{R}^n$,
\begin{align} 
&\mbox{Conv}(S) \notag \\
&\coloneqq \left\{ \boldsymbol{x} \in \mathbb{R}^n\ ;\ \boldsymbol{x} = \sum_{i=1}^{M}z_i \boldsymbol{y}_i \ \mbox{for some}\ \boldsymbol{y}_1, \cdots, \boldsymbol{y}_M \in S \ \mbox{and}\ z_1, \cdots, z_M \in \mathbb{R}_+ \ \mbox{with}\  \sum_{i=1}^{M}z_i = 1. \right\} \label{conv} 
\end{align}
is a convex hull of $S$.
\end{itemize}

\noindent \textbf{Acknowledgement.}
The author would like to thank Professor Sergio Albeverio and Professor Seiichiro Kusuoka for introducing me to the problem that inspired this study and for helpful comments on the manuscript.

\section{Examples}\label{exel}
In this section, we consider some specific examples and see what can be said for them as an application of our main results.

\subsection{$1$-variate polynomials}\label{ex1}
Consider the case that $\alpha = \frac{d}{2}$, $n=1$ and 
\[ F(x) = \sum_{i=1}^k c_i x^i \] 
with $c_k \neq 0$ and $k\ge 3$.
In this case, we can recover the known results from Theorem \ref{nornon}:
\begin{align}
&\int e^{ \int_{\mathbb{T}^d} :F(\Phi):(x)dx} \mu_{\alpha}(d\Phi) < \infty  \label{lay} \\
& \iff \sup_{x\in\mathbb{R}} F(x) < +\infty \iff c_k < 0 \ \mbox{and}\ k\ \mbox{is even}. \notag
\end{align}
Indeed, if $c_k < 0$ and $k$ is even, taking $q$ such that $1< q < \frac{k}{k-1}$,
\[ \sup_{x\in \mathbb{R}} \left(F(x) + \sum_{i=1}^{k-1}|x^i|^q - m |x|^2    \right) < \infty \]
and this means that the condition \eqref{nornon8} holds. On the other hand, if $c_k >0$ or $k\ge 3$ is odd,
\[ \sup_{x\in \mathbb{R}} \left( F(x) - \frac{1}{2}|x|^2   \right) = \infty   \]
and this especially means that the condition \eqref{nornon9} holds in this case.
\begin{rem}\label{2jou}
When $k=2$, we can see that \eqref{lay} holds if $c_k <\frac{1}{2}$ and does not hold if $c_k >\frac{1}{2}$. The author does not know what happens in the case $c_k = \frac{1}{2}$ and $k=2$.
\end{rem}
Moreover, we can see from Theorem \ref{nono} that 
\[ \int e^{ \int_{\mathbb{T}^d} :F(\Phi):(x)dx -K \left|\int_{\mathbb{T}^d} :| \Phi |^2:(x)dx \right|^{b}} \mu_{\alpha}(d\Phi) < \infty  \]
for some $K>0$ and $b>0$ if $k=3$, and 
\[ \int e^{ \int_{\mathbb{T}^d} :F(\Phi):(x)dx -K \left|\int_{\mathbb{T}^d} :| \Phi |^2:(x)dx \right|^{b}} \mu_{\alpha}(d\Phi) = \infty  \]
for any $K>0$ and $b>0$ if ``$c_k>0$ or $k$ is odd'' and $k\ge 4$. 
Indeed, if $k=3$,
\[ \sup_{x\in \mathbb{R}} \left( F(x) - |x|^{\frac{7}{2}}   \right) < \infty   \]
and in the latter case,
\[ \sup_{x\in \mathbb{R}} \left( F(x) - C|x|^\kappa   \right) = \infty   \]
for any $\kappa <4$ and $C>0$. Therefore, we can apply Theorem \ref{nono} in these cases and we get the above result.
Note that the normalizability of $\Phi^3_2$ measure and the non-normalizability of the focusing $\Phi^k_2$ measure with $k\ge 4$ are already known, see \cite{heat} and references therein.



\subsection{Monomials}\label{ex2}
Consider the case that $\alpha = \frac{d}{2}$ and 
\[ F(\boldsymbol{x}) = a\boldsymbol{x}^\beta,\ \ \ \ \boldsymbol{x}\in\mathbb{R}^n   \]
for some $a\in\mathbb{R}\backslash \{ 0\}$ and $\beta\in\mathbb{N}^n$ with $|\beta|= \beta_1 + \cdots + \beta_n \ge 3$, $\beta_1 \neq 0$ and $\beta_2 \neq 0$. Note that the case where $\beta_1 \neq 0$ and $\beta_i =0$ for all $i \in \{ 2, \cdots, n\}$ is already treated in Example \ref{ex1}. In this case, we can get the following result.
\begin{thm}
In the above setting, the following statements hold.
\begin{enumerate}
\item
There holds
\begin{align*}
\int e^{a \int_{\mathbb{T}^d} :\Phi^\beta:(x)dx} \mu_{\alpha}^{\otimes n}(d\Phi) = \infty .
\end{align*}
\item
There holds
\begin{align*}
\int e^{a \int_{\mathbb{T}^d} :\Phi^\beta:(x)dx-K \left|\int_{\mathbb{T}^d} :| \Phi |^2:(x)dx \right|^{b}} \mu_{\alpha}^{\otimes n}(d\Phi) < \infty 
\end{align*}
for some $K,b>0$ if and only if $|\beta|=3$ or ``$a<0$ and $\beta=(2,2,0,\cdots,0)=2e_1+2e_2$''.
\end{enumerate}
\end{thm}
\begin{proof}
Because we are assuming $\beta_2 >0$, we have $(x,0,x,\cdots, x)^\beta = x^{\beta_1}0^{\beta_2}\cdots x^{\beta_n} = 0$ for any $x \in \mathbb{R}$ and $(\beta_1,0, \cdots,0) = \beta_1 e_1 < \beta_1 e_1 + \beta_2 e_2 \le \beta$. Thus, we can see that if $\beta_1 \ge 2$ and $C>n-1$,
\begin{align*}
&\sup_{\boldsymbol{x}\in\mathbb{R}^n, \boldsymbol{x} = (x, 0, x, \cdots, x)= \sum_{i\neq 2}xe_i} \left( a\boldsymbol{x}^\beta + C\sum_{\gamma < \beta} |\boldsymbol{x}^\gamma| -|\boldsymbol{x}|^2 \right)\\
&\ge \sup_{\boldsymbol{x}\in\mathbb{R}^n, \boldsymbol{x} = (x, 0, x, \cdots, x)} \left( a\boldsymbol{x}^\beta + C|\boldsymbol{x}^{\beta_1 e_1}| -|\boldsymbol{x}|^2  \right) = \sup_{x\in\mathbb{R}} \left\{ C|x^{\beta_1}|-(n-1)|x|^2\right\} = \infty .
\end{align*} 
Now, we recall that $\boldsymbol{a}_{\boldsymbol{r}}(x)$ is defined in Theorem \ref{nonA}. When $\beta_i \le 1$ for any $1\le i \le n$, by setting $\boldsymbol{r}_1 \coloneqq (1, \cdots, 1) = \sum_{1\le i \le n}e_i$ and noting that we are assuming $|\beta|\ge 3$,
\begin{align*}
&\sup_{x\in\mathbb{R}} \left( a \boldsymbol{a}_{\boldsymbol{r}_1} (x)^\beta - |\boldsymbol{a}_{\boldsymbol{r}_1} (x)|^2  \right) = \sup_{x\in\mathbb{R}} \left( a \boldsymbol{a}^\beta x^{|\beta|}     - |\boldsymbol{a}_{\boldsymbol{r}_1} (x)|^2  \right) = \infty
\end{align*}
 if we choose $\boldsymbol{a} = (a_1, \cdots, a_n)$ such that $a \boldsymbol{a}^\beta = a \left( \prod_{i=1}^n a_i^{\beta_i}  \right)>0$, which is possible because $\beta_i \in \{ 0,1\}$ and especially $\beta_1 = 1$. Therefore, assertion (i) follows from Theorem \ref{nornon} (ii).

Next, we prove assertion (ii). If $|\beta|=3$, it is easy to see that from Theorem \ref{nono} (i) ,
\begin{equation}\label{cat}
\int e^{a \int_{\mathbb{T}^d} :\Phi^\beta:(x)dx -K \left|\int_{\mathbb{T}^d} :| \Phi |^2:(x)dx \right|^{b}} \mu_{\alpha}^{\otimes n}(d\Phi) < \infty
\end{equation}
for some $K>0$ and $b>0$. When $|\beta|\ge 4$ and  ``$\beta_i$ is odd  for some $1\le i \le n$ or $a>0$'', \eqref{cat} does not follow from Theorem \ref{nono} (ii) because we can see that
by setting $\boldsymbol{r}_1 \coloneqq (1, \cdots, 1) = \sum_{1\le i \le n}e_i$,
\begin{align*}
&\sup_{x\in\mathbb{R}} \left( a \boldsymbol{a}_{\boldsymbol{r}_1} (x)^\beta - C|\boldsymbol{a}_{\boldsymbol{r}_1} (x)|^\kappa  \right) =\sup_{x\in\mathbb{R}} \left( a \boldsymbol{a}^\beta x^{|\beta|}  - C|\boldsymbol{a}_{\boldsymbol{r}_1} (x)|^\kappa  \right) = \infty
\end{align*}
for any $C>0$ and $\kappa<4$ if we choose $\boldsymbol{a} = (a_1, \cdots, a_n)$ such that $a\boldsymbol{a}^\beta =a \left( \prod_{i=1}^n a_i^{\beta_i}  \right)>0$, which is possible when $\beta_i$ is odd for some $1\le i \le n$ or $a>0$.

Finally, we consider the case that $\beta_i$ is even for any $1\le i \le n$ and $a<0$. When $\beta = (2,2,0,\cdots, 0) \iff |\beta| =4$, there holds 
\begin{align*}
&\sup_{\boldsymbol{x}\in\mathbb{R}^n} \left( a\boldsymbol{x}^\beta + \sum_{\gamma < \beta}|\boldsymbol{x}^\gamma|- |\boldsymbol{x}|^{\frac{7}{2}} \right) \\
&\lesssim \sup_{\boldsymbol{x}\in\mathbb{R}^n} \left(  ax_1^2 x_2^2 + |x_1^2x_2| + |x_1^2| + |x_1x_2| + |x_1x_2^2| + |x_2^2| + |x_1| + |x_2| - |x_1|^{\frac{7}{2}} - |x_2|^{\frac{7}{2}} \right) < \infty.
\end{align*}
Therefore, from Theorem \ref{nono} (i), \eqref{cat} holds in this case. When $|\beta| \ge 6$, we can prove that for some $\boldsymbol{a}\in\mathbb{R}^n$, 
\[ \sup_{x\in\mathbb{R}} \left( a \boldsymbol{a}_{\boldsymbol{r}_1} (x)^\beta + \sum_{\gamma <\beta} |\boldsymbol{a}_{\boldsymbol{r}_1} (x)^\gamma| - C|\boldsymbol{a}_{\boldsymbol{r}_1} (x)|^\kappa  \right) = \infty \]
for any $C>0$ and $\kappa<4$, where $\boldsymbol{r}_1 \coloneqq (1, \cdots, 1) = \sum_{1\le i \le n}e_i$. Indeed, letting $i_0 \in \{ 1, \cdots n\}$ and $\boldsymbol{a} = (a_1, \cdots, a_n) \in \mathbb{R}^n$ be such that $\beta_{i_0} = \min \{ \beta_i \ ;\ 1\le i\le n, \beta_i >0 \}$ and $a_{i_0} =0$,
\begin{align}
&\sup_{x\in\mathbb{R}} \left( a \boldsymbol{a}_{\boldsymbol{r}_1} (x)^\beta + \sum_{\gamma <\beta} |\boldsymbol{a}_{\boldsymbol{r}_1} (x)^\gamma| - C|\boldsymbol{a}_{\boldsymbol{r}_1} (x)|^\kappa  \right)\notag\\
&\ge \sup_{x\in\mathbb{R}} \left( a  \Big( \prod_{i \neq i_0} a_i^{\beta_i}  \Big) 0^{\beta_{i_0}}x^{|\beta|} + |\boldsymbol{a}_{\boldsymbol{r}_1} (x)^{\beta - \beta_{i_0}}| - C|\boldsymbol{a}_{\boldsymbol{r}_1} (x)|^\kappa  \right)\notag \\
&\ge \sup_{x\in\mathbb{R}} \left(\left| \Big(  \prod_{i \neq i_0} a_i^{\beta_i}  \Big) x^{|\beta | -\beta_{i_0}}   \right| - C|\boldsymbol{a}_{\boldsymbol{r}_1} (x)|^\kappa  \right). \label{saishuu}
\end{align}
Here, the supremum in the right hand side of \eqref{saishuu} is infinite for any $\kappa <4$ and $C>0$ when we choose $\boldsymbol{a} \in \mathbb{R}^n$ such that $\prod_{i \neq i_0} a_i^{\beta_i} \neq 0$ because $|\beta| - \beta_{i_0} \ge 4$ in this case. Indeed, when $\beta_{i_0} \le 2$, we have  $|\beta| - \beta_{i_0} \ge 4$, noting that we are assuming $|\beta| \ge 6$. When $\beta_{i_0} > 2$, we can take $i_1 \neq i_0$ such that $\beta_{i_1} \ge 4$ from the way that we chose $i_0$ because all $\beta_i$ are even. This especially means that $|\beta| - \beta_{i_0} \ge \beta_{i_1} \ge 4$.

\end{proof}

\subsection{Homogenous polynomials}\label{ex3}
Consider the case where $\alpha = \frac{d}{2}$ and $F$ is a homogeneous polynomial of degree $3$ or larger i.e.
\[ F(\boldsymbol{x}) = \sum_{\beta \in A} c_\beta \boldsymbol{x}^\beta \]
with $A\subset \{ \beta\in\mathbb{N}^n \ ;\ |\beta| = k\}$, $c_\beta \neq 0$ and $k\ge 3$. In this particular case, we can completely determine whether the Gibbs measure is normalizable by applying Theorem \ref{nornon}.
\begin{thm}
In the above setting, the following conditions are equivalent to each other.
\begin{enumerate}
\item
There holds
\[ \int e^{ \int_{\mathbb{T}^d} :F(\Phi):(x)dx} \mu_{\alpha}^{\otimes n}(d\Phi) < \infty. \]
\item
There holds
\[ \sup_{\boldsymbol{x}\in\mathbb{R}^n}\left( F(\boldsymbol{x}) + C\sum_{\beta\in A^-}|\boldsymbol{x}^\beta|  \right) < \infty \]
for any $C>0$.
\item
There holds
\[ F(\boldsymbol{x}) \lesssim -\sum_{\beta\in A} |\boldsymbol{x}^\beta|     \]
uniformly in $\boldsymbol{x}\in\mathbb{R}^n$.
\end{enumerate}
\end{thm}

\begin{proof}
First, we prove the assertion (i) $\implies$ (iii). We assume that (i) holds. Then, from Theorem \ref{nornon} (ii), there holds
\[ \sup_{x\in\mathbb{R}}\left( F(\boldsymbol{a}_{\boldsymbol{r}}(x)) + C\sum_{\beta\in A^-}|\boldsymbol{a}_{\boldsymbol{r}}(x)^\beta| -m|\boldsymbol{a}_{\boldsymbol{r}}(x)|^2  \right) < \infty    \]
for any $\boldsymbol{a}=(a_1,\cdots , a_n)\in \mathbb{R}^n$, $\boldsymbol{r}= (r_1, \cdots, r_n)\in \mathbb{N}^n$, $m>\frac{1}{2}$ and $C>0$. Especially, by setting $\boldsymbol{r}=(1,1,\cdots,1)$, we obtain
\begin{equation} \label{kama}
\sup_{x\in\mathbb{R}}\left( F(\boldsymbol{a})x^k + C\sum_{\beta\in A^-}|\boldsymbol{a}^\beta| |x|^{|\beta|} -m\sum_{i=1}^n a_i^2x^2  \right) < \infty    
\end{equation}
for any $\boldsymbol{a}\in \mathbb{R}^n$ and $C>0$. Therefore, we can see that $k$ is an even integer and there holds
\[ \sup_{\boldsymbol{a}\in \mathbb{R}^n}F(\boldsymbol{a})\le 0\]
because 
\begin{equation}\label{kaba}
\max_{\beta\in A^-}|\beta| = k-1 < k
\end{equation}
and $k\ge 3$. Moreover, if we assume $F(\boldsymbol{a}_0) = 0$ for some $\boldsymbol{a}_0\neq 0$, we get
\begin{align*}
&\sup_{x\in\mathbb{R}}\left( F(\boldsymbol{a}_0)x^k + C\sum_{\beta\in A^-}|\boldsymbol{a}_0^\beta||x|^{|\beta|} -m\sum_{i=1}^n a_{0,i}^2x^2  \right) \\
&= \sup_{x\in\mathbb{R}}\left( C\sum_{\beta\in A^-}|\boldsymbol{a}_0^\beta||x|^{|\beta|} -m\sum_{i=1}^n a_{0,i}^2x^2  \right) = +\infty
\end{align*}
for any $C \gg 1$ from \eqref{kaba} and this contradicts to \eqref{kama}. Therefore, $F(\boldsymbol{a})<0$ for any $\boldsymbol{a}\in\mathbb{R}^n\backslash \{ \boldsymbol{0}\}$. In particular, 
\begin{equation}\label{hana}
\sup_{|\boldsymbol{x}|=1} F(\boldsymbol{x}) \eqqcolon -\epsilon_1 <0.
\end{equation}
Now, we define $F^{\epsilon_1}(\boldsymbol{x}) \coloneqq F(\boldsymbol{x}) + \frac{\epsilon_1}{n\times \# A}\sum_{\beta \in A}|\boldsymbol{x}^\beta|$. Then, there holds $\sup_{\boldsymbol{x}\in\mathbb{R}^n}F^{\epsilon_1}(\boldsymbol{x})\le 0$ and (iii) holds. Indeed, if we assume $F^{\epsilon_1}(\boldsymbol{a}^1)> 0$ for some $\boldsymbol{a}^1\in\mathbb{R}^n$, there holds
\[ F^{\epsilon_1}(a_1^1x, a_2^1x,\cdots, a_n^1x) = x^k F(\boldsymbol{a}^1) + x^k \frac{\epsilon_1}{n\times \# A}\sum_{\beta\in A}|(\boldsymbol{a}^1)^\beta| = x^k F^{\epsilon_1}(\boldsymbol{a}^1) >0     \]
for any $x>0$ and this contradicts to \eqref{hana} because
\begin{align*}
\sup_{|\boldsymbol{x}|=1} F^{\epsilon_1}(\boldsymbol{x}) &= \sup_{|\boldsymbol{x}|=1} \left( F(\boldsymbol{x}) + \frac{\epsilon_1}{n\times \# A}\sum_{\beta \in A}|\boldsymbol{x}^\beta| \right)\\
&\le \sup_{|\boldsymbol{x}|=1} \left( F(\boldsymbol{x}) + \frac{\epsilon_1}{n\times \# A}\sum_{\beta \in A}\sum_{i=1}^n\frac{\beta_i}{k}x_i^k \right)\\
&\le \sup_{|\boldsymbol{x}|=1} \left( F(\boldsymbol{x})+\epsilon_1  \right),
\end{align*}
where we used Young's inequality and $\sum_{i}\beta_i = k$.
The assertion (ii) $\implies$ (iii) follows from a similar argument and (iii) $\implies$ (ii) is obvious from \eqref{kaba}. The assertion (iii) $\implies$ (i) can be seen from Theorem \ref{nornon} (i), \eqref{kaba} and Young's inequality. 

\end{proof}

\subsection{The critical case}\label{ex4}
There are cases that the Gibbs measure associated with polynomial $\lambda F$ with $\lambda >0$ is normalizable only when $\lambda$ is sufficiently small. Such phenomena can occur in the ``critical'' case, see \cite{phi3, heat} and references therein. In this section, we see one of such examples.

Let $n=2$, $\alpha = \frac{3m+1}{6m+3}d$ and $F(x,y)=x^m y^3 - Cx^{4m+2}$ with $m\in\mathbb{N}$ and $C>0$. Then, we can see that this example does not satisfy the assumption of Theorem \ref{nornon1} (i) because for any $C'>0$ and $\kappa < \frac{2d}{d-\alpha}=\frac{12m+6}{3m+2}$,
\begin{align*}
F(x^3,Kx^{3m+2}) - C'|(x^3,Kx^{3m+2})|^\kappa \gtrsim K^3 x^{12m + 6} - Cx^{12m+6} - C'_K|x|^{\kappa (3m+2)} \rightarrow +\infty
\end{align*}
as $x\rightarrow \infty$ if $K\gg 1$. However, it turns out that when $m=1$ and $C$ is sufficiently large, this example ``almost'' satisfies the assumption of Theorem \ref{nornon1} (i) and we can prove the following assertion, see Section \ref{criti} for the proof.
\begin{thm}\label{corocoro}
Let $n=2$, $\alpha = \frac{4}{9}d$ and $F(x,y)=x y^3 -Cx^{6}$ with $C \ge 100$. Then, there exists $\lambda_1 \ge \lambda_0 >0$ such that the followings hold:
\begin{enumerate}
\item
For any $0\le \lambda < \lambda_0$, 
\[ \int e^{ \lambda \int_{\mathbb{T}^d} :F(\Phi):(x)dx -K \left|\int_{\mathbb{T}^d}:  \Phi_1 ^2:(x) + :   \Phi_2 ^2:(x)dx \right|^{9}} \mu_{\alpha}^{\otimes 2}(d\Phi) < \infty  \]
for some $K >0$ which can be chosen independently of $\lambda$, where $\mu_{\alpha}^{\otimes 2}(d\Phi)= \mu_{\alpha}\otimes\mu_{\alpha}(d\Phi_1 d\Phi_2)$.
\item
For any $\lambda > \lambda_1$,
\[ \int e^{ \lambda \int_{\mathbb{T}^d} :F(\Phi):(x)dx -K \left|\int_{\mathbb{T}^d}:  \Phi_1 ^2:(x) + : \Phi_2 ^2:(x)dx \right|^b} \mu_{\alpha}^{\otimes 2}(d\Phi) = \infty  \]
for any $K>0$ and $b>0$.
\end{enumerate}
\end{thm}


\begin{rem}
It is known that a similar phenomenon occurs for $\Phi^3_3$ model as observed in \cite{phi3}. Our strategy for the proof of Theorem \ref{corocoro} is based on their work.
\end{rem}

\begin{rem}
The author expects that it might also possible to prove similar results for general $m\in\mathbb{N}$ although it is not sure at this point. 
\end{rem}

\section{Preliminary}\label{katana}

\subsection{Sobolev spaces}
We use the notations introduced in Section \ref{jai}. Sobolev spaces $W^{\alpha,p}(\mathbb{T}^d)$ for $\alpha\in\mathbb{R}$ and $1\le p\le \infty$ are defined as the spaces of all distributions $f\in\mathcal{D}'(\mathbb{T}^d)$ with
\[ \| f \|_{W^{\alpha, p}(\mathbb{T}^d)} \coloneqq \left\| \langle \nabla \rangle^{\alpha} f \right\|_{L^p (\mathbb{T}^d)} = \| \sum_{l \in \mathbb{Z}^d} \left( 1 + |l|^2 \right)^\frac{\alpha}{2} \hat{f}(l)e_l \|_{L^p(\mathbb{T}^d)} < \infty. \]
For $p=2$, we write $H^\alpha(\mathbb{T}^d)\coloneqq W^{\alpha,2}(\mathbb{T}^d)$. 
We recall some basic inequalities related to Sobolev spaces, see \cite{bcd,bara9}, for example.
\begin{lemm}{\textup{(duality)}}\label{duality}
For $f,g\in C^\infty(\mathbb{T}^d)$, $\alpha\in\mathbb{R}$, $p,q\in [1,\infty]$ with $\frac{1}{p}+ \frac{1}{q}=1$,
\[  \left|\int_{\mathbb{T}^d} fg dx \right| \le \| f \|_{W^{\alpha,p}(\mathbb{T}^d)} \| g\|_{W^{-\alpha,q}(\mathbb{T}^d)} .\]
\end{lemm}

\begin{lemm}{\textup{(Sobolev embedding)}}\label{embed}
Let $1\le p \le q <\infty$. Then, the embedding
\[ \| f\|_{L^q(\mathbb{T}^d)} \lesssim \| f \|_{W^{\frac{d}{p}-\frac{d}{q}, p}(\mathbb{T}^d)}    \]
holds. Especially,
\[  \| f\|_{L^q(\mathbb{T}^d)} \lesssim \| f \|_{H^{\frac{d}{2}-\frac{d}{q}}(\mathbb{T}^d)}   .\]
\end{lemm}

\begin{lemm}{\textup{(interpolation inequality, the Gagliardo-Nirenberg inequality)}}\label{inter}
For $1\le p, p_1, p_2 \le \infty$, $s, s_1, s_2\ge 0$, $\theta \in (0,1)$ with $s= \theta s_1 + (1-\theta)s_2$ with $\frac{1}{p} = \frac{\theta}{p_1} + \frac{1 -\theta}{p_2}$,
\[ \| f\|_{W^{s,p}(\mathbb{T}^d)} \lesssim \| f\|_{W^{s_1,p_1}(\mathbb{T}^d)}^\theta \| f\|_{W^{s_2,p_2}(\mathbb{T}^d)}^{1-\theta}     .\]
\end{lemm}
For the proof of the following lemma, see \cite[Lemma 3.4]{wave} and references therein.
\begin{lemm}{\textup{(fractional Leibniz rule, the Kato-Ponce inequality)}}\label{fractional}
For $s\ge 0$ and $1< p_i, q_i, r <\infty$ with $\frac{1}{p_i}+\frac{1}{q_i}=\frac{1}{r}$, $i=1,2$,
\[  \| \langle \nabla \rangle^s (fg) \|_{L^r(\mathbb{T}^d)} \lesssim \| f\|_{L^{p_1}(\mathbb{T}^d)}\| \langle \nabla \rangle^s g  \|_{L^{q_1}(\mathbb{T}^d)} + \| \langle \nabla \rangle^s f  \|_{L^{p_2}(\mathbb{T}^d)}\| g\|_{L^{q_2}(\mathbb{T}^d)} . \]
\end{lemm}

\subsection{Gaussian random variables}\label{pre2}
We prepare a few notations and lemmas related to Gaussian random variables which will be needed later.
First, we recall that the (1-variate) Hermite polynomials $H_l(x;\sigma)$ for $l\in\mathbb{N}, x\in \mathbb{R}$ and $\sigma\in\mathbb{R}$ are defined by 
\begin{equation}
e^{tx-\frac{1}{2}\sigma t^2} = \sum_{k=0}^\infty \frac{t^k}{k!} H_k(x;\sigma).
\end{equation}
The first Hermite polynomials are given by 
\[ H_0 (x;\sigma) = 1,\  H_1(x;\sigma) = x,\ H_2(x;\sigma) = x^2 -\sigma,\ H_3(x;\sigma) = x^3 -3\sigma x. \]
Then, we extend this definition to $H_\beta(\boldsymbol{x};\sigma)$ for $\beta\in\mathbb{N}^n, \boldsymbol{x}\in\mathbb{R}^n$ and $\sigma\in\mathbb{R}$ by
\begin{equation}
H_\beta (\boldsymbol{x};\sigma) \coloneqq \prod^n_{i=1} H_{\beta_i}(x_i;\sigma).
\end{equation}
We can also define $H_\beta(\boldsymbol{x};\sigma)$ by
\begin{equation}
e^{\boldsymbol{t}\cdot\boldsymbol{x}-\frac{1}{2}\sigma |\boldsymbol{t}|^2} = \sum_{\beta \in\mathbb{N}^n} \frac{\boldsymbol{t}^\beta}{\beta!} H_\beta(\boldsymbol{x};\sigma), \ \ \ \ \ \boldsymbol{t}\in\mathbb{R}^n
\end{equation}
and these definitions are equivalent. Then, for $\boldsymbol{x},\boldsymbol{y}\in\mathbb{R}^n$, there holds
\begin{align}
H_\alpha (\boldsymbol{x}+\boldsymbol{y};c)  = \sum_{0\leq \beta \leq \alpha}\binom{\alpha}{\beta}H_\beta(\boldsymbol{x};c) \boldsymbol{y}^{\alpha -\beta}. \label{hersei}
\end{align}

\begin{lemm}\label{gauss}
Let $k,m\in\mathbb{N}$ and $\xi_1 \sim N(0,\sigma_1^2)$ and $\xi_2 \sim N(0,\sigma_2^2)$ are jointly Gaussian random variables. Then, there holds
\[ \mathbb{E} [ H_k (\xi_1; \sigma_1^2) H_m (\xi_2; \sigma_2^2) ] = k! \delta_{k,m} \mathbb{E}[\xi_1 \xi_2]^k \]
where
\begin{gather}
\delta_{k,m} \coloneqq 
\begin{cases}
\displaystyle{1} \ \ \ \mathrm{when} \ k=m \\
\displaystyle{0} \ \ \ \mathrm{when} \ k\neq m \notag
\end{cases}
\end{gather}
denotes the Kronecker delta.
\end{lemm}
\begin{proof}
See \cite[Lemma 1.1.1]{nua}.
\end{proof}

\begin{lemm}\label{gausstoku}
Let $k,m\in\mathbb{N}$. Then, for mutually independent centered Gaussian random variables $\xi_1$ and $\xi_2$,
\[ \mathbb{E} [ H_k (\xi_1 + \xi_2; \sigma_{1,2}^2) \xi_1^m ] =  \mathbb{E} [ H_k (\xi_1; \sigma_{1}^2) \xi_1^m ]\]
where $\sigma_1^2 \coloneqq  \mathbb{E} [\xi_1^2]$ and $\sigma_{1,2}^2 \coloneqq  \mathbb{E} [(\xi_1 + \xi_2)^2]$.
\end{lemm}
\begin{proof}
The statement follows from \cite[Theorem 3.8]{jans}, for example.
\end{proof}

\subsection{The Bou\'e-Dupuis formula}\label{varia}

Let $W_i(t)$, $i=1,2, \cdots ,n$ be mutually independent cylindrical Wiener processes on $L^2(\mathbb{T}^d)$ expressed by
\begin{equation}
W_i (t) =  \sum_{l\in\mathbb{Z}^d} \beta_i^l (t) e_l
\end{equation}
where $(\beta_i^l)_{l\in\mathbb{Z}^d}$ are independent sequences of $\mathbb{C}$-valued standard Brownian motions conditioned with $\bar \beta_i^l = \beta_i^{-l}$ and $ \mbox{var} (\beta_i^l(t))=t$. 
We define $Y_i(t)\coloneqq Y_i^\alpha(t) \coloneqq \langle \nabla \rangle^{-\alpha}W_i(t)$ where $\langle \nabla \rangle \coloneqq (1-\Delta)^{\frac{1}{2}}$. Then, $(Y_1(1),Y_2(2),\cdots,Y_n(1))$ is distributed according to $\mu_{\alpha}^{\otimes n}$. Let $\mathcal{D}'_N(\mathbb{T}^d)\coloneqq \{ f\in\mathcal{D}'(\mathbb{T}^d)\ ;\ \hat f (l) = 0 \ \mbox{for any}\ l \in \mathbb{Z}^d\ \mbox{with}\ |l|>N\}$ and we identify as 
\[ \mathcal{D}'_N(\mathbb{T}^d) \cong \mathbb{C}^{\mathbb{Z}_N^d} \cong \mathbb{R}^{\mathbb{Z}_N^d}\times \mathbb{R}^{\mathbb{Z}_N^d}\]
by the Fourier expansion where $\mathbb{Z}_N^d \coloneqq \{ l \in \mathbb{Z}^d; |l|\leq N$\}.
The following lemma follows from \cite[Theorem 7]{as}, see also \cite[Lemma 5.11]{hartree}.
\begin{lemm}{\textup{(Bou\'e-Dupuis formula)}}\label{var}
Let $N\in \mathbb{N}$. Suppose that $F:\left( \mathbb{R}^{\mathbb{Z}_N^d}\times \mathbb{R}^{\mathbb{Z}_N^d} \right)^n \rightarrow \mathbb{R}$ is a measurable mapping such that
\[ \mathbb{E}\left[ \left| F( P_N Y_1 (1), P_N Y_2 (1),\cdots, P_N Y_n(1)) \right|^p \right] <\infty \]
and 
\[ \mathbb{E}\left[ \left| e^{-F( P_N Y_1 (1), P_N Y_2 (1),\cdots, P_N Y_n(1))} \right|^q \right] <\infty \]
for some $1<p, q<\infty$ with $\frac{1}{p}+\frac{1}{q}=1$. Then,
\begin{align*}
&-\log \mathbb{E} \left[ e^{-F( P_N Y_1 (1), P_N Y_2 (1),\cdots, P_N Y_n(1))} \right] \\
&= \inf_{\theta \in \mathbb{H}} \mathbb{E} \left[ F(( P_N Y_1 (1), P_N Y_2 (1),\cdots, P_N Y_n(1))+ P_N^{\otimes n}I(\theta)(1))+ \frac{1}{2}\int^1_0 \| P_N^{\otimes n} \theta (t)\|^2_{L^2(\mathbb{T}^d)}dt  \right]
\end{align*}
where
\[ I(\theta)(t) \coloneqq I_\alpha(\theta)(t) \coloneqq \int^t_0 \langle \nabla \rangle^{-\alpha} \theta (s)ds\] and $\mathbb{H}$ is the set of progressively measurable processes belonging to $L^2([0,1];(L^2(\mathbb{T}^d)^{\otimes n}))$ $\mathbb{P}$-almost surely.
\end{lemm}

\section{Normalizability 1}\label{norma1}
\subsection{Proof of Theorem \ref{norA}}
We use the notation introduced in Sections \ref{jai} and \ref{varia}.
Throughout this paper, we write
\begin{align*}
&\mathbb{Y}(t)= (Y_1(t),\cdots, Y_n(t)),\\
& I(\theta)(t)= (I_1(\theta)(t),\cdots, I_n(\theta)(t))
\end{align*}
and
\begin{align*}
&P_N^{\otimes n}\mathbb{Y}(t)= \mathbb{Y}_N(t)= (P_N Y_1(t),\cdots, P_N Y_n(t))=(Y_{1,N}(t),\cdots, Y_{n,N}(t)),\\
&P_N^{\otimes n} I(\theta)(t)= I_N(\theta)(t)= (P_N I_1(\theta)(t),\cdots, P_N I_n(\theta)(t)) = (I_{1,N}(\theta)(t),\cdots, I_{n,N}(\theta)(t)).
\end{align*}
Moreover, when $t=1$, we simplify some symbols and write
\begin{align*}
&\mathbb{Y}(1)= \mathbb{Y} = (Y_1,\cdots, Y_n),\\
& I(\theta)(1)= I(\theta)= (I_1(\theta),\cdots, I_n(\theta))
\end{align*}
and
\begin{align*}
&P_N^{\otimes n}\mathbb{Y}(1)= \mathbb{Y}_N =(Y_{1,N},\cdots, Y_{n,N}),\\
&P_N^{\otimes n} I(\theta)(1)= I_N(\theta)= (I_{1,N}(\theta),\cdots, I_{n,N}(\theta)).
\end{align*}
In the following arguments, the various functions $F(\mathbb{Y}_N)$ of $\mathbb{Y}_N$ satisfying
\[ \sup_{N\in\mathbb{N}}\mathbb{E} \left[ |F(\mathbb{Y}_N)|\right] <\infty\]
will appear such as $\|:\mathbb{Y}_N^\beta: \|^p_{W^{(\alpha -d/2)|\beta|-\epsilon,q}(\mathbb{T}^d)}$ with $\beta \in \mathbb{N}^n$, $\epsilon>0, p\ge 1$ and $1\le q \le \infty$, see Proposition \ref{gono} and \eqref{gonon}. For simplicity of notation, we denote by $Q_N$ these functions. So $Q_N$ can be different from line to line. 

Recall that $R_N^F$ is defined by \eqref{rndef}.
In view of the Bou\'e-Depuis formula (Lemma \ref{var}), it suffices to prove that 
\begin{align}
&\limsup_{N\rightarrow\infty} \inf_{\theta\in\mathbb{H}} \mathbb{E} \Bigg[ -R_N^F(\mathbb{Y}+I(\theta)) + \frac{1}{2}\int^1_0 \| P_N^{\otimes n} \theta (t)\|^2_{L^2(\mathbb{T}^d)}dt \Bigg] >  -\infty \label{moku}
\end{align}
to prove Theorem \ref{norA}, because from Fatou's lemma,
\begin{align*}
&\int e^{ \int_{\mathbb{T}^d} :F(\Phi):(x)dx} \mu_{\alpha}^{\otimes n}(d\Phi) = \int \lim_{N\rightarrow\infty} e^{ R_N^F(\Phi)} \mu_{\alpha}^{\otimes n}(d\Phi) \le \liminf_{N\rightarrow\infty} \int e^{ R_N^F(\Phi)} \mu_{\alpha}^{\otimes n}(d\Phi).
\end{align*}

\begin{proof}[Proof of Theorem \ref{norA}]
Let $\frac{k-1}{2k}d < \alpha \le \frac{d}{2}$. We assume that $F$ satisfies the assumption of Theorem \ref{norA} for some $m<\frac{1}{2}$.
From \eqref{rndef} and \eqref{hersei},
\begin{align}
-R_N^F(\mathbb{Y}+ I(\theta)) &= -\sum_{\beta \in A} c_\beta \int_{\mathbb{T}^d} H_\beta (P_N^{\otimes n}(\mathbb{Y}+ I(\theta));\sigma_{\alpha, N}) dx \notag\\
&= -\int_{\mathbb{T}^d} \left\{ F(P_N^{\otimes n}I(\theta)) -m |P_N^{\otimes n}I(\theta)|^2\right\}dx  \notag \\
&\ - \sum_{\beta \in  A} c_\beta \sum_{0<\gamma\leq \beta,\ \gamma\in\mathbb{N}^n} \binom{\beta}{\gamma}\int_{\mathbb{T}^d} :\mathbb{Y}_N^\gamma: I_N(\theta)^{\beta - \gamma} dx  -m  \int_{\mathbb{T}^d} |P_N^{\otimes n}I(\theta)|^2 dx \notag\\
&\eqqcolon A_N^1 + A_N^2 + A_N^3. \label{pr1}
\end{align}
From \eqref{lalala},
\begin{equation}\label{new1}
A_N^1 \ge \sum_{\beta \in A^-}\int_{\mathbb{T}^d} |I_N(\theta)^{\beta}|^{q(\beta)} dx - C 
\end{equation}
for some $C>0$ and $q: A^- \rightarrow \mathbb{R}_+$ with $q>p^{\alpha,F,q}$. Therefore, noting that there holds 
\begin{equation}\label{new3}
 \| I_N(\theta) \|^2_{H^{\alpha}(\mathbb{T}^d)} \le \int^1_0 \| P_N^{\otimes n}\theta (t)\|^2_{L^2(\mathbb{T}^d)}dt,
\end{equation}
we obtain 
\begin{align}
&A_N^1+A_N^3+  \frac{1}{2}\int^1_0 \| P_N^{\otimes n}\theta (t)\|^2_{L^2(\mathbb{T}^d)}dt \notag \\
&\ge \sum_{\beta\in A^-}\int_{\mathbb{T}^d}|I_N(\theta)^\beta |^{q(\beta)} dx +\frac{1}{2} \| I_N(\theta) \|^2_{{H}^{\alpha}(\mathbb{T}^d)} -m  \| I_N(\theta) \|^2_{L^{2}(\mathbb{T}^d)} - C \notag\\
&\gtrsim \sum_{\beta\in A^-}\int_{\mathbb{T}^d}|I_N(\theta)^\beta |^{q(\beta)} dx + \| I_N(\theta) \|^2_{H^{\alpha}(\mathbb{T}^d)}  - C . \label{pr2}
\end{align}
On the other hand, for each term of $A_N^2$, from Lemma \ref{duality} 
\begin{align}
\left|\int_{\mathbb{T}^d} :\mathbb{Y}_N^\gamma: I_N(\theta)^{\beta - \gamma} dx \right| \leq \| :\mathbb{Y}_N^\gamma:\|_{W^{(\alpha - \frac{d}{2})|\gamma|-\epsilon,\infty}(\mathbb{T}^d)} \| I_N(\theta)^{\beta -\gamma}\|_{W^{(\frac{d}{2}-\alpha)|\gamma| + 2\epsilon,1}(\mathbb{T}^d)}.\label{pr3}
\end{align}
And from the interpolation inequality (Lemma \ref{inter}), 
\begin{align} \| I_N(\theta)^{\beta -\gamma}\|_{W^{(\frac{d}{2}-\alpha)|\gamma| +2\epsilon,1}(\mathbb{T}^d)} \lesssim \| I_N(\theta)^{\beta -\gamma}\|_{L^{1}(\mathbb{T}^d)}^{1-(\frac{d}{2\alpha}-1)|\gamma|-\frac{2\epsilon}{\alpha}}\| I_N(\theta)^{\beta -\gamma}\|_{W^{\alpha,1}(\mathbb{T}^d)}^{(\frac{d}{2\alpha}-1)|\gamma|+ \frac{2\epsilon}{\alpha}}.\label{pr4}
\end{align}

Now, we take any $\left( (\epsilon^1,\tilde \epsilon^1), (\epsilon^2,\tilde \epsilon^2),\cdots ,(\epsilon^n,\tilde \epsilon^n) \right) \in \prod^n_{i=1}R^i(\beta -\gamma)$ where $R^i(\beta -\gamma)$ is the one defined just before Theorem \ref{norA}. Then, from the fractional Leibniz rule (Lemma \ref{fractional}) and the Sobolev embedding (Lemma \ref{embed}),
\begin{align}
&\| I_N (\theta) ^{\beta -\gamma} \|_{W^{\alpha,1}(\mathbb{T}^d)} = \| \prod^{n}_{i=1} I_{i,N}(\theta)^{\beta_i - \gamma_i} \|_{W^{\alpha,1}(\mathbb{T}^d)} \notag \\
&\lesssim \sum_{1\le i \le n\ \beta_i -\gamma_i\neq 0} \| I_{i,N}(\theta)\|_{H^\alpha(\mathbb{T}^d)}\left( \|I_{i,N}(\theta) \|^{\epsilon^i_i}_{L^{\frac{2d}{d-2\alpha}}(\mathbb{T}^d)}  \|I_{i,N}(\theta) \|^{\tilde \epsilon^i_i}_{L^{\kappa_i(q)}(\mathbb{T}^d)} \right) \notag \\
&\ \ \ \ \ \ \ \ \ \ \ \ \ \ \ \ \ \ \ \ \ \ \ \ \ \ \ \ \ \ \ \ \ \ \ \ \  \times\prod_{1\le j \le n,\ j\neq i}\left( \|I_{j,N}(\theta) \|^{\epsilon^i_j}_{L^{\frac{2d}{d-2\alpha}}(\mathbb{T}^d)}  \|I_{j,N}(\theta) \|^{\tilde \epsilon^i_j}_{L^{ \kappa_j(q)}(\mathbb{T}^d)} \right)\notag\\
&\lesssim \sum_{1\le i \le n\ \beta_i-\gamma_i\neq 0} \| I_{i,N}(\theta)\|_{H^\alpha(\mathbb{T}^d)}\left( \prod_{j=1}^n \| I_{j,N}(\theta)\|_{H^\alpha(\mathbb{T}^d)}^{\epsilon^i_j}  \right) \left( \prod_{j=1}^n \| I_{j,N}(\theta)\|_{L^{\kappa_j(q)}(\mathbb{T}^d)}^{\tilde \epsilon^i_j}  \right) \label{jujutu}
\end{align}
where we used 
\begin{align}
\epsilon_j ^i+ \tilde \epsilon_j^i =
\left\{ \begin{array}{l}
\beta_j -\gamma_j\ \ \ \ \ \  \mbox{for}\ j \neq i \\
\beta_j - \gamma_j -1 \  \mbox{for}\ j = i 
\end{array}\right.\ \mbox{if}\ \beta_j -\gamma_j \neq 0
\end{align}
and
\begin{align}
\frac{d-2\alpha}{2d}\sum^n_{j=1}\epsilon_j^i + \sum_{j=1}^n \frac{\tilde \epsilon_j^i}{\kappa_j(q)} + \frac{1}{2} \le 1, 
\end{align}
which are guaranteed by the definition of $R^i(\beta -\gamma)$, to apply the fractional Leibniz rule. 
Therefore, from \eqref{pr3}, \eqref{pr4}, \eqref{jujutu} and Young's inequality,
\begin{align}
&\left|\int_{\mathbb{T}^d} :\mathbb{Y}_N^\gamma: I_N(\theta)^{\beta - \gamma} dx \right| \notag \\
&\leq \epsilon ' \| I_N(\theta)^{\beta -\gamma}\|_{L^{q(\beta -\gamma)}(\mathbb{T}^d)}^{q(\beta -\gamma)} + \epsilon ' \sum_{i=1}^n \| I_{i,N}(\theta)\|^2_{H^\alpha(\mathbb{T}^d)} + \epsilon ' \sum_{i=1}^n \| I_{i,N}\|_{L^{ \kappa_i(q)}(\mathbb{T}^d)}^{ \kappa_i(q)} + C_{\epsilon '}Q_N\label{sime}
\end{align}
for any $\epsilon '>0$ if
\begin{align}
\left\{ 1 - \left(\frac{d}{2\alpha} - 1\right)|\gamma| - \frac{2\epsilon}{\alpha} \right\}\frac{1}{q(\beta-\gamma)} + \left\{  \left(\frac{d}{2\alpha} - 1\right)|\gamma| + \frac{2\epsilon}{\alpha} \right\}\left\{ \frac{1+\sum^n_{j=1}\epsilon_j^i}{2} + \sum_{j=1}^n \frac{\tilde \epsilon_j^i}{ \kappa_j(q)} \right\} <1 \label{gori}
\end{align} 
for any $1\le i \le n$, for sufficiently small $\epsilon>0$. Actually, we can choose $\left( (\epsilon^1,\tilde \epsilon^1), (\epsilon^2,\tilde \epsilon^2),\cdots ,(\epsilon^n,\tilde \epsilon^n) \right) \in \prod^n_{i=1}R^i(\beta -\gamma)$ such that \eqref{gori} holds because  $q(\beta - \gamma) > p^{\alpha,F,q} (\beta-\gamma)$.
Combining \eqref{pr1}, \eqref{pr2} and \eqref{sime}, we get from Lemma \ref{tuitui} below that
\begin{align*}
&\mathbb{E} \left[ -R_N^F(\mathbb{Y}+I(\theta)) + \frac{1}{2}\int^1_0 \| P_N^{\otimes n} \theta (t)\|^2_{L^2(\mathbb{T}^d)}dt \right]\\
&\gtrsim \mathbb{E} \left[ \sum_{\beta \in A^-} \int_{\mathbb{T}^d} |I_N(\theta)^\beta|^{q(\beta)} dx  +   \| I_N(\theta) \|^2_{H^{\alpha}(\mathbb{T}^d)} - \epsilon ' \sum_{i=1}^n \int_{\mathbb{T}^d} |I_{i,N}(\theta)|^{ \kappa_i(q)}dx \right]  - C \notag \\
&\gtrsim \mathbb{E} \left[ \sum_{\beta \in A^-} \int_{\mathbb{T}^d} |I_N(\theta)^\beta|^{q(\beta)} dx  +   \| I_N(\theta) \|^2_{H^{\alpha}(\mathbb{T}^d)}    \right] - C
\end{align*}
by choosing sufficiently small $\epsilon ' >0$. Then, \eqref{moku} follows immediately from this.

\end{proof}

\begin{lemm}\label{tuitui}
Let $q:A^-\rightarrow \mathbb{R}_+$. Then,
\[ \sum_{i=1}^n \int_{\mathbb{T}^d} |I_{i,N}(\theta)|^{ \kappa_i(q)}dx \lesssim \sum_{\beta \in A^-} \int_{\mathbb{T}^d} |I_N(\theta)^\beta|^{q(\beta)} dx \]
where $\kappa_i(q)$ is defined by \eqref{kawadou}. 
\end{lemm}

\begin{proof}
Fix $1\le i\le n$. From \eqref{kawadou}, $\kappa_i(q)e_i\in\mbox{Conv}\left( \bigcup_{\beta\in A^-} q(\beta)\beta  \right)$. Thus, there exist some $\{ a_l \}_{l=1}^M\subset \mathbb{R}_+$ and $\{ \beta^l \}_{l=1}^M\subset A^-$ such that 
\begin{equation}\label{kl}
\sum_{l=1}^Ma_l = 1
\end{equation}
and 
\begin{equation}\label{kll}
\kappa_i(q)e_i = \sum_{l=1}^M a_l q(\beta_l)\beta_l.
\end{equation}
From $\kappa_i(q)\ge 0$, $\beta^l\in\mathbb{R}_+^n$ and $q\ge 0$, it is easy to see that there exists $\{ b_l\}_{l=1}^M\subset\mathbb{R}_+$ such that
$\beta_l = b_le_i$ for any $1\le l\le M$. Therefore, from \eqref{kl} and \eqref{kll},
\[ \kappa_i(q) \le \max_{1\le l\le M}q(\beta^l)b_l \eqqcolon q(\beta^L) b_L  .\]
After all, we get
\begin{align*}
\int_{\mathbb{T}^d} |I_{i,N}(\theta)|^{ \kappa_i(q)}dx \lesssim \int_{\mathbb{T}^d} |I_{i,N}(\theta)|^{b_Lq(\beta^L)} dx + 1 &= \int_{\mathbb{T}^d} |I_N(\theta)^{\beta_L}|^{q(\beta^L)} dx + 1 \\
&\lesssim \sum_{\beta \in A^-} \int_{\mathbb{T}^d} |I_N(\theta)^\beta|^{q(\beta)} dx.
\end{align*}

\end{proof}

\section{Non-normalizability 1}\label{norma2}

In this section, we prove the non-normalizability parts of the main theorems in Section \ref{jai2}:  Theorems \ref{nornon} (ii) and \ref{nonA}. We use the notations introduced in Sections \ref{jai} and \ref{varia}.

\subsection{Proof of Theorem \ref{nornon} (ii)}
From the almost sure convergence $R_N^F \rightarrow R^F$ as mentioned in Section \ref{jai}, the dominated convergence theorem and the monotone convergence theorem,
\begin{equation}
\int e^{ \int_{\mathbb{T}^d} :F(\Phi):(x)dx} \mu_{\alpha}^{\otimes n}(d\Phi) = \lim_{L\rightarrow \infty} \lim_{N\rightarrow\infty} \int e^{R_N^F(\Phi)\wedge L} \mu_{\alpha}^{\otimes n}(d\Phi).
\end{equation}
Therefore, in view of Lemma \ref{var}, it suffices to show
\begin{equation} \label{kore}
\limsup_{L\rightarrow\infty}\limsup_{N\rightarrow\infty} \inf_{\theta\in\mathbb{H}} \mathbb{E} \left[ -\left( R_N^F(\mathbb{Y}+I(\theta))\wedge L \right) + \frac{1}{2}\int^1_0 \| P_N^{\otimes n} \theta (t)\|^2_{L^2(\mathbb{T}^d)}dt \right] = -\infty
\end{equation}
to prove Theorems \ref{nornon} (ii) and \ref{nonA}. Although the most part of the proof is basically the same argument, we will prove Theorems \ref{nornon} (ii) and \ref{nonA} separately due to the difference in the details of the proofs. First, we prove Theorem \ref{nornon} (ii).
Our proof is based on the non-normalizability argument used in the works \cite{hartree, phi3, heat}: to prove \eqref{kore}, we consider the sequence $\{\theta^M\}_{M\in\mathbb{N}}\subset\mathbb{H}$ that looks like ``$- \langle \nabla \rangle^{\alpha}\mathbb{Y}_M(1) + \langle \nabla \rangle^{\alpha}f_M$''  where $f_M$ is a suitable deterministic function or constant. The major difference between our argument and the ones in those previous works is the appearance of the third term in \eqref{tokutoku} and the last part of the proof of Theorem \ref{nornon} \textup{(ii)}.

\begin{proof}[Proof of Theorem \ref{nornon} \textup{(ii)}]
Let $\alpha = \frac{d}{2}$. We assume that there exist some $\boldsymbol{a}=(a_1,\cdots,a_n)\in\mathbb{R}^n$, $\boldsymbol{r}= (r_1,\cdots,r_n)\in\mathbb{N}^n$, $m>\frac{1}{2}$ and $C>0$ such that 
\begin{align}
\lim_{x\rightarrow\infty} \left\{  F (\boldsymbol{a}_{\boldsymbol{r}}(x)) + C\sum_{\beta\in A^-}|\boldsymbol{a}_{\boldsymbol{r}}(x)^\beta| -m|\boldsymbol{a}_{\boldsymbol{r}}(x)|^2 \right\} = +\infty,
\end{align}
where $\boldsymbol{a}_{\boldsymbol{r}}(x) \coloneqq (a_1 x^{r_1}, \cdots, a_n x^{r_n})$.
We set $\kappa_1$ and $\kappa_2$ by 
\begin{equation}\label{hq}
F (\boldsymbol{a}_{\boldsymbol{r}}(x)) + C\sum_{\beta\in A^-}|\boldsymbol{a}_{\boldsymbol{r}}(x)^\beta| -m|\boldsymbol{a}_{\boldsymbol{r}}(x)|^2 \asymp |x|^{\kappa_1} \ \ \ \mbox{as}\ |x|\rightarrow \infty
\end{equation}
and
\begin{equation}\label{hq2}
\kappa_2 \coloneqq \max_{\beta\in A^-,\ \boldsymbol{a}^\beta \neq 0}  \beta \cdot \boldsymbol{r}.
\end{equation}
Given $M\in\mathbb{N}$ and $\delta>0$, we define $\theta^M=(\theta^M_1,\cdots,\theta^M_n)\in\mathbb{H}$ by
\begin{equation}\label{tokutoku} 
\theta^M (t) \coloneqq - \langle \nabla \rangle^{\frac{d}{2}}\boldsymbol{Z}_M^\eta (t) + \boldsymbol{a}_{\boldsymbol{r}}(M^{\frac{d+\delta}{\kappa_1}}) + \sigma_{\alpha,M}^{\frac{1}{2}}\boldsymbol{b} 
\end{equation}
where $\boldsymbol{Z}_M^\eta (t) \coloneqq (Z_{1,M}^\eta (t),\cdots,Z_{n,M}^\eta (t))$ is defined by 
\begin{gather}\label{popo}
\boldsymbol{Z}_M^\eta (t) \coloneqq
\begin{cases}
0 \ \ \ \ \ \ \ \ \ \ \ \ \ \ \ \ \   \mbox{for}\ 0\le t < \eta\\
\frac{1}{1-\eta}\mathbb{Y}_{M}(\eta)\ \ \ \ \mbox{for}\ \eta \le t < 1  
\end{cases}
\end{gather}
and the value of $0<\eta < 1$ and $\boldsymbol{b}=(b_1,\cdots,b_n)\in\mathbb{R}^n$ will be determined later. Note that
\[( I(\theta^M)(1)=)\ I (\theta^M)   = -\mathbb{Y}_{M}(\eta) + \boldsymbol{a}_{\boldsymbol{r}}(M^{\frac{d+\delta}{\kappa_1}}) + \sigma_{\alpha,M}^{\frac{1}{2}}\boldsymbol{b}.\]
For simplicity, we write $\boldsymbol{a}_M\coloneqq \boldsymbol{a}_{\boldsymbol{r}}(M^{\frac{d+\delta}{\kappa_1}}) = (a_1M^{\frac{d+\delta}{\kappa_1}r_1},\cdots,a_n M^{\frac{d+\delta}{\kappa_1}r_n})$.
Then, for $N>M$, from \eqref{hersei} and \eqref{rndef},
\begin{align}
-R_N^F(\mathbb{Y}+I(\theta^M)) &= -R_N^F\left(\mathbb{Y}-\mathbb{Y}_M (\eta)+\boldsymbol{a}_M + \sigma_{\alpha, M}^{\frac{1}{2}}\boldsymbol{b}\right)\notag \\
&= -\int_{\mathbb{T}^d} \left\{ F (\boldsymbol{a}_M) + C\sum_{\beta\in A^-}|\boldsymbol{a}_M|^\beta -m|\boldsymbol{a}_M|^2 \right\}dx + \int_{\mathbb{T}^d} C\sum_{\beta\in A^-}|\boldsymbol{a}_M|^\beta dx \notag\\
&\ -\sum_{\beta\in A} c_\beta \sum_{0<\gamma\le\beta,\gamma\in\mathbb{N}^n} \binom{\beta}{\gamma} \int_{\mathbb{T}^d} H_\gamma (\mathbb{Y}_N - \mathbb{Y}_M(\eta) + \sigma_{\alpha,M}^{\frac{1}{2}}\boldsymbol{b};\sigma_{\alpha,N}) \boldsymbol{a}_M^{\beta-\gamma}dx \notag \\
&\ -m  \int_{\mathbb{T}^d} |\boldsymbol{a}_M|^2 dx. \label{messi}
\end{align}
For the third term, from \eqref{hersei} and Proposition \ref{gono},
\begin{align}
&-\sum_{\beta\in  A}  c_\beta \sum_{0<\gamma\le\beta,\ \gamma\in\mathbb{N}^n} \binom{\beta}{\gamma} \int_{\mathbb{T}^d} H_\gamma (\mathbb{Y}_N - \mathbb{Y}_M(\eta) + \sigma_{\alpha,M}^{\frac{1}{2}}\boldsymbol{b};\sigma_{\alpha,N}) \boldsymbol{a}_M^{\beta-\gamma}dx \notag \\
&= -\sum_{\beta\in  A}  c_\beta \sum_{0<\gamma\le\beta,\ \gamma\in\mathbb{N}^n} \binom{\beta}{\gamma} \sum_{0\le \tilde \gamma \le \gamma, \tilde \gamma \in \mathbb{N}^n} \binom{\gamma}{\tilde \gamma} \int_{\mathbb{T}^d} :\mathbb{Y}_N^{\tilde\gamma}: (- \mathbb{Y}_M(\eta) + \sigma_{\alpha,M}^{\frac{1}{2}}\boldsymbol{b})^{\gamma - \tilde \gamma} \boldsymbol{a}_M^{\beta-\gamma}dx \notag \\
&\rightarrow - \sum_{\beta\in  A}  c_\beta \sum_{0<\gamma\le\beta,\ \gamma\in\mathbb{N}^n} \binom{\beta}{\gamma} \sum_{0\le \tilde \gamma \le \gamma, \tilde \gamma \in \mathbb{N}^n} \binom{\gamma}{\tilde \gamma} \int_{\mathbb{T}^d} :\mathbb{Y}^{\tilde\gamma}: (- \mathbb{Y}_M(\eta) + \sigma_{\alpha,M}^{\frac{1}{2}}\boldsymbol{b})^{\gamma - \tilde \gamma} \boldsymbol{a}_M^{\beta-\gamma}dx \label{messi2}
\end{align}
as $N\rightarrow\infty$ almost surely. 
On the other hand, 
\begin{align}
&\mathbb{E} \left[ \frac{1}{2} \int^1_0 \| P_N^{\otimes n} \theta^M (t) \|^2_{L^2(\mathbb{T}^d)} dt    \right]\notag\\
&= \sum_{i=1}^n \frac{1}{2} \int_0^1 \int_{\mathbb{T}^d} \mathbb{E} [ (\langle \nabla \rangle ^{\frac{d}{2}} Z^\eta_{i,M}(t))^2 + (a_iM^{\frac{d+\delta}{\kappa_1}r_i} + b_i \sigma_{\alpha,M}^{\frac{1}{2}})^2 ] dxdt\notag \\ 
&= \frac{1}{2} \mathbb{E} \left[  \int_{\mathbb{T}^d} \left\{|\boldsymbol{a}_M|^2 +  \sigma_{\alpha, M}^{\frac{1}{2}}\boldsymbol{a}_M \cdot\boldsymbol{b} + \sigma_{\alpha, M}|\boldsymbol{b}|^2 +  \int^1_0 \left|\langle \nabla \rangle^{\frac{d}{2}}\boldsymbol{Z}_{M}^\eta(t)\right|^2 dt     \right\} dx\right]         \label{messi4}
\end{align}
Therefore, from \eqref{messi}, \eqref{messi2} and \eqref{messi4},
\begin{align}
&\lim_{L\rightarrow\infty}\lim_{N\rightarrow\infty}\mathbb{E} \left[ -\left( R_N^F(\mathbb{Y}+I(\theta^M))\wedge L \right) + \frac{1}{2}\int^1_0 \| P_N^{\otimes n} \theta^M (t)\|^2_{L^2(\mathbb{T}^d)}dt  \right]\notag\\
&= \lim_{L\rightarrow\infty}\mathbb{E} \left[ -\left(\left\{\lim_{N\rightarrow\infty} R_N^F(\mathbb{Y}+I(\theta^M))\right\}\wedge L \right)\right] + \lim_{N\rightarrow\infty}\mathbb{E} \left[ \frac{1}{2}\int^1_0 \| P_N^{\otimes n} \theta^M (t)\|^2_{L^2(\mathbb{T}^d)}dt  \right]\notag\\
&= -\int_{\mathbb{T}^d} \left\{ F (\boldsymbol{a}_M) + C\sum_{\beta\in A^-}|\boldsymbol{a}_M|^\beta -m|\boldsymbol{a}_M|^2 \right\}dx + \int_{\mathbb{T}^d} C\sum_{\beta\in A^-}|\boldsymbol{a}_M|^\beta dx          \notag\\
&\ - \mathbb{E} \left[ \sum_{\beta\in  A}  c_\beta \sum_{0<\gamma\le\beta,\ \gamma\in\mathbb{N}^n} \binom{\beta}{\gamma} \sum_{0\le \tilde \gamma \le \gamma, \tilde \gamma \in \mathbb{N}^n} \binom{\gamma}{\tilde \gamma} \int_{\mathbb{T}^d} :\mathbb{Y}^{\tilde\gamma}: (- \mathbb{Y}_M(\eta) + \sigma_{\alpha,M}^{\frac{1}{2}}\boldsymbol{b})^{\gamma - \tilde \gamma} \boldsymbol{a}_M^{\beta-\gamma}dx  \right]\notag \\
&\ +\frac{1}{2} \mathbb{E} \left[  \int_{\mathbb{T}^d} \left\{ \sigma_{\alpha, M}^{\frac{1}{2}}\boldsymbol{a}_M \cdot\boldsymbol{b} + \sigma_{\alpha, M}|\boldsymbol{b}|^2 +  \int^1_0 \left|\langle \nabla \rangle^{\frac{d}{2}}\boldsymbol{Z}_{M}^\eta(t)\right|^2 dt     \right\} dx\right] \notag \\
&\ + \left(\frac{1}{2} - m\right)\int_{\mathbb{T}^d} |\boldsymbol{a}_M|^2 dx.
\end{align}
If we can choose $0<\eta<1$ and $\boldsymbol{b}\in\mathbb{R}^n$ such that
\begin{align}\label{pistol}
&- \mathbb{E} \left[ \sum_{\beta\in  A}  c_\beta \sum_{0<\gamma\le\beta,\ \gamma\in\mathbb{N}^n} \binom{\beta}{\gamma} \sum_{0\le \tilde \gamma \le \gamma, \tilde \gamma \in \mathbb{N}^n} \binom{\gamma}{\tilde \gamma} \int_{\mathbb{T}^d} :\mathbb{Y}^{\tilde\gamma}: (- \mathbb{Y}_M(\eta) + \sigma_{\alpha,M}^{\frac{1}{2}}\boldsymbol{b})^{\gamma - \tilde \gamma} \boldsymbol{a}_{M}^{\beta-\gamma}dx  \right] \notag \\
&\lesssim -(\log M)^h M^{\frac{\kappa_2}{\kappa_1}(d+\delta)} + C
\end{align}
for some $h>0$, then, from \eqref{hq}, \eqref{hq2} and Lemma \ref{hq3} below,
\begin{align*}
&\lim_{L\rightarrow\infty}\lim_{N\rightarrow\infty} \mathbb{E} \left[  -\left( R_N^F(\mathbb{Y}+I(\theta^{M}))\wedge L \right) + \frac{1}{2}\int^1_0 \| P_N^{\otimes n} \theta^{M} (t)\|^2_{L^2(\mathbb{T}^d)}dt \right]\\
&\lesssim -cM^{d+\delta} + M^{\frac{\kappa_2}{\kappa_1}(d+\delta)} - c(\log M)^h M^{\frac{\kappa_2}{\kappa_1}(d+\delta)} + M^d + (\log M)^{\frac{1}{2}} M^{\frac{r^{\max}}{\kappa_1}(d+\delta)} -cM^{\frac{2r^{\max}}{\kappa_1}(d+\delta)} +  1  \\
&\rightarrow -\infty
\end{align*}
as $M\rightarrow \infty$ where $r^{\max} \coloneqq \max_{1\le i\le n,\ a_i\neq 0} r_i$, and this proves Theorem \ref{nornon} (ii). In the following, we prove the existence of $0<\eta<1$ and $\boldsymbol{b}\in\mathbb{R}^n$ which satisfy \eqref{pistol}. From Lemma \ref{gausstoku} and the mutually independence of $\mathbb{Y}_N-\mathbb{Y}_M$ and $\mathbb{Y}_M$ for $N>M$,
\begin{align}
&- \mathbb{E} \left[ \sum_{\beta\in  A}  c_\beta \sum_{0<\gamma\le\beta,\ \gamma\in\mathbb{N}^n} \binom{\beta}{\gamma} \sum_{0\le \tilde \gamma \le \gamma, \tilde \gamma \in \mathbb{N}^n} \binom{\gamma}{\tilde \gamma} \int_{\mathbb{T}^d} :\mathbb{Y}^{\tilde\gamma}: (- \mathbb{Y}_M(\eta) + \sigma_{\alpha,M}^{\frac{1}{2}}\boldsymbol{b})^{\gamma - \tilde \gamma} \boldsymbol{a}_M^{\beta-\gamma}dx  \right] \notag \\
&=- \lim_{N\rightarrow \infty}\mathbb{E} \left[ \sum_{\beta\in  A}  c_\beta \sum_{0<\gamma\le\beta,\ \gamma\in\mathbb{N}^n} \binom{\beta}{\gamma} \sum_{0\le \tilde \gamma \le \gamma, \tilde \gamma \in \mathbb{N}^n} \binom{\gamma}{\tilde \gamma} \int_{\mathbb{T}^d} :\mathbb{Y}_N^{\tilde\gamma}: (- \mathbb{Y}_M(\eta) + \sigma_{\alpha,M}^{\frac{1}{2}}\boldsymbol{b})^{\gamma - \tilde \gamma} \boldsymbol{a}_M^{\beta-\gamma}dx  \right] \notag \\ 
&= - \mathbb{E} \left[ \sum_{\beta\in  A}  c_\beta \sum_{0<\gamma\le\beta,\ \gamma\in\mathbb{N}^n} \binom{\beta}{\gamma} \sum_{0\le \tilde \gamma \le \gamma, \tilde \gamma \in \mathbb{N}^n} \binom{\gamma}{\tilde \gamma} \int_{\mathbb{T}^d} :\mathbb{Y}_M^{\tilde\gamma}: (- \mathbb{Y}_M(\eta) + \sigma_{\alpha,M}^{\frac{1}{2}}\boldsymbol{b})^{\gamma - \tilde \gamma} \boldsymbol{a}_M^{\beta-\gamma}dx  \right] \notag \\
&= - \mathbb{E} \left[ \sum_{\beta\in  A}  c_\beta \sum_{0<\gamma\le\beta,\ \gamma\in\mathbb{N}^n} \binom{\beta}{\gamma} \int_{\mathbb{T}^d} H_\gamma \left(  \mathbb{Y}_M -\mathbb{Y}_M(\eta) + \sigma_{\alpha,M}^{\frac{1}{2}}\boldsymbol{b}\ ;\ \sigma_{\alpha,M}  \right)   \boldsymbol{a}_M^{\beta-\gamma}dx  \right], \label{hoshu2}
\end{align}
where we used \eqref{hersei} to go from the third line to the fourth line.
Let 
\begin{equation}
\kappa_3 \coloneqq  \max \left\{|\gamma| \ ;\ \beta\in A,\ \gamma\in\mathbb{N}^n,\ 0< \gamma\le\beta,\ (\beta -\gamma)\cdot \boldsymbol{r} = \kappa_2,\ \boldsymbol{a}^{\beta-\gamma}\neq 0      \right\} 
\end{equation}
and
\begin{equation}
S \coloneqq   \left\{ (\beta,\gamma)\in  A \times \mathbb{N}^n \  ; \ 0< \beta\le \beta,\ (\beta-\gamma)\cdot \boldsymbol{r} = \kappa_2,\ \boldsymbol{a}^{\beta-\gamma}\neq 0,\ |\gamma|=\kappa_3   \right\}. \end{equation}
Then, noting that for $(\beta,\gamma )\in S$ and $\sigma_{\alpha,M} \asymp \log M$ from \eqref{dyn}, there holds
\begin{equation}\label{kau}
L(M)\coloneqq \sigma_{\alpha,M}^{\frac{|\gamma|}{2}} M^{\frac{d+\delta}{\kappa_1}((\beta-\gamma)\cdot \boldsymbol{r})} = \sigma_{\alpha,M}^{\frac{\kappa_3}{2}} M^{\frac{\kappa_2}{\kappa_1}(d+\delta )}  \asymp   (\log M)^{\frac{\kappa_3}{2}}M^{\frac{\kappa_2}{\kappa_1}(d+\delta)},
\end{equation}
we can see that 
\begin{align}
p_M (\eta,\boldsymbol{b}) &\coloneqq - \mathbb{E} \left[ \sum_{(\beta,\gamma )\in S}  c_\beta \binom{\beta}{\gamma} \int_{\mathbb{T}^d} H_\gamma \left(  \mathbb{Y}_M -\mathbb{Y}_M(\eta) + \sigma_{\alpha,M}^{\frac{1}{2}}\boldsymbol{b} ; \sigma_{\alpha,M}  \right)   \boldsymbol{a}_M^{\beta-\gamma}dx  \right] \notag\\
&= -\mathbb{E} \left[ \sum_{(\beta,\gamma )\in S}  c_\beta \binom{\beta}{\gamma} \int_{\mathbb{T}^d} H_\gamma \left(  \frac{\mathbb{Y}_M -\mathbb{Y}_M(\eta)}{\sigma_{\alpha,M}^{\frac{1}{2}}} + \boldsymbol{b} ; 1  \right) \sigma_{\alpha,M}^{\frac{|\gamma|}{2}} M^{\frac{(d+\delta)}{\kappa_1}((\beta-\gamma) \cdot \boldsymbol{r})}  \boldsymbol{a}^{\beta-\gamma}dx  \right]\notag \\
&= -L(M)\mathbb{E} \left[ \sum_{(\beta,\gamma )\in S}  c_\beta \binom{\beta}{\gamma} \int_{\mathbb{T}^d} H_\gamma \left(  \frac{\mathbb{Y}_M -\mathbb{Y}_M(\eta)}{\sigma_{\alpha,M}^{\frac{1}{2}}} + \boldsymbol{b} ; 1  \right) \boldsymbol{a}^{\beta-\gamma}dx  \right] \label{hoshu}
\end{align}
is a polynomial of $\eta$ and $\boldsymbol{b}$.
Then, $\tilde p (\eta,\boldsymbol{b})\coloneqq L(M)^{-1} p_M(\eta,\boldsymbol{b})$ does not depend on $M\in\mathbb{N}$ and 
\begin{align*}
\tilde p (1,\boldsymbol{b}) &\simeq -\sum_{(\beta,\gamma )\in S}  c_\beta \binom{\beta}{\gamma}  H_\gamma \left( \boldsymbol{b} ; 1  \right) \boldsymbol{a}^{\beta-\gamma}.
\end{align*}
 Therefore, by choosing $\boldsymbol{b}$ such that 
\[ -\sum_{(\beta,\gamma )\in S} \tilde c_\beta \binom{\beta}{\gamma}  H_\gamma \left( \boldsymbol{b} ; 1  \right) \boldsymbol{a}^{\beta-\gamma} < 0\]
and $0<\eta<1$ sufficiently close to $1$, we get from \eqref{kau} that
\begin{equation}\label{kaka}
p_{M} (\eta,\boldsymbol{b}) = L(M) \tilde p(\eta,\boldsymbol{b}) \asymp  -(\log M)^{\frac{\kappa_3}{2}}M^{\frac{\kappa_2}{\kappa_1}(d+\delta)}. 
\end{equation} 
Note that such $\boldsymbol{b}\in\mathbb{R}^n$ indeed exists because
\[ \int_{\mathbb{R}^d} \sum_{(\beta,\gamma )\in S}  c_\beta \binom{\beta}{\gamma}  H_\gamma \left( \boldsymbol{x} ; 1  \right) \boldsymbol{a}^{\beta-\gamma} e^{-\frac{|\boldsymbol{x}|^2}{2}} d\boldsymbol{x} = 0 \]
from the ($L^2(\mathbb{R}^d;e^{-|x|^2/ 2} dx)$-)orthogonality of the Hermite polynomials and 
\[ \sum_{(\beta,\gamma )\in S} c_\beta \binom{\beta}{\gamma}  H_\gamma \left( \boldsymbol{x} ; 1  \right) \boldsymbol{a}^{\beta-\gamma} \not\equiv 0.\]
On the other hand, letting
\[ \tilde S \coloneqq \{ (\beta,\gamma)\in A \times \mathbb{N}^n\ ;\ 0< \gamma\le \beta \}\  \backslash \  S,\]
we can see from the definition of $S$ and $\kappa_2$, and $\sigma_{\alpha,M} \asymp \log M$ from \eqref{dyn} that
\begin{align}
&- \mathbb{E} \left[ \sum_{(\beta,\gamma )\in \tilde S}  c_\beta \binom{\beta}{\gamma} \int_{\mathbb{T}^d} H_\gamma \left(  \mathbb{Y}_M -\mathbb{Y}_M(\eta) + \sigma_M^{\frac{1}{2}}\boldsymbol{b} ; \sigma_{\alpha,M}  \right)   \boldsymbol{a}_M^{\beta-\gamma}dx  \right] \notag \\
&= -\mathbb{E} \left[ \sum_{(\beta,\gamma )\in \tilde S}  c_\beta \binom{\beta}{\gamma} \int_{\mathbb{T}^d} H_\gamma \left(  \frac{\mathbb{Y}_M -\mathbb{Y}_M(\eta)}{\sigma_{\alpha,M}^{\frac{1}{2}}} + \boldsymbol{b} ; 1  \right) \sigma_{\alpha,M}^{\frac{|\gamma|}{2}} M^{\frac{d+\delta}{\kappa_1}((\beta-\gamma) \cdot \boldsymbol{r})}  \boldsymbol{a}^{\beta-\gamma}dx  \right] \notag \\
& \lesssim \sum_{(\beta,\gamma )\in \tilde S} \sigma_{\alpha,M}^{\frac{|\gamma|}{2}} M^{\frac{d+\delta}{\kappa_1}((\beta-\gamma) \cdot \boldsymbol{r})} \lesssim (\log M)^{h_1} M^{h_2}\label{ronaldo}
\end{align}
with ``$h_2 < \frac{\kappa_2}{\kappa_1}(d+\delta)$'' or ``$h_2 = \frac{\kappa_2}{\kappa_1}(d+\delta)$ and $h_1 < \frac{\kappa_3}{2}$'' . Therefore, from \eqref{kaka} and \eqref{ronaldo}, we can see that \eqref{pistol}  holds.

\end{proof}

Note that the following lemma holds not only for $\alpha= \frac{d}{2}$ but also for any $\alpha \in\mathbb{R}$.
\begin{lemm}\label{hq3}
\[  \mathbb{E} \left[  \int_{\mathbb{T}^d} \int^1_0 \left|\langle \nabla \rangle^{\alpha}\boldsymbol{Z}_{M}^\eta(t)\right|^2 dt   dx\right] \lesssim M^d, \]
where $\boldsymbol{Z}_{M}^\eta(t)$ is defined by \eqref{popo}, 
\begin{equation}\label{oumou}
\mathbb{Y}_M(t) = P_M \langle \nabla \rangle^{-\alpha} \mathbb{W}(t) 
\end{equation}
is defined in Section \ref{varia} and $\mathbb{W}(t)=(W_1(t), \cdots, W_n(t))$.
\end{lemm}

\begin{proof}
For $0\le t <\eta$,
\[  \mathbb{E} \left[ \left| \langle \nabla \rangle^{\alpha}\boldsymbol{Z}_{M}^\eta(t)\right|^2 \right]=0. \]
For $\eta \le t \le 1$,
\[ \langle \nabla \rangle^{\alpha}\boldsymbol{Z}_{M}^\eta(t) = \frac{1}{1-\eta} \langle \nabla \rangle^{\alpha}\mathbb{Y}_M(\eta) = \frac{1}{1-\eta}\mathbb{W}(\eta)        \]
from \eqref{oumou} and thus it is easy to see that 
\[  \mathbb{E} \left[ \left|\langle \nabla \rangle^{\alpha}\boldsymbol{Z}_{M}^\eta(t)\right|^2   \right]\lesssim_\eta M^d       \]
uniformly in $\eta \le t \le 1$ and $x\in\mathbb{T}^d$.

\end{proof}

\subsection{Proof of Theorem \ref{nonA}}
In this section, we prove Theorem \ref{nonA}.

\begin{proof}[Proof of Theorem \ref{nonA}]
Let $\frac{k-1}{2k}d < \alpha < \frac{d}{2}$. We assume that there exist some $\boldsymbol{a}=(a_1,\cdots,a_n)\in\mathbb{R}^n$, $\boldsymbol{r}= (r_1,\cdots,r_n)\in\mathbb{N}^n$, $m>\frac{1}{2}$ and $q:A^- \rightarrow \mathbb{R}_+$ satisfying \eqref{paka} such that 
\begin{align}
\lim_{x\rightarrow\infty} \left\{  F (\boldsymbol{a}_{\boldsymbol{r}}(x)) + \sum_{\beta\in A^-}|\boldsymbol{a}_{\boldsymbol{r}}(x)^\beta|^{q(\beta)} -m|\boldsymbol{a}_{\boldsymbol{r}}(x)|^2 \right\} = +\infty,
\end{align}
where $\boldsymbol{a}_{\boldsymbol{r}}(x) \coloneqq (a_1 x^{r_1}, \cdots, a_n x^{r_n})$.
We set $\kappa_1$ and $\kappa_2$ by 
\begin{equation}\label{ryu}
F (\boldsymbol{a}_{\boldsymbol{r}}(x)) + \sum_{\beta\in A^-}|\boldsymbol{a}_{\boldsymbol{r}}(x)^\beta|^{q(\beta)} -m|\boldsymbol{a}_{\boldsymbol{r}}(x)|^2 \asymp |x|^{\kappa_1} \ \ \ \mbox{as}\ |x|\rightarrow \infty
\end{equation}
and
\begin{equation}\label{ryu2}
\kappa_2 \coloneqq \max_{\beta\in A^-,\ \boldsymbol{a}^\beta \neq 0} q(\beta) ( \beta \cdot \boldsymbol{r}).
\end{equation}
Note that $\kappa_1 \ge \kappa_2$ holds.
Given $M\in\mathbb{N}$ and $\delta\ge 0$, we define $\theta^M=(\theta^M_1,\cdots,\theta^M_n)\in\mathbb{H}$ by
\[ \theta^M (t) \coloneqq - \langle \nabla \rangle^{\frac{d}{2}}\boldsymbol{Z}_M^\eta (t) + \boldsymbol{a}_{\boldsymbol{r}}(M^{\frac{d+\delta}{\kappa_1}}) + \sigma_{\alpha,M}^{\frac{1}{2}}\boldsymbol{b} \]
where this definition is exactly the same as the one in the proof of Theorem \ref{nornon} (ii).
Then, by proceeding in exactly the same way as in the proof of Theorem \ref{nornon} (ii),  we obtain
\begin{align}
&\lim_{L\rightarrow\infty}\lim_{N\rightarrow\infty}\mathbb{E} \left[ -\left( R_N^F(\mathbb{Y}+I(\theta^M))\wedge L \right) + \frac{1}{2}\int^1_0 \| P_N^{\otimes n} \theta^M (t)\|^2_{L^2(\mathbb{T}^d)}dt  \right]\notag\\
&= -\int_{\mathbb{T}^d} \left\{ F (\boldsymbol{a}_M) + \sum_{\beta\in A^-}|\boldsymbol{a}_M^\beta|^{q(\beta)} -m|\boldsymbol{a}_M|^2 \right\}dx + \int_{\mathbb{T}^d} \sum_{\beta\in A^-}|\boldsymbol{a}_M^\beta|^{q(\beta)} dx \notag\\
&\ - \mathbb{E} \left[ \sum_{\beta\in  A}  c_\beta \sum_{0<\gamma\le\beta,\ \gamma\in\mathbb{N}^n} \binom{\beta}{\gamma} \sum_{0\le \tilde \gamma \le \gamma, \tilde \gamma \in \mathbb{N}^n} \binom{\gamma}{\tilde \gamma} \int_{\mathbb{T}^d} :\mathbb{Y}^{\tilde\gamma}: (- \mathbb{Y}_M(\eta) + \sigma_{\alpha,M}^{\frac{1}{2}}\boldsymbol{b})^{\gamma - \tilde \gamma} \boldsymbol{a}_M^{\beta-\gamma}dx  \right]\notag \\
&\ +\frac{1}{2} \mathbb{E} \left[  \int_{\mathbb{T}^d} \left\{ \sigma_{\alpha, M}^{\frac{1}{2}}\boldsymbol{a}_M \cdot\boldsymbol{b} + \sigma_{\alpha, M}|\boldsymbol{b}|^2 +  \int^1_0 \left|\langle \nabla \rangle^{\frac{d}{2}}\boldsymbol{Z}_{M}^\eta(t)\right|^2 dt     \right\} dx\right] \notag \\
&\ + \left(\frac{1}{2} - m\right)\int_{\mathbb{T}^d} |\boldsymbol{a}_M|^2 dx. \label{ryu3}
\end{align}
We divide the rest of the proof into two parts: The case $\kappa_1 > \kappa_2$ and the case $\kappa_1 = \kappa_2$. \\
\textbf{(In the case that $\boldsymbol{\kappa_1 > \kappa_2}$)}
By a similar argument to the proof of Theorem \ref{nornon} (ii), we can choose $0<\eta<1$ and $\boldsymbol{b}\in\mathbb{R}^n$ such that
\begin{align}\label{pistol2}
&- \mathbb{E} \left[ \sum_{\beta\in  A}  c_\beta \sum_{0<\gamma\le\beta,\ \gamma\in\mathbb{N}^n} \binom{\beta}{\gamma} \sum_{0\le \tilde \gamma \le \gamma, \tilde \gamma \in \mathbb{N}^n} \binom{\gamma}{\tilde \gamma} \int_{\mathbb{T}^d} :\mathbb{Y}^{\tilde\gamma}: (- \mathbb{Y}_{M}(\eta) + \sigma_{\alpha,M}^{\frac{1}{2}}\boldsymbol{b})^{\gamma - \tilde \gamma} \boldsymbol{a}_{M}^{\beta-\gamma}dx  \right] \notag \\
&\lesssim  1 
\end{align}
uniformly in $M \in \mathbb{N}^n$.
Then, from \eqref{ryu}, \eqref{ryu2}, \eqref{ryu3}, \eqref{pistol2}, \eqref{dyn}, Lemma \ref{hq3} and $\kappa_1>\kappa_2$,
\begin{align*}
&\lim_{L\rightarrow\infty}\lim_{N\rightarrow\infty} \mathbb{E} \left[  -\left( R_N^F(\mathbb{Y}+I(\theta^{M}))\wedge L \right) + \frac{1}{2}\int^1_0 \| P_N^{\otimes n} \theta^{M} (t)\|^2_{L^2(\mathbb{T}^d)}dt \right]\\
&\lesssim -cM^{d+\delta} + M^{\frac{\kappa_2}{\kappa_1}(d+\delta)} +(\sigma_{\alpha,M})^{\frac{1}{2}} M^{\frac{r^{\max}}{\kappa_1}(d+\delta)} + \sigma_{\alpha,M} + M^d -cM^{\frac{2r^{\max}}{\kappa_1}(d+\delta)} + 1 \\
&\lesssim -cM^{d+\delta}  + 1 \rightarrow -\infty
\end{align*}
as $M\rightarrow \infty$ where $r^{\max} \coloneqq \max_{1\le i\le n, a_i \neq 0}r_i$, and this proves Theorem \ref{norA} (ii) in this case.
Note that the term $(\sigma_{\alpha,M})^{\frac{1}{2}} M^{\frac{r^{\max}}{\kappa_1}(d+\delta)}$ can be controlled as
\[ (\sigma_{\alpha,M})^{\frac{1}{2}} M^{\frac{r^{\max}}{\kappa_1}(d+\delta)} \le C_\epsilon \sigma_{\alpha, M} + \epsilon M^{\frac{2r^{\max}}{\kappa_1}(d+\delta)}   \]
for any $\epsilon >0$, for example. \\
\textbf{(In the case that $\boldsymbol{\kappa_1 = \kappa_2}$)}
We set $\delta = 0$. If we can choose $0<\eta<1$ and $\boldsymbol{b}\in\mathbb{R}^n$ such that
\begin{align}\label{pistol3}
&- \mathbb{E} \left[ \sum_{\beta\in  A}  c_\beta \sum_{0<\gamma\le\beta,\ \gamma\in\mathbb{N}^n} \binom{\beta}{\gamma} \sum_{0\le \tilde \gamma \le \gamma, \tilde \gamma \in \mathbb{N}^n} \binom{\gamma}{\tilde \gamma} \int_{\mathbb{T}^d} :\mathbb{Y}^{\tilde\gamma}: (- \mathbb{Y}_{M}(\eta) + \sigma_{\alpha,M}^{\frac{1}{2}}\boldsymbol{b})^{\gamma - \tilde \gamma} \boldsymbol{a}_{M}^{\beta-\gamma}dx  \right] \notag \\
&\lesssim -cM^h + 1
\end{align}
for some $h>d$, then, from \eqref{dyn} and $\kappa_1=\kappa_2$,
\begin{align*}
&\lim_{L\rightarrow\infty}\lim_{N\rightarrow\infty} \mathbb{E} \left[  -\left( R_N^F(\mathbb{Y}+I(\theta^{M}))\wedge L \right) + \frac{1}{2}\int^1_0 \| P_N^{\otimes n} \theta^{M} (t)\|^2_{L^2(\mathbb{T}^d)}dt \right]\\
&\lesssim -cM^{d} + M^{\frac{\kappa_2}{\kappa_1}d} - cM^h + (\sigma_{\alpha,M})^{\frac{1}{2}} M^{\frac{r^{\max}}{\kappa_1}d} + \sigma_{\alpha, M} + M^d -cM^{\frac{2r^{\max}}{\kappa_1}d} + 1 \rightarrow -\infty
\end{align*}
as $M\rightarrow \infty$ and this proves Theorem \ref{nonA}. In the following, we prove the existence of $0<\eta<1$ and $\boldsymbol{b}\in\mathbb{R}^n$ which satisfy \eqref{pistol3}. 
From Lemma \ref{gausstoku}, there holds
\begin{align*}
&- \mathbb{E} \left[ \sum_{\beta\in  A}  c_\beta \sum_{0<\gamma\le\beta,\ \gamma\in\mathbb{N}^n} \binom{\beta}{\gamma} \sum_{0\le \tilde \gamma \le \gamma, \tilde \gamma \in \mathbb{N}^n} \binom{\gamma}{\tilde \gamma} \int_{\mathbb{T}^d} :\mathbb{Y}^{\tilde\gamma}: (- \mathbb{Y}_M(\eta) + \sigma_{\alpha,M}^{\frac{1}{2}}\boldsymbol{b})^{\gamma - \tilde \gamma} \boldsymbol{a}_M^{\beta-\gamma}dx  \right] \\
&= - \mathbb{E} \left[ \sum_{\beta\in  A}  c_\beta \sum_{0<\gamma\le\beta,\ \gamma\in\mathbb{N}^n} \binom{\beta}{\gamma} \sum_{0\le \tilde \gamma \le \gamma, \tilde \gamma \in \mathbb{N}^n} \binom{\gamma}{\tilde \gamma} \int_{\mathbb{T}^d} :\mathbb{Y}_M^{\tilde\gamma}: (- \mathbb{Y}_M(\eta) + \sigma_{\alpha,M}^{\frac{1}{2}}\boldsymbol{b})^{\gamma - \tilde \gamma} \boldsymbol{a}_M^{\beta-\gamma}dx  \right] \\
&= - \mathbb{E} \left[ \sum_{\beta\in  A}  c_\beta \sum_{0<\gamma\le\beta,\ \gamma\in\mathbb{N}^n} \binom{\beta}{\gamma} \int_{\mathbb{T}^d} H_\gamma \left(  \mathbb{Y}_M -\mathbb{Y}_M(\eta) + \sigma_{\alpha,M}^{\frac{1}{2}}\boldsymbol{b} ; \sigma_{\alpha,M}  \right)   \boldsymbol{a}_M^{\beta-\gamma}dx  \right]\\
&=  - \mathbb{E} \left[ \sum_{\beta\in  A}  c_\beta \sum_{0<\gamma\le\beta,\ \gamma\in\mathbb{N}^n} \binom{\beta}{\gamma} \int_{\mathbb{T}^d} H_\gamma \left(  \frac{\mathbb{Y}_M -\mathbb{Y}_M(\eta)}{\sigma_{\alpha,M}^{\frac{1}{2}}} + \boldsymbol{b} ; 1  \right) \sigma_{\alpha,M}^{\frac{|\gamma|}{2}} M^{\frac{d}{\kappa_1}((\beta-\gamma) \cdot \boldsymbol{r})}  \boldsymbol{a}^{\beta-\gamma}dx  \right].       
\end{align*}
Then, noting that $ \sigma_{\alpha,M}^{\frac{|\gamma|}{2}} M^{\frac{d}{\kappa_1}((\beta-\gamma) \cdot \boldsymbol{r})} \asymp M^{\frac{|\gamma|}{2}(d-2\alpha) + \frac{d }{\kappa_1} (\beta -\gamma)\cdot \boldsymbol{r}}$ from \eqref{dyn}, we define
\begin{equation}
\kappa_4 \coloneqq  \max \left\{ \frac{|\gamma|}{2}(d-2\alpha) + \frac{d }{\kappa_1} (\beta -\gamma)\cdot \boldsymbol{r} \ ;\ \beta\in\tilde A,\ \gamma\in\mathbb{N}^n,\ 0< \gamma\le\beta,\ \boldsymbol{a}^{\beta-\gamma}\neq 0  \right\}. 
\end{equation}
Then, by a similar argument to the last part of the proof of Theorem \ref{nornon} (ii), we can choose $\eta$ and $\boldsymbol{b}$ satisfying
\begin{align}
&- \mathbb{E} \left[ \sum_{\beta\in  A}  c_\beta \sum_{0<\gamma\le\beta,\ \gamma\in\mathbb{N}^n} \binom{\beta}{\gamma} \sum_{0\le \tilde \gamma \le \gamma, \tilde \gamma \in \mathbb{N}^n} \binom{\gamma}{\tilde \gamma} \int_{\mathbb{T}^d} :\mathbb{Y}^{\tilde\gamma}: (- \mathbb{Y}_{M}(\eta) + \sigma_{\alpha,M}^{\frac{1}{2}}\boldsymbol{b})^{\gamma - \tilde \gamma} \boldsymbol{a}_{M}^{\beta-\gamma}dx  \right] \notag \\
&\asymp - M^{\kappa_4}.\label{pistol8}
\end{align}
Moreover, from $\kappa_1 = \kappa_2$, \eqref{paka} and \eqref{ryu2}, there holds
\begin{align}
\kappa_4 &= \max \left\{ \frac{|\gamma|}{2}(d-2\alpha) + \frac{d }{\kappa_1} (\beta -\gamma)\cdot \boldsymbol{r} \ ;\ \beta\in A,\ \gamma\in\mathbb{N}^n,\ 0< \gamma\le\beta,\ \boldsymbol{a}^{\beta-\gamma}\neq 0  \right\} \notag \\
&= \max \left\{ \frac{|\xi - \beta|}{2}(d-2\alpha) + \frac{d }{\kappa_2} (\beta\cdot \boldsymbol{r}) \ ;\ \beta\in A^-,\ \xi\in  A,\ \xi > \beta,\ \boldsymbol{a}^{\beta}\neq 0  \right\} \notag \\
&=\min_{(\eta\in A^-,\ \boldsymbol{a}^\eta \neq 0)} \max_{(\beta\in A^-,\ \xi\in  A,\ \xi > \beta,\ \boldsymbol{a}^{\beta}\neq 0 )}  \left\{ \frac{|\xi - \beta|}{2}(d-2\alpha) + \frac{d}{q (\eta)(\eta \cdot \boldsymbol{r})} (\beta\cdot \boldsymbol{r}) \right\} \notag \\
&> d. \label{pistol9}
\end{align}
Therefore, from \eqref{pistol8} and \eqref{pistol9}, we can see that \eqref{pistol3} holds.
\end{proof}

\section{Normalizability 2}\label{norma3}

In this section, we prove Theorem \ref{nornon1} (i).
Similarly to the proof of Theorem \ref{norA}, it suffices to prove that
\begin{align}
&\limsup_{N\rightarrow\infty} \inf_{\theta\in\mathbb{H}} \mathbb{E} \Bigg[ -R_N^F(\mathbb{Y}+I(\theta))+ K \left| \int_{\mathbb{T}^d} \sum_{i=1}^n H_2 (P_N Y_i + P_N I_i(\theta);\sigma_{\alpha,N})dx \right|^b \notag \\
&\ \ \ \ \ \ \ \ \ \ \ \ \ \ \ \ \ \ \ \ \ \ \ \ \ \ \ \ \ \ \ \ \ \ \ \ \ \ \ \ \ \ \ \ \ \ \ \ \ \ \ \ \ \ \ \ \ \ \ \ \ \ \ \ \ \ \ \ \ \ \ \  + \frac{1}{2}\int^1_0 \| P_N^{\otimes n} \theta (t)\|^2_{L^2(\mathbb{T}^d)}dt \Bigg] >  -\infty , \label{moku2}
\end{align}
to show the normalizability, see the argument just before the proof of Theorem \ref{norA} in Section \ref{norma1}. 

\subsection{Lemmas}
Before the proof, we prepare a few lemmas.

\begin{lemm}\label{gogo}
Let $\frac{k-1}{2k}d < \alpha \le \frac{d}{2}$. For any $0<b<\frac{d}{d-2\alpha}$ and $\epsilon>0$, there exists some $c=c(b)>0$ and $C_\epsilon >0$ such that
\begin{align*}
&\left| \int_{\mathbb{T}^d} \sum_{i=1}^n \left( :Y_{i,N}^2:+2Y_{i,N}I_{i,N}(\theta) + I_{i,N}(\theta)^2  \right) dx \right|^b \ge c\| I_N (\theta) \|_{L^2(\mathbb{T}^d)}^{2b} -\epsilon \| I_N (\theta) \|_{H^\alpha(\mathbb{T}^d)}^2 - C_\epsilon Q_N,
\end{align*}
where $Q_N$ is some random variable as in the beginning of Section \ref{norma1}.
\end{lemm}

\begin{proof}
For simplicity of notation, we only prove the case $n=1$.
From the elementary inequality
\begin{equation}
|x + y|^b \lesssim_b |x|^b + |y|^b,
\end{equation}
there holds
\begin{align*}
&\left| \int_{\mathbb{T}^d} \left( :Y_{1,N}^2:+2Y_{1,N}I_{1,N}(\theta) + I_{1,N}(\theta)^2  \right)dx \right|^b \\
&\ge c_1\| I_N (\theta) \|_{L^2(\mathbb{T}^d)}^{2b} - \left| \int_{\mathbb{T}^d} :Y_{1,N}^2:dx \right|^b - c_2\left| \int_{\mathbb{T}^d}  Y_{1,N}I_{1,N}(\theta)dx \right|^b 
\end{align*}
for some $c_1, c_2>0$. Moreover, from Lemmas \ref{duality}, \ref{inter} and Young's inequality,
\begin{align*}
\left| \int_{\mathbb{T}^d}  Y_{1,N}I_{1,N}(\theta)dx \right|^b &\lesssim \| Y_{1,N} \|^b_{W^{\alpha -\frac{d}{2}-\delta, \infty}(\mathbb{T}^d)} \| I_{1,N}(\theta) \|_{H^{\frac{d}{2}-\alpha + \delta}(\mathbb{T}^d)}^b \\
&\lesssim \| Y_{1,N} \|^b_{W^{\alpha -\frac{d}{2}-\epsilon, \infty}(\mathbb{T}^d)}  \| I_{1,N}(\theta) \|_{L^2(\mathbb{T}^d)}^{b\{1-\frac{1}{\alpha}(\frac{d}{2}-\alpha +\delta)\}} \| I_{1,N}(\theta) \|_{H^\alpha (\mathbb{T}^d)}^{\frac{b}{\alpha}(\frac{d}{2}-\delta +\delta)}\\
&\lesssim Q_N + \epsilon \| I_{1,N} (\theta) \|_{L^2(\mathbb{T}^d)}^{2b} + \epsilon \| I_{1,N} (\theta) \|_{H^\alpha(\mathbb{T}^d)}^2
\end{align*}
for any $\epsilon >0$. Note that if $0<b<\frac{d}{d-2\alpha}$, there holds
\[  \frac{b\{1-\frac{1}{\alpha}(\frac{d}{2}-\alpha +\delta)\}}{2b} + \frac{\frac{b}{\alpha}(\frac{d}{2}-\alpha +\delta)}{2} < 1    \]
for sufficiently small $\delta>0$.

\end{proof}

\begin{lemm}\label{kann}
Let $2\le p <4$. If $\frac{p}{\alpha} ( \frac{d}{2} - \frac{d}{p}  )<2$,
\[ \| f \|_{L^p(\mathbb{T}^d)}^p \le  \epsilon \| f \|_{H^\alpha(\mathbb{T}^d)}^2 + \epsilon \| f \|_{L^2(\mathbb{T}^d)}^{r} + C_{\epsilon, r} \]
for any $\epsilon>0$ and $r > \frac{2p\left\{ 1-\frac{1}{\alpha}(\frac{d}{2} -\frac{d}{p})\right\}}{2-\frac{p}{\alpha}( \frac{d}{2}-\frac{d}{p} )}$.
\end{lemm}

\begin{proof}
From Lemmas \ref{embed} and \ref{inter} and Young's inequality,
\begin{align*}
\| f \|_{L^p(\mathbb{T}^d)}^p &\le C \| f \|^p_{H^{\frac{d}{2}-\frac{d}{p}}(\mathbb{T}^d)} \le C' \| f \|^{\frac{p}{\alpha}(\frac{d}{2}-\frac{d}{p})}_{H^\alpha (\mathbb{T}^d)} \| f \|^{p\{ 1- \frac{1}{\alpha}(\frac{d}{2}-\frac{d}{p})\}}_{L^2(\mathbb{T}^d)} \\
&\le \epsilon \| f \|_{H^\alpha(\mathbb{T}^d)}^2 + \epsilon\| f \|_{L^2(\mathbb{T}^d)}^{pq\{ 1- \frac{1}{\alpha}(\frac{d}{2}-\frac{d}{p})\}} + C_{\epsilon,q}
\end{align*}
if $\frac{p}{\alpha} ( \frac{d}{2} - \frac{d}{p}  )<2$, where we chose $q$ as
\[ q > \frac{1}{1-\frac{p}{2\alpha}(\frac{d}{2}-\frac{d}{p})} \iff \frac{p}{2\alpha}\left(\frac{d}{2}-\frac{d}{p}\right) + \frac{1}{q} <1.\]
\end{proof}

\begin{lemm}\label{mizu}
Let $\frac{k-1}{2k}d < \alpha \le \frac{d}{2}$. For any $\kappa<\frac{2d}{d-\alpha}$, there exists some $0<b<\frac{d}{d-2\alpha}$ such that
\[ \| f \|_{L^\kappa (\mathbb{T}^d)}^\kappa \le \epsilon \| f \|_{H^\alpha(\mathbb{T}^d)}^2 + \epsilon \| f \|_{L^2(\mathbb{T}^d)}^{2b} + C_{\epsilon} \]
for any $\epsilon>0$.
\end{lemm}

\begin{proof}
One can check that
\[ \frac{2p\left\{ 1-\frac{1}{\alpha}(\frac{d}{2} -\frac{d}{p})\right\}}{2-\frac{p}{\alpha}( \frac{d}{2}-\frac{d}{p} )} < \frac{2d}{d-2\alpha} \]
if $p<\frac{2d}{d-\alpha}$. Therefore, the desired result follows from Lemma \ref{kann}.
\end{proof}

The following lemma can be proven similarly to Lemma \ref{tuitui}.
\begin{lemm}\label{atoato}
Let $q:A^-\rightarrow \mathbb{R}_+$. Then,
\[ \sum_{i=1}^n \int_{\mathbb{T}^d} |I_{i,N}(\theta)|^{\tilde \kappa_i(q)}dx \lesssim \sum_{\beta \in A^-} \int_{\mathbb{T}^d} |I_N(\theta)^\beta|^{q(\beta)} dx + \int_{\mathbb{T}^d} |I_N(\theta)|^{\kappa}dx     \]
where $\tilde \kappa_i(q)$ is defined by \eqref{kawadou2}.
\end{lemm}

\subsection{Proof of Theorem \ref{nornon1} (i)}

\begin{proof}[Proof of Theorem \ref{nornon1} \textup{(i)}]
Let $\frac{k-1}{2k}d < \alpha \le \frac{d}{2}$. We assume that $F$ satisfies the assumption of Theorem \ref{nornon1} (i).
From \eqref{rndef} and \eqref{hersei},
\begin{align}
-R_N^F(\mathbb{Y}+ I(\theta)) &= -\sum_{\beta \in A} c_\beta \int_{\mathbb{T}^d} H_\beta (P_N^{\otimes n}(\mathbb{Y}+ I(\theta));\sigma_{\alpha, N}) dx \notag\\
&= -\int_{\mathbb{T}^d}  F(P_N^{\otimes n}I(\theta)) dx - \sum_{\beta \in  A} c_\beta \sum_{0<\gamma\leq \beta,\ \gamma\in\mathbb{N}^n} \binom{\beta}{\gamma}\int_{\mathbb{T}^d} :\mathbb{Y}_N^\gamma: I_N(\theta)^{\beta - \gamma} dx  \notag\\
&\eqqcolon A_N^1 + A_N^2. \label{pr1g}
\end{align}
From \eqref{lalalag},
\begin{equation}\label{new1g}
A_N^1 \gtrsim \sum_{\beta \in A^-}\int_{\mathbb{T}^d} |I_N(\theta)^{\beta}|^{q(\beta)} dx - C\int_{\mathbb{T}^d} |I_{N}(\theta)|^{\kappa}dx - C 
\end{equation}
for some $C>0$, $\kappa <\frac{2d}{d-\alpha}$ and $q: A^- \rightarrow \mathbb{R}_+$ with $q>\tilde p^{\alpha,F,q}$. Therefore, from \eqref{new1g}, \eqref{new3} and Lemma \ref{gogo}, we obtain 
\begin{align}
&A_N^1+ K \left| \int_{\mathbb{T}^d} \sum_{i=1}^n H_2 (P_N Y_i + P_N I_i(\theta);\sigma_{\alpha,N})dx \right|^b +  \frac{1}{2}\int^1_0 \| P_N^{\otimes n}\theta (t)\|^2_{L^2(\mathbb{T}^d)}dt \notag \\
&\gtrsim \sum_{\beta\in A^-}\int_{\mathbb{T}^d}|I_N(\theta)^\beta |^{q(\beta)} dx + \| I_N(\theta) \|_{L^2(\mathbb{T}^d)}^{2b}- C\int_{\mathbb{T}^d} |I_{N}(\theta)|^{\kappa}dx + \| I_N(\theta) \|^2_{{H}^{\alpha}(\mathbb{T}^d)}  - Q_N   \label{pr2g}
\end{align}
On the other hand, for each term of $A_N^2$, by proceeding exactly in the same way as in the proof of Theorem \ref{norA}, we obtain
\begin{align}
&\left|\int_{\mathbb{T}^d} :\mathbb{Y}_N^\gamma: I_N(\theta)^{\beta - \gamma} dx \right| \notag \\
&\leq \epsilon ' \| I_N(\theta)^{\beta -\gamma}\|_{L^{q(\beta -\gamma)}(\mathbb{T}^d)}^{q(\beta -\gamma)} + \epsilon '  \| I_{N}(\theta)\|^2_{H^\alpha(\mathbb{T}^d)} + \epsilon ' \sum_{i=1}^n \| I_{i,N}(\theta)\|_{L^{\tilde \kappa_i(q)}(\mathbb{T}^d)}^{\tilde \kappa_i(q)} + C_{\epsilon '}Q_N\label{simeg}
\end{align}
for any $\epsilon '>0$.
Then, from \eqref{pr1g}, \eqref{pr2g}, \eqref{simeg}, Lemma \ref{mizu} and Lemma \ref{atoato}, we get
\begin{align*}
&\mathbb{E} \left[ -R_N^F(\mathbb{Y}+I(\theta)) + K \left| \int_{\mathbb{T}^d} \sum_{i=1}^n H_2 (P_N Y_i + P_N I_i(\theta);\sigma_N) \right|^b + \frac{1}{2}\int^1_0 \| P_N^{\otimes n} \theta (t)\|^2_{L^2(\mathbb{T}^d)}dt \right]\\
&\gtrsim \mathbb{E} \left[ \sum_{\beta \in A^-} \int_{\mathbb{T}^d} |I_N(\theta)^\beta|^{q(\beta)} dx  + \| I_N(\theta) \|_{L^2(\mathbb{T}^d)}^{2b}- C\int_{\mathbb{T}^d} |I_{N}(\theta)|^{\kappa}dx  +   \| I_N(\theta) \|^2_{H^{\alpha}(\mathbb{T}^d)} \right. \notag \\
&\ \left.\ \ \ \ \ \ \ \ \ \ \ \ \ \ \ \ \ \ \ \ \ \ \ \ \ \ \ \ \ \ \ \ \ \ \ \ \ \ \ \ \ \ \ \ \ \ \ \ \ \ \ \ \ \ \ \ \ \ \ \ \ \ \ \ \ \ \ \ \ \ \ \ \ \ \ \ \  - \epsilon ' \sum_{i=1}^n \int_{\mathbb{T}^d} |I_{i,N}(\theta)|^{\tilde \kappa_i(q)}dx \right]  - C \notag \\
&\gtrsim \mathbb{E} \left[ \sum_{\beta \in A^-} \int_{\mathbb{T}^d} |I_N(\theta)^\beta|^{q(\beta)} dx + \| I_N(\theta) \|_{L^2(\mathbb{T}^d)}^{2b} +   \| I_N(\theta) \|^2_{H^{\alpha}(\mathbb{T}^d)}    \right] - C
\end{align*}
by choosing sufficiently small $\epsilon ' >0$ and sufficiently large $b$ so that
\[  2b > \left( \frac{2\kappa\left\{ 1-\frac{1}{\alpha}(\frac{d}{2} -\frac{d}{\kappa})\right\}}{2-\frac{\kappa}{\alpha}( \frac{d}{2}-\frac{d}{\kappa} )}  \right).     \]
Then, \eqref{moku2} follows immediately from this.
\end{proof}

\section{Non-normalizability 2}\label{norma4}
In this section, we prove Theorems  \ref{nono} (ii) and \ref{nornon1} (ii).
Similarly to the proof of Theorems \ref{nonA} and \ref{nornon} (ii), it suffices to prove
\begin{align}
&\limsup_{L\rightarrow\infty}\limsup_{N\rightarrow\infty} \inf_{\theta\in\mathbb{H}} \mathbb{E} \Bigg[ -\min \left( R_N^F(\mathbb{Y}+I(\theta)), L \right) +  K\left| \int_{\mathbb{T}^d} \sum_{i=1}^n H_2 (P_N Y_i + P_N I_i(\theta);\sigma_{\alpha,N})    \right|^b  \notag \\
& \ \ \ \ \ \ \ \ \ \ \ \ \ \ \ \ \ \ \ \ \ \ \ \ \ \ \ \ \ \ \ \ \ \ \ \ \ \ \ \ \ \ \ \ \ \ \ \ \ \ \ \ \ \ \ \ \ \ \ \ \ \ \ \ \ \ \ \ \ \ \   + \frac{1}{2}\int^1_0 \| P_N^{\otimes n} \theta (t)\|^2_{L^2(\mathbb{T}^d)}dt \Bigg] = -\infty , \label{garagara}
\end{align}
see the argument in the beginning of Section \ref{norma2}.
The strategy for the proof is basically the same as in Section \ref{norma2}, but we need to use another technique used in \cite{hartree, phi3, heat}: Let $f:\mathbb{R}^d\rightarrow \mathbb{R}$ be a Schwartz function (rapidly decreasing function) such that $f\not\equiv 0$, $\hat f$ is an even function and $\mbox{supp}(\hat f) \subset \{ |\xi|\le 1\}$ where $\mbox{supp}(\hat f)$ is the support of $f$. For $M\in \mathbb{N}$ and $r\in\mathbb{R}_+\backslash \{0\}$, we define $f_M:\mathbb{T}^d \rightarrow \mathbb{R}$ by
\begin{equation}
f_M (x) \coloneqq M^{\frac{d}{r}-d}\sum_{|l|\le M}\hat f \left( \frac{l}{M}\right)e_l (x).
\end{equation}
Then, the following lemma holds. The proof is similar to that of \cite[Lemma 5.12]{hartree}.
\begin{lemm}\label{haki}
For any $s\in \mathbb{R}_+\backslash \{0\}$, there holds
\[ \int_{\mathbb{T}^d} |f_M(x)|^{sr}dx \asymp M^{sd-d} \]
as $M \rightarrow \infty$. Moreover, if $sr\in\mathbb{N}$, there holds
\[\int_{\mathbb{T}^d}|\langle \nabla \rangle^\beta (f_M(x)^{sr})|^2 dx \lesssim M^{2\beta + 2sd -d} \]
for any $\beta\in\mathbb{R}$.
\end{lemm}
\begin{proof}
Writing $F_M(x)\coloneqq M^{\frac{d}{r}}f(Mx)$, there holds
\[ \widehat{F_M(x+\cdot)}(l) = M^{\frac{d}{r}-d}\hat f \left( \frac{l}{M}\right)e_l (x). \]
Therefore, from Poisson's summation formula,
\begin{equation}
\sum_{l\in\mathbb{Z}^d}F_M (x+l) = f_M (x).
\end{equation}
Because $f$ is a Schwartz function, for $|x|\le 1$ and $l\in \mathbb{Z}^d\backslash \{ 0\}$, 
\begin{equation}
|F_M(x+l)| = M^{\frac{d}{r}} |f(Mx+Ml)| \lesssim \frac{M^{\frac{d}{r}}}{M^\alpha |l|^\alpha}
\end{equation}
for any $\alpha >0$.
Therefore, for any $s\in \mathbb{R}_+\backslash \{0\}$ and $|x|\le 1$,
\begin{align}
|f_M(x)|^{sr} &= \Bigg| F_M(x) + \sum_{l\in\mathbb{Z}^d\backslash \{0\}}F_M(x+l)    \Bigg|^{sr}\notag\\
&\lesssim |F_M(x)|^{sr} + \Bigg| \sum_{l\in\mathbb{Z}^d\backslash \{0\}}F_M(x+l)    \Bigg|^{sr}\notag\\
&\lesssim |F_M(x)|^{sr} + M^{-\beta}\label{metago1}
\end{align}
and
\begin{align}
|f_M(x)|^{sr} &\gtrsim |F_M(x)|^{sr} - C\Bigg| \sum_{l\in\mathbb{Z}^d\backslash \{0\}}F_M(x+l)    \Bigg|^{sr} \gtrsim |F_M(x)|^{sr} - C' M^{-\beta}\label{metag2}
\end{align}
for any $\beta >0$.
Moreover, 
\begin{align}
\int_{|x|\le 1}|F_M(x)|^{sr} dx = M^{sd}\int_{|x|\le 1} |f(Mx)|^{sr} dx =  M^{sd-d}\int_{|x|\le M} |f(x)|^{sr} dx \asymp M^{sd-d}\label{metago3}
\end{align}
as $M\rightarrow \infty$.
Therefore, from \eqref{metago1}, \eqref{metag2} and \eqref{metago3}, we obtain
\begin{equation}\label{siage}
\int_{\mathbb{T}^d} |f_M(x)|^{sr}dx = \int_{|x|\le 1} |f_M(x)|^{sr} \asymp M^{sd -d}    .
\end{equation}
Moreover, from $\mbox{supp}(\hat f_M) \subset \{ |l|\le M\}$ and \eqref{siage}, 
\[ \int_{\mathbb{T}^d}|\langle \nabla \rangle^\beta (f_M(x)^{sr})|^2 dx \lesssim M^{2\beta} \| f_M^{sr}\|^2_{L^2(\mathbb{T}^d)} \lesssim M^{2\beta}M^{2sd-d} \]
if $sr\in\mathbb{N}$.

\end{proof}

\subsection{Proof of Theorem \ref{nono} (ii)}

\begin{proof}[Proof of Theorem \ref{nono} \textup{(ii)}]
Let $\alpha = \frac{d}{2}$. We assume that there exist some $\boldsymbol{a}=(a_1,\cdots,a_n)\in\mathbb{R}^n$, $\boldsymbol{r}= (r_1,\cdots,r_n)\in\mathbb{N}^n$ and $C_1 >0$ such that 
\begin{align}\label{sweat}
\lim_{x\rightarrow\infty} \left\{  F (\boldsymbol{a}_{\boldsymbol{r}}(x)) + C_1\sum_{\beta\in A^-}|\boldsymbol{a}_{\boldsymbol{r}}(x)^\beta| -C_2|\boldsymbol{a}_{\boldsymbol{r}}(x)|^{\kappa} \right\} = +\infty
\end{align}
for any $\kappa <4$ and $C_2>0$, where $\boldsymbol{a}_{\boldsymbol{r}}(x) \coloneqq (a_1 x^{r_1}, \cdots, a_n x^{r_n})$.
We set $\kappa_1$ and $\kappa_2$ by 
\begin{equation}\label{kamenn1}
F (\boldsymbol{a}_{\boldsymbol{r}}(x)) + C_1\sum_{\beta\in A^-}|\boldsymbol{a}_{\boldsymbol{r}}(x)^\beta|  \asymp x^{\kappa_1} \ \ \ \mbox{as}\ x\rightarrow \infty
\end{equation}
and
\begin{equation}\label{kamenn2}
\kappa_2 \coloneqq \max_{\beta\in A^-,\ \boldsymbol{a}^\beta \neq 0}  \beta \cdot \boldsymbol{r}.
\end{equation}
Then, from \eqref{sweat}, $\kappa_1$ must satisfy
\begin{equation}\label{toumoro}
\kappa_1 \ge \max_{1\le i \le n,\ a_i\neq 0} 4r_i \eqqcolon 4r^{\max} >0.
\end{equation}
Given $M\in\mathbb{N}$, we define $\theta^M=(\theta^M_1,\cdots,\theta^M_n)\in\mathbb{H}$ by
\begin{equation}
\theta^M(t) \coloneqq -\langle \nabla \rangle^{\frac{d}{2}} \boldsymbol{Z}_{M}^\eta (t) + \langle \nabla \rangle^{\frac{d}{2}}\boldsymbol{a}_{\boldsymbol{r}}\left(  (\log M)g_M \right) + \sigma_{\alpha,M}^{\frac{1}{2}}\boldsymbol{b} 
\end{equation}
where $\boldsymbol{Z}_{M}^\eta$ is the same one as in the proof of Theorem \ref{nornon} (ii) and $g_M$ is a deterministic function with $\hat g_M (l) = 0$ for $|l|>M$ satisfying
\begin{align}\label{yayaya}
\int_{\mathbb{T}^d} g_M^{2r^{\max}} dx \asymp 1
\end{align}
and
\begin{align}
\int_{\mathbb{T}^d} g_M^{4r^{\max}} dx \asymp M^d
\end{align}
as $M\rightarrow \infty$. Such $g_M$ indeed exists in view of Lemma \ref{haki}. Similarly to the proof of Theorems \ref{nornon} (ii) and \ref{nonA}, $0<\eta<1$ and $\boldsymbol{b} = (b_1,\cdots,b_n)\in\mathbb{R}^n$ will be determined later. Note that there holds
\[ I (\theta^M) = -\mathbb{Y}_{M}(\eta)  + \boldsymbol{a}_{\boldsymbol{r}}\left(  (\log M)g_M \right) + \sigma_{\alpha,M}^{\frac{1}{2}}\boldsymbol{b}.\]
We write $\boldsymbol{a}_M^g \coloneqq \boldsymbol{a}_{\boldsymbol{r}}\left(  (\log M)g_M \right) = (a_1(\log M)^{r_1}g_M^{r_1},\cdots,a_n(\log M)^{r_n}g_M^{r_n})$. 
Then, for $N>M$,
\begin{align}
-R_N^F(\mathbb{Y}+I(\theta^M)) &= -R_N^F\left(\mathbb{Y}-\mathbb{Y}_M (\eta)+\boldsymbol{a}_M^g + \sigma_{\alpha, M}^{\frac{1}{2}}\boldsymbol{b}\right)\notag \\
&= -\int_{\mathbb{T}^d} \left\{ F (\boldsymbol{a}_M) + C_1\sum_{\beta\in A^-}|(\boldsymbol{a}_M^g)^\beta| \right\}dx + \int_{\mathbb{T}^d} C_1\sum_{\beta\in A^-}|(\boldsymbol{a}_M^g)^\beta| dx \notag\\
&\ -\sum_{\beta\in A} c_\beta \sum_{0<\gamma\le\beta,\gamma\in\mathbb{N}^n} \binom{\beta}{\gamma} \int_{\mathbb{T}^d} H_\gamma (\mathbb{Y}_N - \mathbb{Y}_M(\eta) + \sigma_{\alpha,M}^{\frac{1}{2}}\boldsymbol{b};\sigma_{\alpha,N}) (\boldsymbol{a}_M^g)^{\beta-\gamma}dx. \label{messig}
\end{align}
From the almost same argument as the proof of Theorems \ref{nornon} (ii) and \ref{nonA}, we obtain
\begin{align}
&\lim_{L\rightarrow\infty}\lim_{N\rightarrow\infty}\mathbb{E} \left[ -\left( R_N^F(\mathbb{Y}+I(\theta^M))\wedge L \right) + K\left| \int_{\mathbb{T}^d} \sum_{i=1}^n H_2 (P_N Y_i + P_N I_i(\theta^M);\sigma_{\alpha,N})    \right|^b \right. \notag \\
&\left.\ \ \ \ \ \ \ \ \ \ \ \ \ \ \ \ \ \ \ \ \ \ \ \ \ \ \ \ \ \ \ \ \ \ \ \ \ \ \ \ \ \ \ \ \ \ \ \ \ \ \ \ \ \   \ \ \ \ \ \ \ \ \ \ \ \ \ \ \ \ \ \ \ \ \ \ \ \ \ \ \ \ \ \ \ \  + \frac{1}{2}\int^1_0 \| P_N^{\otimes n} \theta (t)\|^2_{L^2(\mathbb{T}^d)}dt  \right]\notag\\
&= -\int_{\mathbb{T}^d} \left\{ F (\boldsymbol{a}_M) + C_1\sum_{\beta\in A^-}|(\boldsymbol{a}_M^g)^\beta| \right\}dx + \int_{\mathbb{T}^d} C_1\sum_{\beta\in A^-}|(\boldsymbol{a}_M^g)^\beta| dx        \notag\\
&\ - \mathbb{E} \left[ \sum_{\beta\in  A}  c_\beta \sum_{0<\gamma\le\beta,\ \gamma\in\mathbb{N}^n} \binom{\beta}{\gamma} \sum_{0\le \tilde \gamma \le \gamma, \tilde \gamma \in \mathbb{N}^n} \binom{\gamma}{\tilde \gamma} \int_{\mathbb{T}^d} :\mathbb{Y}^{\tilde\gamma}: (- \mathbb{Y}_M(\eta) + \sigma_{\alpha,M}^{\frac{1}{2}}\boldsymbol{b})^{\gamma - \tilde \gamma}( \boldsymbol{a}_M^g)^{\beta-\gamma}dx  \right]\notag \\
&\ +\frac{1}{2} \mathbb{E} \left[  \int_{\mathbb{T}^d} \left\{ |\langle \nabla \rangle^{\frac{d}{2}} \boldsymbol{a}_M^g|^2 + \sigma_{\alpha, M}^{\frac{1}{2}}(\langle \nabla \rangle^{\frac{d}{2}} \boldsymbol{a}_M^g) \cdot\boldsymbol{b} + \sigma_{\alpha, M}|\boldsymbol{b}|^2 +  \int^1_0 \left|\langle \nabla \rangle^{\frac{d}{2}}\boldsymbol{Z}_{M}^\eta(t)\right|^2 dt     \right\} dx\right] \notag \\
&\ + \mathbb{E} \left[ K\left| \sum^n_{i=1} \int_{\mathbb{T}^d} \left\{ :Y_{i}^2: + 2Y_{i} (-Y_{i,M}(\eta) + a_i (\log M)^{r_i}g_M^{r_i} + \sigma_{\alpha,M}^{\frac{1}{2}}b_i)\right.\right.\right.\notag \\
&\ \ \ \ \ \ \ \ \ \ \ \ \ \ \ \ \ \ \ \ \ \ \ \ \ \ \ \ \ \ \ \ \ \ \ \ \ \ \ \ \ \ \ \ \ \ \ \   \ \ \ \ \ \ \ \ \ \ \          \left.\left.\left. + (-Y_{i,M}(\eta) + a_i (\log M)^{r_i}g_M^{r_i} + \sigma_{\alpha, M}^{\frac{1}{2}}b_i)^2   \right\} dx    \right|^b    \right] \notag \\
&\eqqcolon A_M^1 + A_M^2 + A_M^3 + A_M^4 + A_M^5. \label{kurai}
\end{align}
Then, from Lemma \ref{hq3}, \eqref{yayaya} and $\sigma_{\alpha,M}\asymp \log M$ as in \eqref{dyn},
\begin{equation}\label{owa1}
|A_M^4| + |A_M^5| \lesssim (\log M)^{2r^{\max}}M^d.
\end{equation}
Moreover, from \eqref{kamenn1} and \eqref{kamenn2},
\begin{equation}\label{owa2}
A_M^1 = -\int_{\mathbb{T}^d}\left\{ F (\boldsymbol{a}_M) + C_1\sum_{\beta\in A^-}|(\boldsymbol{a}_M^g)^\beta| \right\}  dx \asymp - (\log M)^{\kappa_1} \int_{\mathbb{T}^d} |g_M|^{\kappa_1} dx 
\end{equation}
and
\begin{equation}\label{owa3}
|A_M^2| \lesssim (\log M)^{\kappa_2} \int_{\mathbb{T}^d} |g_M|^{\kappa_2} dx
\end{equation}
as $M\rightarrow \infty$.
Note that because $\kappa_1 \ge 4 r^{\max}$ as mentioned in \eqref{toumoro}, there holds 
\begin{equation}\label{owa4}
\int_{\mathbb{T}^d} |g_M|^{\kappa_1} dx \gtrsim M^d.
\end{equation}
For the third term of the right hand side of \eqref{kurai}, by a similar argument to the last parts of the proofs of Theorems \ref{nornon} (ii) and \ref{nonA}, we can choose $0<\eta<1$ and $\boldsymbol{b}\in\mathbb{R}^n$ such that
\begin{align}\label{owa5}
A_{M}^3 \asymp - (\log M)^h (\log M)^{\kappa_2} \int_{\mathbb{T}^d} |g_{M}|^{\kappa_2} dx = -(\log M)^{h +\kappa_2} \int_{\mathbb{T}^d} |g_{M}|^{\kappa_2} dx
\end{align}
as $M\rightarrow \infty$ for some $h>0$.
Therefore, from \eqref{kurai}, \eqref{owa1}, \eqref{owa2}, \eqref{owa3} and \eqref{owa5}, 
\begin{align}
&\lim_{L\rightarrow\infty}\lim_{N\rightarrow\infty}\mathbb{E} \left[ -\left( R_N^F(\mathbb{Y}+I(\theta^{M}))\wedge L \right) + K\left| \int_{\mathbb{T}^d} \sum_{i=1}^n H_2 (P_N Y_i + P_N I_i(\theta^{M});\sigma_{\alpha,N})    \right|^b  \right. \notag \\
&\left.\ \ \ \ \ \ \ \ \ \ \ \ \ \ \  \ \ \ \ \ \ \ \ \ \ \ \ \ \ \  \ \ \ \ \ \ \ \ \ \ \ \ \ \ \ \ \ \  \ \ \ \ \ \ \ \ \ \ \ \ \ \ \ \ \ \ \ \ \ \ \ \ \ \ \ \ \ \ \ \ \ \ \ \ \  + \frac{1}{2}\int^1_0 \| P_N^{\otimes n} \theta^{M} (t)\|^2_{L^2(\mathbb{T}^d)}dt  \right]\notag \\ 
&\lesssim -c(\log M)^{\kappa_1} \int_{\mathbb{T}^d} |g_{M}|^{\kappa_1} dx + (\log M)^{\kappa_2} \int_{\mathbb{T}^d} |g_{M}|^{\kappa_2} dx - c(\log M)^{h + \kappa_2} \int_{\mathbb{T}^d} |g_{M}|^{\kappa_2} dx \notag \\
&\ \ \ \ \ \ \ \ \ \ \ \ \ \ \ \ \ \  \ \ \ \ \ \ \ \ \ \ \ \ \ \ \ \ \ \ \ \ \ \ \ \ \ \ \ \ \ \ \ \ \  \ \ \ \ \ \ \ \ \ \ \ \ \ \ \ \ \ \ \ \ \ \ \ \ \ \ \ \ \ \ \ \ \ \ \ \ \ \  + (\log M)^{2r^{\max}}M^d + 1 \label{kimono}
\end{align} 
And from \eqref{owa4} and $\kappa_1 \ge 4 r^{\max}$, we can see that the right hand side of \eqref{kimono} goes to $-\infty$ as $M\rightarrow \infty$.
This proves \eqref{garagara}.
\end{proof}

\subsection{Proof of Theorem \ref{nornon1} (ii)}

\begin{proof}[Proof of Theorem \ref{nornon1} \textup{(ii)}]
Let $\frac{k-1}{2k}d<\alpha<\frac{d}{2}$. We assume that there exist some $q:A^- \rightarrow \mathbb{R}_+$, $\boldsymbol{a}=(a_1,\cdots,a_n)\in\mathbb{R}^n$ and $\boldsymbol{r}= (r_1,\cdots,r_n)\in\mathbb{N}^n$ satisfying \eqref{paka2} such that 
\begin{align}
\lim_{x\rightarrow\infty} \left\{  F (\boldsymbol{a}_{\boldsymbol{r}}(x)) + \sum_{\beta\in A^-}|\boldsymbol{a}_{\boldsymbol{r}}(x)^\beta|^{q(\beta)} -C|\boldsymbol{a}_{\boldsymbol{r}}(x)|^{\frac{2d}{d-\alpha}} \right\} = +\infty
\end{align}
for any $C>0$, where $\boldsymbol{a}_{\boldsymbol{r}}(x) \coloneqq (a_1 x^{r_1}, \cdots, a_n x^{r_n})$.
We set $\kappa_1$ and $\kappa_2$ by 
\begin{equation}\label{shin}
F (\boldsymbol{a}_{\boldsymbol{r}}(x)) + \sum_{\beta\in A^-}|\boldsymbol{a}_{\boldsymbol{r}}(x)^\beta|^{q(\beta)}  \asymp x^{\kappa_1} \ \ \ \mbox{as}\ x\rightarrow \infty
\end{equation}
and
\begin{equation}\label{shin2}
\kappa_2 \coloneqq \max_{\beta\in A^-,\ \boldsymbol{a}^\beta \neq 0} q(\beta)( \beta \cdot \boldsymbol{r}).
\end{equation}
Then, $\kappa_1$ must satisfy
\begin{equation}\label{must}
\kappa_1 > \max_{1\le i\le n,\ a_i\neq 0} \frac{2d}{d-\alpha}r_i \eqqcolon \frac{2d}{d-\alpha}r^{\max} >0
\end{equation}
and $\kappa_1 \ge \kappa_2$.
We divide the proof into two parts: The cases $\kappa_1 > \kappa_2$ and $\kappa_1 = \kappa_2$.\\
\textbf{(In the case that $\boldsymbol{\kappa_1 > \kappa_2}$)}
Given $M\in\mathbb{N}$, we define $\theta^M=(\theta^M_1,\cdots,\theta^M_n)\in\mathbb{H}$ by
\[ \theta^M (t) \coloneqq - \langle \nabla \rangle^{\alpha} \boldsymbol{Z}_{M}^\eta (t) + \langle \nabla \rangle^{\alpha} \boldsymbol{a}_{\boldsymbol{r}}(g_M) + \sigma_{\alpha,M}^{\frac{1}{2}}\boldsymbol{b} \]
where $g_M$ is a deterministic function satisfying $\mbox{supp}(\hat g_M) \subset \{ |l|\le M\}$,
\begin{equation} 
\int_{\mathbb{T}^d} |g_M(x)|^{\frac{d}{d-\alpha}r^{\max}}dx \sim 1 
\end{equation}
and
\begin{equation}\label{asita}
\int_{\mathbb{T}^d} |g_M(x)|^{\frac{2d}{d-\alpha}r^{\max}}dx \sim M^d .
\end{equation}
See the proofs of the previous theorems for the definitions of the notations other than $g_M$. We write $\boldsymbol{a}_M^g \coloneqq \boldsymbol{a}_{\boldsymbol{r}}(g_M)$.
Then, similarly to the proof of Theorem \ref{nono} (ii), we  obtain
\begin{align}
&\lim_{L\rightarrow\infty}\lim_{N\rightarrow\infty}\mathbb{E} \Bigg[ -\left( R_N^F(\mathbb{Y}+I(\theta^M))\wedge L \right) + K\left| \int_{\mathbb{T}^d} \sum_{i=1}^n H_2 (P_N Y_i + P_N I_i(\theta^M);\sigma_{\alpha,N})dx    \right|^b \notag \\
&\ \ \ \  \ \ \ \ \ \ \ \ \ \ \ \ \ \ \ \ \  \ \ \ \ \ \ \ \ \ \ \ \ \ \ \ \ \ \ \ \ \ \ \ \ \ \ \  \ \ \ \ \ \ \ \ \ \ \ \ \ \ \ \ \ \ \ \ \ \ \ \ \ \ \ \ \ \ \ \ \ \ \ \ \   + \frac{1}{2}\int^1_0 \| P_N^{\otimes n} \theta^M (t)\|^2_{L^2(\mathbb{T}^d)}dt  \Bigg]\notag\\
&= -\int_{\mathbb{T}^d}  \left\{  F (\boldsymbol{a}_{M}^g) + \sum_{\beta\in A^-}|(\boldsymbol{a}_{M}(x)^g)^\beta|^{q(\beta)} \right\}  dx + \int_{\mathbb{T}^d} \sum_{\beta\in A^-}|(\boldsymbol{a}_{M}^g)^\beta|^{q(\beta)}   dx \notag\\
&\ - \mathbb{E} \left[ \sum_{\beta\in A}  c_\beta \sum_{0<\gamma\le\beta,\ \gamma\in\mathbb{N}^n} \binom{\beta}{\gamma} \sum_{0\le \tilde \gamma \le \gamma, \tilde \gamma \in \mathbb{N}^n} \binom{\gamma}{\tilde \gamma} \int_{\mathbb{T}^d} :\mathbb{Y}^{\tilde\gamma}: (- \mathbb{Y}_M(\eta) + \sigma_{\alpha,M}^{\frac{1}{2}}\boldsymbol{b})^{\gamma - \tilde \gamma} (\boldsymbol{a}_M^g)^{\beta-\gamma}dx  \right]\notag \\
&\ + \frac{1}{2}\mathbb{E} \left[ \sum_{i=1}^n  \int_0^1 \int_{\mathbb{T}^d}  (\langle \nabla \rangle ^{\alpha} Z^\eta_{i,M}(t))^2 + ( a_i \langle \nabla \rangle ^{\alpha} (g_M^{r_i}) + b_i \sigma_{\alpha,M}^{\frac{1}{2}})^2dxdt \right] \notag \\
&\ + \mathbb{E} \left[ K\left| \sum^n_{i=1} \int_{\mathbb{T}^d} \left\{ :Y_{i}^2: + 2Y_{i} (-Y_{i,M}(\eta) + a_i g_M^{r_i} + \sigma_{\alpha,M}^{\frac{1}{2}}b_i) + (-Y_{i,M}(\eta) + a_i g_M^{r_i} + \sigma_{\alpha, M}^{\frac{1}{2}}b_i)^2   \right\} dx    \right|^b    \right] \notag \\
&\eqqcolon A_M^1 + A_M^2 + A_M^3 + A_M^4 + A_M^5. \label{kurai21}
\end{align}
Then, noting that  $\sigma_{\alpha,M}\asymp M^{d-2\alpha}$ from \eqref{dyn},$\ b\le \frac{d}{d-2\alpha}$ and
\[ \int_{\mathbb{T}^d} g_M(x)^{2r^{\max}}dx \sim M^{d-2\alpha}, \]
which can be seen from Lemma \ref{haki},
we can see that there holds
\begin{equation}\label{tako1}
|A_M^4 + A_M^5| \lesssim M^d.
\end{equation}
Moreover, 
\begin{equation}\label{tako3}
A_M^1 \asymp -\int_{\mathbb{T}^d}|g_M|^{\kappa_1}dx\asymp -M^{h_1} \ \ \ \ \ \mbox{with}\ h_1 > d
\end{equation}
from \eqref{shin}, \eqref{must} and \eqref{asita}, and 
\begin{equation}\label{tako4}
|A_M^2| \lesssim \int_{\mathbb{T}^d}|g_M|^{\kappa_2}dx\asymp M^{h_2} \ \ \ \ \ \ \  \mbox{with}\ h_2 < h_1
\end{equation}
from \eqref{shin2} and $\kappa_1 > \kappa_2$.
Also, similarly to the proof of previous theorems, we can choose $0<\eta<1$ and $\boldsymbol{b}\in\mathbb{R}^n$ such that
\begin{equation}\label{tako5}
|A_{M}^3|\lesssim 1.
\end{equation}
From \eqref{kurai21}, \eqref{tako1}, \eqref{tako3}, \eqref{tako4} and \eqref{tako5}, we obtain
\begin{align*}
&\lim_{L\rightarrow\infty}\lim_{N\rightarrow\infty}\mathbb{E} \Bigg[ -\left( R_N^F(\mathbb{Y}+I(\theta^{M}))\wedge L \right) + K\left| \int_{\mathbb{T}^d} \sum_{i=1}^n H_2 (P_N Y_i + P_N I_i(\theta^{M});\sigma_{\alpha,N})dx    \right|^b \\
&\ \ \ \ \ \ \ \ \ \ \ \  \ \ \ \ \ \ \ \ \ \ \ \ \ \ \ \ \  \ \ \ \ \ \ \ \ \ \ \ \ \ \ \ \ \ \ \ \ \ \ \ \ \ \ \  \ \ \ \ \ \ \ \ \ \ \ \ \ \ \ \ \ \ \ \ \ \ \ \ \ \ \ + \frac{1}{2}\int^1_0 \| P_N^{\otimes n} \theta^M (t)\|^2_{L^2(\mathbb{T}^d)}dt  \Bigg] \\
&\lesssim -cM^{h_1} + M^{h_2} + 1 + M^d \rightarrow -\infty
\end{align*}
as $M\rightarrow \infty$.\\
\textbf{(In the case that $\boldsymbol{\kappa_1 = \kappa_2}$)}
Let $h_M$ be a deterministic function satisfying $\mbox{supp}(\hat h_M)\subset \{ |l|\le M\}$,
\begin{equation} 
\int_{\mathbb{T}^d} | h_M(x)|^{\frac{\kappa_1}{2}}dx \asymp 1 
\end{equation}
and
\begin{equation}\label{asatte}
\int_{\mathbb{T}^d} |h_M(x)|^{\kappa_1}dx \asymp M^d .
\end{equation}
Note that $h_M$ satisfies
\begin{equation}\label{case} 
\int_{\mathbb{T}^d} |h_M(x)|^{2r^{\max}}dx \asymp M^{\frac{4d}{\kappa_1}r^{\max} - d} \lesssim M^{d-2\alpha} 
\end{equation}
from \eqref{must} and Lemma \ref{haki}. Then, we redefine $\theta^M=(\theta^M_1,\cdots,\theta^M_n)\in\mathbb{H}$ by
\[ \theta^M (t) \coloneqq - \langle \nabla \rangle^{\alpha} \boldsymbol{Z}_{M}^\eta (t) + \langle \nabla \rangle^{\alpha} \boldsymbol{a}_{\boldsymbol{r}}(h_M) + \sigma_{\alpha,M}^{\frac{1}{2}}\boldsymbol{b}. \]
Then, from the exact same calculation as above, we obtain
\begin{align}
&\lim_{L\rightarrow\infty}\lim_{N\rightarrow\infty}\mathbb{E} \Bigg[ -\left( R_N^F(\mathbb{Y}+I(\theta^M))\wedge L \right) + K\left| \int_{\mathbb{T}^d} \sum_{i=1}^n H_2 (P_N Y_i + P_N I_i(\theta^M);\sigma_{\alpha,N})dx    \right|^b \notag \\
&\ \ \ \ \ \ \ \ \ \ \ \  \ \ \ \ \ \ \ \ \ \ \ \ \ \ \ \ \  \ \ \ \ \ \ \ \ \ \ \ \ \ \ \ \ \ \ \ \ \ \ \ \ \ \ \  \ \ \ \ \ \ \ \ \ \ \ \ \ \ \ \ \ \ \ \ \ \ \ \ \ \ \ + \frac{1}{2}\int^1_0 \| P_N^{\otimes n} \theta^M (t)\|^2_{L^2(\mathbb{T}^d)}dt  \Bigg]\notag\\
&= -\int_{\mathbb{T}^d}  \left\{  F (\boldsymbol{a}_{M}^h) + \sum_{\beta\in A^-}|(\boldsymbol{a}_{M}(x)^h)^\beta|^{q(\beta)} \right\}  dx + \int_{\mathbb{T}^d} \sum_{\beta\in A^-}|(\boldsymbol{a}_{M}^h)^\beta|^{q(\beta)}   dx    \notag\\
&\ - \mathbb{E} \left[ \sum_{\beta\in  A} c_\beta \sum_{0<\gamma\le\beta,\ \gamma\in\mathbb{N}^n} \binom{\beta}{\gamma} \sum_{0\le \tilde \gamma \le \gamma, \tilde \gamma \in \mathbb{N}^n} \binom{\gamma}{\tilde \gamma} \int_{\mathbb{T}^d} :\mathbb{Y}^{\tilde\gamma}: (- \mathbb{Y}_M(\eta) + \sigma_{\alpha,M}^{\frac{1}{2}}\boldsymbol{b})^{\gamma - \tilde \gamma} (\boldsymbol{a}_M^h)^{\beta-\gamma}dx  \right]\notag \\
&\ + \frac{1}{2}\mathbb{E} \left[ \sum_{i=1}^n  \int_0^1 \int_{\mathbb{T}^d}  (\langle \nabla \rangle ^{\alpha} Z^\eta_{i,M}(t))^2 + ( a_i \langle \nabla \rangle ^{\alpha} (h_M^{r_i}) + b_i \sigma_{\alpha,M}^{\frac{1}{2}})^2 dxdt\right] \notag \\
&\ + \mathbb{E} \left[ K\left| \sum^n_{i=1} \int_{\mathbb{T}^d} \left\{ :Y_{i}^2: + 2Y_{i} (-Y_{i,M}(\eta) + a_i h_M^{r_i} + \sigma_{\alpha,M}^{\frac{1}{2}}b_i) + (-Y_{i,M}(\eta) + a_i h_M^{r_i} + \sigma_{\alpha, M}^{\frac{1}{2}}b_i)^2   \right\} dx    \right|^b    \right] \notag \\
&\eqqcolon B_M^1 + B_M^2 + B_M^3 + B_M^4 + B_M^5. \label{kurai31}
\end{align}
Then, from $\kappa_1 = \kappa_2$, $\sigma_{\alpha,M}\asymp M^{d-2\alpha}$, $b\le \frac{d}{d-2\alpha}$, \eqref{asatte} and \eqref{case}, we obtain
\begin{equation}
B_M^1 + B_M^2 + B_M^4 + B_M^5  \lesssim M^d.
\end{equation}
Therefore, if we can choose $0<\eta<1$ and $\boldsymbol{b}\in\mathbb{R}^n$ such that
\begin{equation}\label{sakanaf}
B_{M}^3 \asymp -M^{h_3} \ \ \ \ \ \mbox{with}\ h_3 > d,
\end{equation}
there holds
\begin{align*}
&\lim_{L\rightarrow\infty}\lim_{N\rightarrow\infty}\mathbb{E} \Bigg[ -\left( R_N^F(\mathbb{Y}+I(\theta^{M}))\wedge L \right) + K\left| \int_{\mathbb{T}^d} \sum_{i=1}^n H_2 (P_N Y_i + P_N I_i(\theta^{M});\sigma_{\alpha,N})dx    \right|^b \\
&\ \ \ \ \ \ \ \ \ \ \ \ \ \  \ \ \ \ \ \ \ \ \ \ \ \ \ \ \ \ \  \ \ \ \ \ \ \ \ \ \ \ \ \ \ \ \ \ \ \ \ \ \ \ \ \ \ \  \ \ \ \ \ \ \ \ \ \ \ \ \ \ \ \ \ \ \ \ \ \ \ \ \ \ + \frac{1}{2}\int^1_0 \| P_N^{\otimes n} \theta^M (t)\|^2_{L^2(\mathbb{T}^d)}dt  \Bigg]\\
&\lesssim -cM^{h_3} + M^{d} \rightarrow -\infty
\end{align*}
as $M\rightarrow \infty$ and we can finish the proof. In the following, we prove the existence of such $0<\eta<1$ and $\boldsymbol{b}\in\mathbb{R}^n$. From the same kind of argument as the proofs of the previous theorems (see for example the argument around \eqref{hoshu2} and \eqref{hoshu} in the proof of Theorem \ref{nornon} (ii)), we get
\begin{equation}
B_M^3 \simeq - \mathbb{E} \left[ \sum_{\beta\in  A}  c_\beta \sum_{0<\gamma\le\beta,\ \gamma\in\mathbb{N}^n} \binom{\beta}{\gamma} \int_{\mathbb{T}^d} H_\gamma \left(  \frac{\mathbb{Y}_M -\mathbb{Y}_M(\eta)}{\sigma_{\alpha,M}^{\frac{1}{2}}} + \boldsymbol{b} ; 1  \right) \sigma_{\alpha,M}^{\frac{|\gamma|}{2}}h_M^{(\beta-\gamma) \cdot \boldsymbol{r}}  \boldsymbol{a}^{\beta-\gamma}dx  \right]. 
\end{equation}
Noting that  
\[ \sigma_{\alpha,M}^{\frac{|\gamma|}{2}}\int_{\mathbb{T}^d}h_M^{(\beta -\gamma)\cdot\boldsymbol{r}}dx \asymp M^{\frac{|\gamma|}{2}(d-2\alpha)+\frac{2d}{\kappa_1}(\beta -\gamma)\cdot \boldsymbol{r}-d}\]
from Lemma \ref{haki} and $\sigma_{\alpha,M} \asymp M^{d-2\alpha}$ as in \eqref{dyn}, we define 
\[ \kappa_3 \coloneqq \max\left\{ \frac{|\gamma|}{2}(d-2\alpha)+ \frac{2d}{\kappa_1}(\beta-\gamma)\cdot\boldsymbol{r}-d\ ;\ \beta\in A,\ \gamma\in\mathbb{N}^n,\ 0<\gamma\le\beta,\ \boldsymbol{a}^{\beta-\gamma}\neq 0   \right\}. \]
Then, from a similar argument to the last part of the proofs of Theorems \ref{nornon} (ii), we can choose $\eta$ and $\boldsymbol{b}$  such that
\begin{equation} \label{gan}
B_{M}^3 \asymp - M^{\kappa_3}.
\end{equation}
Moreover, from the assumption $\kappa_1 = \kappa_2$, \eqref{shin2} and \eqref{paka2},
\begin{align}
\kappa_3 &= \max \left\{ \frac{|\gamma|}{2}(d-2\alpha) + \frac{2d}{\kappa_1}(\beta-\gamma)\cdot\boldsymbol{r}-d \ ;\ \beta\in A,\ \gamma\in\mathbb{N}^n,\ 0< \gamma\le\beta,\ \boldsymbol{a}^{\beta-\gamma}\neq 0  \right\} \notag \\
&= \max \left\{ \frac{|\xi - \beta|}{2}(d-2\alpha) + \frac{2d}{\kappa_2}\beta\cdot\boldsymbol{r}-d \ ;\ \beta\in A^-,\ \xi\in A,\ \xi > \beta,\ \boldsymbol{a}^{\beta}\neq 0  \right\} \notag \\
&=\min_{(\eta\in A^-,\ \boldsymbol{a}^\eta \neq 0)} \max_{(\beta\in A^-,\ \xi\in  A,\ \xi > \beta,\ \boldsymbol{a}^{\beta}\neq 0 )}  \left\{ \frac{|\xi - \beta|}{2}(d-2\alpha) + \frac{2d(\beta\cdot\boldsymbol{r})}{q(\eta)(\eta\cdot\boldsymbol{r})}-d \right\} \notag \\
&> d. \label{pistol99}
\end{align}
From \eqref{gan} and \eqref{pistol99}, we obtain \eqref{sakanaf}.

\end{proof}

\section{The proof of the critical case}\label{criti}
In this section, we prove Theorem \ref{corocoro}. To prove Theorem \ref{corocoro} (i), it suffices to show that
\begin{align}
&\limsup_{N\rightarrow\infty} \inf_{\theta\in\mathbb{H}} \mathbb{E} \Bigg[ -R_N^{\lambda F}(\mathbb{Y}+I(\theta))+ K \left| \sum_{i=1,2} \int_{\mathbb{T}^d}  H_2 \left(P_N Y_i + P_N I_i(\theta);\sigma_{\frac{4}{9}d,N}\right)dx \right|^9 \notag \\
&\ \ \ \ \ \ \ \ \ \ \ \ \ \ \ \ \ \ \ \ \ \ \ \ \ \ \ \ \ \ \ \ \ \ \ \ \ \ \ \ \ \ \ \ \ \ \ \ \ \ \ \ \ \ \ \ \ \ \ \ \ \ \ \ \ \ \ \ \ \ \ \  + \frac{1}{2}\int^1_0 \| P_N^{\otimes n} \theta (t)\|^2_{L^2(\mathbb{T}^d)}dt \Bigg] >  -\infty  \label{moku9}
\end{align}
for sufficiently small $\lambda>0$ similarly to the proof of the previous theorems, where $R_N^{\lambda F}$ is defined by \eqref{rndef}. 
We are now in the situation that $b=\frac{d}{d-2\alpha}$, which means that we cannot use Lemma \ref{gogo} to control the second term in \eqref{moku9}.
However, there holds weak version of Lemma \ref{gogo}.
\begin{lemm}\label{tamanegi}
There holds
\[ \sum_{i=1,2}\mathbb{E} \left[ \| I_{i,N} (\theta) \|_{L^2 (\mathbb{T}^d)}^{\frac{2d}{d-2\alpha}} \right] \lesssim \mathbb{E} \left[  \left| \sum_{i=1,2}\int_{\mathbb{T}^d} \left( 2Y_{i,N} I_{i,N}(\theta) + I_{i,N}(\theta)^2 \right)dx \right|^{\frac{d}{d-2\alpha}} + \| I_{N} (\theta) \|_{H^\alpha(\mathbb{T}^d)}^2 \right] +1  \]
where the implicit constant is uniform in $\theta \in \mathbb{H}$ and $N\in \mathbb{N}$.
\end{lemm}
This lemma can be thought of a simple generalization of \cite[Lemma 3.6]{phi3} and can be proven by basically the same argument, so we omit the proof. (Letting $\alpha =1$ and $d=3$ in Lemma \ref{tamanegi}, we can essentially recover \cite[Lemma 3.6]{phi3}. )

\begin{proof}[Proof of Theorem \ref{corocoro} \textup{(i)}]
The proof is similar to that of Theorem \ref{nornon1} (i) and the only difference is we have to use Lemma \ref{corocoro} to control the second term in \eqref{moku9} instead of Lemma \ref{gogo}. 
For simplicity, we only prove for $F(x,y)=xy^3-100x^6$, but the same argument can be applied to any polynomial $F(x,y)=xy^3-Cx^6$ with $C\ge 100$. 
For $\lambda >0$, from \eqref{rndef},
\begin{align}
-R_N^{\lambda F} (\mathbb{Y}+I(\theta)) &= -\lambda \int_{\mathbb{T}^d} \left( Y_{1,N}+I_{1,N}(\theta) \right) H_3(Y_{2,N}+I_{2,N}(\theta);\sigma_{\frac{4}{9}d,N})dx \notag \\
&\  + 100\lambda \int_{\mathbb{T}^d} H_6 (Y_{1,N}+I_{1,N}(\theta);\sigma_{\frac{4}{9}d,N})dx \notag \\
&=  -\lambda \int_{\mathbb{T}^d} \bigg\{ Y_{1,N}:Y_{2,N}^3: + I_{1,N}(\theta):Y_{2,N}^3: + 3Y_{1,N}:Y_{2,N}^2:I_{2,N}(\theta) \notag \\
& \ \ \ \ \ \ \ \ \ \ \ \ \ \ \ \ \  + 3 I_{1,N}(\theta):Y_{2,N}^2:I_{2,N}(\theta) +3Y_{1,N}Y_{2,N}I_{2,N}(\theta)^2 + 3I_{1,N}(\theta)Y_{2,N}I_{2,N}(\theta)^2\notag \\
&\ \ \ \ \ \ \ \ \ \ \ \ \ \ \ \ \ \ \ \ \ \ \ \ \ \ \ \ \ \ \ \ \ \ \ \ \ \ \ \ \ \ \ \ \ \ \ \ \ \ \ \ \ \ \ \ \ \ \ \ \ \  +Y_{1,N}I_{2,N}(\theta)^3 +I_{1,N}(\theta)I_{2,N}(\theta)^3  \bigg\} dx \notag \\
&\ + 100\lambda \int_{\mathbb{T}^d}\sum_{l=1}^6 \binom{6}{l}:Y_{1,N}^l:I_{1,N}(\theta)^{6-l}dx + 100\lambda \int_{\mathbb{T}^d} I_{1,N}(\theta)^6 dx. \label{cara}
\end{align}
On the other hand, from Lemma \ref{tamanegi} and \eqref{new3}, 
\begin{align}
&\mathbb{E} \left[ K \left|\sum_{i=1,2}\int_{\mathbb{T}^d}\left( I_{i,N}(\theta)^2 + 2Y_{i,N}I_{i,N}(\theta) + :Y_{i,N}^2:  \right)dx \right|^9 + \frac{1}{2}\int_0^1 \| P_N^{\otimes 2}\theta (t) \|^2_{L^2(\mathbb{T}^d)}dt  \right] \notag \\
&\ge c_0 \| I_{N}(\theta)\|_{L^2(\mathbb{T}^d)}^{18} + \frac{1}{4}\| I_{N}(\theta)\|_{H^{\frac{4}{9}d}(\mathbb{T}^d)}^2 - C_1 \label{kara}
\end{align}
for some $c_0>0$ and $C_1>0$. Now, we estimate the each term of \eqref{cara} noting that the regularity of the Wick power $:Y_{1,N}^l:$ is $-\frac{l}{18}-\epsilon$ for any $\epsilon>0$, see Proposition \ref{gono} and \eqref{gonon}.
For $1\le l \le 5$, from the interpolation inequality (Lemma \ref{inter}),
\begin{align}
\left| \int_{\mathbb{T}^d} :Y_{1,N}^l:I_{1,N}(\theta)^{6-l} dx   \right| &\le \| :Y_{1,N}^l: \|_{W^{-\frac{l}{18}-\epsilon, \infty}(\mathbb{T}^d)} \| I_{1,N}(\theta)^{6-l}  \|_{W^{\frac{l}{18}+2\epsilon, 1}(\mathbb{T}^d)}\notag \\
& \le Q_N \| I_{1,N}(\theta)^{6-l}  \|_{L^{1}(\mathbb{T}^d)}^{1-\frac{l}{8}-\frac{9}{2d}\epsilon}  \| I_{1,N}(\theta)^{6-l}  \|_{W^{\frac{4}{9}d, 1}(\mathbb{T}^d)}^{\frac{l}{8}+\frac{9}{2d}\epsilon}. \label{hoho} 
\end{align}
Therefore, for $2\le l \le 5$, from the fractional Leibniz rule (Lemma \ref{fractional}) and Young's inequality,
\begin{align}
&\left| \int_{\mathbb{T}^d} :Y_{1,N}^l:I_{1,N}(\theta)^{6-l} dx   \right| \le Q_N \| I_{1,N}(\theta)  \|_{L^{6-l}(\mathbb{T}^d)}^{(6-l)(1-\frac{l}{8}-\frac{9}{2d}\epsilon)}  \left\{ \| I_{1,N}(\theta)  \|_{H^{\frac{4}{9}d}(\mathbb{T}^d)}  \| I_{1,N}(\theta)  \|_{L^{10-2l}(\mathbb{T}^d)}^{5-l}  \right\}^{\frac{l}{8}+\frac{9}{2d}\epsilon} \notag \\
& \le \delta \| I_{1,N}(\theta)  \|_{H^{\frac{4}{9}d}(\mathbb{T}^d)}^2 + \delta \| I_{1,N}(\theta) \|^6_{L^6(\mathbb{T}^d)} + C_\delta Q_N \label{kara2}
\end{align}
for any $\delta>0$ and $0< \epsilon \ll 1$. Note that random constants $Q_N$ can be different from line to line as mentioned in the beginning of Section \ref{norma1}.
For $l=1$, from \eqref{hoho}, the fractional Leibniz rule, Sobolev embedding and Young's inequality, 
\begin{align}
&\left| \int_{\mathbb{T}^d} :Y_{1,N}:I_{1,N}(\theta)^{5} dx   \right| \notag \\
&\le Q_N \|I_{1,N}(\theta)  \|_{L^5(\mathbb{T}^d)}^{5(\frac{7}{8}-\frac{9}{2d}\epsilon)}  \left\{ \| I_{1,N}(\theta)  \|_{H^{\frac{4}{9}d}(\mathbb{T}^d)}  \| I_{1,N}(\theta)  \|_{L^{18}(\mathbb{T}^d)}^{2} \| I_{1,N}(\theta)  \|_{L^{6}(\mathbb{T}^d)}^{2}  \right\}^{\frac{1}{8}+\frac{9}{2d}\epsilon} \notag \\
&\le Q_N \|I_{1,N}(\theta)  \|_{L^5(\mathbb{T}^d)}^{5(\frac{7}{8}-\frac{9}{2d}\epsilon)} \left\{ \| I_{1,N}(\theta)  \|_{H^{\frac{4}{9}d}(\mathbb{T}^d)}^3 \| I_{1,N}(\theta)  \|_{L^{6}(\mathbb{T}^d)}^{2}  \right\}^{\frac{1}{8}+\frac{9}{2d}\epsilon} \notag  \\ 
& \le \delta \| I_{1,N}(\theta)  \|_{H^{\frac{4}{9}d}(\mathbb{T}^d)}^2 + \delta \| I_{1,N}(\theta) \|^6_{L^6(\mathbb{T}^d)} + C_\delta Q_N \label{kara3}
\end{align}
for any $\delta>0$ and $0< \epsilon \ll 1$. The remaining terms of \eqref{cara} can be estimated as follows. For any $\delta>0$, we get 
\begin{align}
\left|\int_{\mathbb{T}^d} I_{1,N}(\theta):Y_{2,N}^3: dx\right| &\le \| :Y_{2,N}^3: \|_{W^{-\frac{1}{6}d-\epsilon,\infty}(\mathbb{T}^d)}\| I_{1,N}(\theta) \|_{W^{\frac{1}{6}d+2\epsilon,1}(\mathbb{T}^d)}\notag\\
&\le \delta \| I_{1,N}(\theta)\|_{H^{\frac{4}{9}d}(\mathbb{T}^d)}^2 + C_\delta Q_N, \label{bara1}
\end{align}
\begin{align}
\left|\int_{\mathbb{T}^d} Y_{1,N}:Y_{2,N}^2:I_{2,N}(\theta)      dx\right| &\le \|  Y_{1,N}:Y_{2,N}^2: \|_{W^{-\frac{1}{6}d-\epsilon,\infty}(\mathbb{T}^d)}\| I_{2,N}(\theta) \|_{W^{\frac{1}{6}d+2\epsilon,1}(\mathbb{T}^d)}\notag\\
&\le \delta \| I_{2,N}(\theta)\|_{H^{\frac{4}{9}d}(\mathbb{T}^d)}^2 + C_\delta Q_N, \label{bara2}
\end{align}
\begin{align}
&\left|\int_{\mathbb{T}^d}  I_{1,N}(\theta):Y_{2,N}^2:I_{2,N}(\theta)     dx\right| \notag \\ 
&\le \| :Y_{2,N}^2:  \|_{W^{-\frac{1}{9}d-\epsilon,\infty}(\mathbb{T}^d)}\| I_{1,N}(\theta) I_{2,N}(\theta) \|_{W^{\frac{1}{9}d+2\epsilon,1}(\mathbb{T}^d)}\notag\\
&\le Q_N \left\{  \| I_{1,N}(\theta)\|_{H^{\frac{4}{9}d}(\mathbb{T}^d)} \| I_{2,N}(\theta) \|_{L^2(\mathbb{T}^d)} +  \| I_{1,N}(\theta) \|_{L^2(\mathbb{T}^d)}\| I_{2,N}(\theta)\|_{H^{\frac{4}{9}d}(\mathbb{T}^d)} \right\} \notag \\
&\le \delta \| I_{1,N}(\theta)\|_{H^{\frac{4}{9}d}(\mathbb{T}^d)}^2 + \delta  \| I_{2,N}(\theta)\|_{H^{\frac{4}{9}d}(\theta)}^2  + \delta \| I_{1,N}(\theta) \|_{L^6(\mathbb{T}^d)}^{6}  +  \delta \| I_{2,N}(\theta) \|_{L^2(\mathbb{T}^d)} ^{18} + C_\delta Q_N, \label{bara3}
\end{align}
\begin{align}
\left|\int_{\mathbb{T}^d} Y_{1,N}Y_{2,N}I_{2,N}(\theta)^2      dx\right| &\le \| Y_{1,N}Y_{2,N}  \|_{W^{-\frac{1}{9}d-\epsilon,\infty}(\mathbb{T}^d)}\|  I_{2,N}(\theta)^2 \|_{W^{\frac{1}{9}d+2\epsilon,1}(\mathbb{T}^d)}\notag \\
&\le Q_N   \| I_{2,N}(\theta)\|_{H^{\frac{4}{9}d}(\mathbb{T}^d)} \| I_{2,N}(\theta) \|_{L^2(\mathbb{T}^d)}  \notag \\
&\le \delta  \| I_{2,N}(\theta)\|_{H^{\frac{4}{9}d}(\mathbb{T}^d)}^2  + \delta \| I_{2,N}(\theta) \|_{L^2(\mathbb{T}^d)}^{18} + C_\delta Q_N, \label{bara4}
\end{align}
and
\begin{align}
\left|\int_{\mathbb{T}^d} Y_{1,N}I_{2,N}(\theta)^3      dx\right| &\le \| Y_{1,N} \|_{W^{-\frac{1}{18}d-\epsilon,\infty}(\mathbb{T}^d)} \| I_{2,N}(\theta)^3 \|_{W^{\frac{1}{18}d+2\epsilon,1}(\mathbb{T}^d)} \notag\\
&\le Q_N \| I_{2,N}(\theta) \|_{H^{\frac{1}{18}d+2\epsilon}(\mathbb{T}^d)} \| I_{2,N}(\theta) \|_{L^4(\mathbb{T}^d)}^2\notag \\
&\le Q_N \left( \| I_{2,N}(\theta) \|_{L^{2}(\mathbb{T}^d)}^{\frac{7}{8}-\frac{9}{2d}\epsilon} \| I_{2,N}(\theta) \|_{H^{\frac{4}{9}d}(\mathbb{T}^d)}^{\frac{1}{8}+\frac{9}{2d}\epsilon}\right) \left( \| I_{2,N}(\theta) \|_{L^{2}(\mathbb{T}^d)}^{\frac{7}{8}} \| I_{2,N}(\theta) \|_{H^{\frac{4}{9}d}(\mathbb{T}^d)}^{\frac{9}{8}}\right) \notag \\
&\le \delta \| I_{2,N}(\theta)\|_{H^{\frac{4}{9}d}(\mathbb{T}^d)}^2  + \delta \| I_{2,N}(\theta) \|_{L^2(\mathbb{T}^d)}^{18} + C_\delta Q_N. \label{bara5}
\end{align}
Moreover,
\begin{align}
&\left|\int_{\mathbb{T}^d}  I_{1,N}(\theta)Y_{2,N}I_{2,N}(\theta)^2     dx\right| \notag \\  
&\le \| Y_{2,N} \|_{W^{-\frac{1}{18}d-\epsilon,\infty}(\mathbb{T}^d)} \| I_{1,N}(\theta)Y_{2,N}I_{2,N}(\theta)^2 \|_{W^{\frac{1}{18}d+2\epsilon,1}(\mathbb{T}^d)} \notag\\
&\le Q_N \left\{ \| I_{1,N}(\theta) \|_{H^{\frac{1}{18}d+2\epsilon}(\mathbb{T}^d)} \| I_{2,N}(\theta) \|_{L^4(\mathbb{T}^d)}^2 + \| I_{1,N}(\theta) \|_{L^{6}(\mathbb{T}^d)} \| I_{2,N}(\theta) \|_{H^{\frac{4}{9}d}(\mathbb{T}^d)}  \| I_{2,N}(\theta) \|_{L^{3}(\mathbb{T}^d)}  \right\}.\label{baratoku} 
\end{align}
Here, there holds
\[ \| I_{2,N}(\theta)\|_{L^4(\mathbb{T}^d)}^2 \lesssim \| I_{2,N}(\theta)\|_{H^{\frac{d}{4}}(\mathbb{T}^d)}^2 \lesssim \| I_{2,N}(\theta)\|_{L^{2}(\mathbb{T}^d)}^{\frac{7}{8}}\| I_{2,N}(\theta)\|_{H^{\frac{4}{9}d}(\mathbb{T}^d)}^{\frac{9}{8}}  ,\] 
\[ \| I_{2,N}(\theta)\|_{L^3(\mathbb{T}^d)} \lesssim \| I_{2,N}(\theta)\|_{H^{\frac{d}{6}}(\mathbb{T}^d)} \lesssim \| I_{2,N}(\theta)\|_{L^{2}(\mathbb{T}^d)}^{\frac{5}{8}}\| I_{2,N}(\theta)\|_{H^{\frac{4}{9}d}(\mathbb{T}^d)}^{\frac{3}{8}}  ,\]
and the first term in the rightmost hand of \eqref{baratoku} can be handled similarly to \eqref{bara5}. Therefore, we get
\begin{align}
&\left|\int_{\mathbb{T}^d}  I_{1,N}(\theta)Y_{2,N}I_{2,N}(\theta)^2     dx\right|\notag \\
&\le \delta \| I_{1,N}(\theta)\|_{H^{\frac{4}{9}d}(\mathbb{T}^d)}^2 + \delta \| I_{2,N}(\theta)\|_{H^{\frac{4}{9}d}(\mathbb{T}^d)}^2 + \delta \| I_{1,N}(\theta) \|_{L^6(\mathbb{T}^d)}^{6}  + \delta \| I_{2,N}(\theta) \|_{L^2(\mathbb{T}^d)}^{18} + C_\delta Q_N. \label{bara6}
\end{align}
And lastly, from Young's inequality, Sobolev embedding and interpolation inequality,
\begin{align}
\lambda \left|\int_{\mathbb{T}^d} I_{1,N}(\theta)I_{2,N}(\theta)^3      dx\right| &\le \lambda \left\{ \frac{1}{6}\| I_{1,N}(\theta)\|_{L^6(\mathbb{T}^d)}^6 + \frac{5}{6}\| I_{2,N}(\theta)\|_{L^{\frac{18}{5}}(\mathbb{T}^d)}^{\frac{18}{5}} \right\}\notag \\
&\le \lambda \left\{\frac{1}{6}\| I_{1,N}(\theta)\|_{L^6(\mathbb{T}^d)}^6 + c_1 \left(\| I_{2,N}(\theta) \|_{H^{\frac{4}{9}d}(\mathbb{T}^d)}^2 + \| I_{2,N}(\theta)\|_{L^2(\mathbb{T}^d)}^{18} \right)\right\} \label{bara7}
\end{align}
for some $c_1 >0$.
Combining \eqref{cara}--\eqref{bara5}, \eqref{bara6} and \eqref{bara7}, we get
\begin{align}
&\mathbb{E} \Bigg[ -R_N^{\lambda F}(\mathbb{Y}+I(\theta))+ K \left| \sum_{i=1,2}\int_{\mathbb{T}^d}  H_2 \left(P_N Y_i + P_N I_i(\theta);\sigma_{\frac{4}{9}d,N}\right)dx \right|^9  + \frac{1}{2}\int^1_0 \| P_N^{\otimes n} \theta (t)\|^2_{L^2(\mathbb{T}^d)}dt \Bigg] \notag \\
&\ge 50\lambda  \| I_{1,N}(\theta)\|_{L^6(\mathbb{T}^d)}^6 + \frac{c_0}{2}\| I_{2,N}(\theta)   \|_{L^2(\mathbb{T}^d)}^{18} + \frac{1}{8}\| I_N(\theta) \|^2_{H^{\frac{4}{9}d}(\mathbb{T}^d)} \notag \\
&\ \ \ \ \ \ \ \ \ \ \ \ \ \ \ \ \ \ \ \ \  -\lambda \left\{\frac{1}{6} \| I_{1,N}(\theta)\|_{L^6(\mathbb{T}^d)}^6 +c_1\| I_{2,N}(\theta) \|_{H^{\frac{4}{9}d}(\mathbb{T}^d)}^2 + c_1\| I_{2,N}(\theta)\|_{L^2(\mathbb{T}^d)}^{18} \right\}  - C  > -\infty \notag
\end{align}
uniformly in $\theta \in \mathbb{H}$ and $N\in \mathbb{N}$ if $\lambda \ll 1$. This proves \eqref{moku9}.

\end{proof}

Next, we prove Theorem \ref{corocoro} (ii).
Similarly to Section \ref{norma4}, it suffices to show 
\begin{align}
&\limsup_{L\rightarrow\infty}\limsup_{N\rightarrow\infty} \inf_{\theta\in\mathbb{H}} \mathbb{E} \Bigg[ -\min \left( R_N^{\lambda F}(\mathbb{Y}+I(\theta)), L \right) +  K\left| \sum_{i=1,2}\int_{\mathbb{T}^d}  H_2 (P_N Y_i + P_N I_i(\theta);\sigma_{\frac{4}{9}d,N})    \right|^b  \notag \\
& \ \ \ \ \ \ \ \ \ \ \ \ \ \ \ \ \ \ \ \ \ \ \ \ \ \ \ \ \ \ \ \ \ \ \ \ \ \ \ \ \ \ \ \ \ \ \ \ \ \ \ \ \ \ \ \ \ \ \ \ \ \ \ \ \ \ \ \ \ \ \   + \frac{1}{2}\int^1_0 \| P_N^{\otimes n} \theta (t)\|^2_{L^2(\mathbb{T}^d)}dt \Bigg] = -\infty  \label{garr}
\end{align}
for $\lambda \gg 1$.

\begin{proof}[Proof of Theorem \ref{corocoro} \textup{(ii)}]
The proof is similar to that of Theorem \ref{nornon1} (ii).
We only prove for $F(x,y)=xy^3-100x^6$, but similar argument also can be applied to any $F(x,y)=xy^3-Cx^6$ with $C>0$. We define $\theta^M(t)= (\theta_1^M(t),\theta_2^M(t))\in\mathbb{H}$ by 
\begin{gather}
\begin{cases}
\displaystyle \theta_1^M(t) \coloneqq -\langle \nabla \rangle^{\frac{4}{9}d} Z^\eta_{1,M}(t) + \alpha_M \langle \nabla \rangle^{\frac{4}{9}d}(g_M^3),  \\
\displaystyle \theta_2^M(t) \coloneqq -\langle \nabla \rangle^{\frac{4}{9}d} Z^\eta_{2,M}(t) + \beta_M \langle \nabla \rangle^{\frac{4}{9}d}(g_M^5) 
\end{cases}
\end{gather}
where $\eta\in(0,1)$ is an arbitrary fixed constant, $g_M$ is a deterministic function such that
\begin{equation}\label{tobi}
\int_{\mathbb{T}^d} |g_M(x)|^{9s}dx \asymp M^{sd-d}
\end{equation}
as $M\rightarrow \infty$ for any $s\in\mathbb{R}_+$, and $\alpha_M, \beta_M>0$ are defined by
\[ \beta_M^2 = \frac{\mathbb{E}[Y_{1,M}(\eta)^2 + Y_{2,M}(\eta)^2] }{\int_{\mathbb{T}^d}g_M^{10}dx}  = 2\frac{\mathbb{E}[Y_{2,M}(\eta)^2]}{\int_{\mathbb{T}^d}g_M^{10}dx}   \]
and
\[  \alpha_M\beta_M^3 -100\alpha_M^6 = \frac{1}{2}\alpha_M\beta_M^3 \iff 200\alpha_M^5 = \beta_M^3.\]
See \eqref{popo} for the definition of $Z^\eta_{i,M}(t)$. Such $g_M$ actually exists in view of Lemma \ref{haki}.
Note that there holds $\alpha_M \asymp \beta_M \asymp 1$ as $M\rightarrow \infty$ because 
\[ \mathbb{E}[Y_{i,M}(\eta)^2] = \eta \sigma_{\frac{4}{9}d,M} \asymp M^{\frac{d}{9}}   \]
and
\[ \int_{\mathbb{T}^d}g_M^{10}dx \asymp M^{\frac{d}{9}}   \]
from \eqref{dyn} and \eqref{tobi}. We write $\boldsymbol{a}_M = (\alpha_M g_M^3, \beta_M g_M^5)$ in the following.
Then, for $N>M$,
\begin{align}
&-R_N^{\lambda F}(\mathbb{Y}+I(\theta^M)) \notag \\
&= -\lambda \int_{\mathbb{T}^d} F(\boldsymbol{a}_M)dx -\lambda \sum_{\beta\in\{ (1,3), (6,0) \}}c_\beta \sum_{0<\gamma \le \beta, \gamma\in\mathbb{N}^2} \binom{\beta}{\gamma}\int_{\mathbb{T}^d} H_\gamma \left( \mathbb{Y}_N - \mathbb{Y}_M(\eta); \sigma_{\frac{4}{9}d,N}  \right)\boldsymbol{a}_M^{\beta -\gamma}dx \notag \\
&= -\lambda \int_{\mathbb{T}^d} F(\boldsymbol{a}_M)dx -\lambda \sum_{\beta\in\{ (1,3), (6,0) \}}c_\beta \sum_{0<\gamma \le \beta} \binom{\beta}{\gamma} \sum_{0\le \tilde \gamma \le \gamma}\binom{\gamma}{\tilde \gamma} \int_{\mathbb{T}^d} :\mathbb{Y}_N^{\tilde \gamma}:\left(-\mathbb{Y}_M(\eta)^{\gamma- \tilde\gamma}\right) \boldsymbol{a}_M^{\beta -\gamma}dx \notag \\ 
&\rightarrow -\lambda \int_{\mathbb{T}^d} F(\boldsymbol{a}_M)dx-\lambda \sum_{\beta\in\{ (1,3), (6,0) \}}c_\beta \sum_{0<\gamma \le \beta} \binom{\beta}{\gamma} \sum_{0\le \tilde \gamma \le \gamma} \binom{\gamma}{\tilde \gamma} \int_{\mathbb{T}^d} :\mathbb{Y}^{\tilde \gamma}:\left( -\mathbb{Y}_M(\eta)^{\gamma- \tilde\gamma}\right) \boldsymbol{a}_M^{\beta -\gamma}dx \notag \\ 
&\eqqcolon A_M + B_{M} 
\end{align}
almost surely as $N\rightarrow \infty$, where $c_{(1,3)}=1$ and $c_{(6,0)}=-100$.
Therefore, 
\begin{align}
&\lim_{L\rightarrow\infty}\lim_{N\rightarrow\infty} \mathbb{E} \Bigg[ -\min \left( R_N^{\lambda F}(\mathbb{Y}+I(\theta^M)), L \right) +  K\left|\sum_{i=1,2} \int_{\mathbb{T}^d}  H_2 (P_N Y_i + P_N I_i(\theta^M);\sigma_{\alpha,N})    \right|^b  \notag \\
& \ \ \ \ \ \ \ \ \ \ \ \ \ \ \ \ \ \ \ \ \ \ \ \ \ \ \ \ \ \ \ \ \ \ \ \ \ \ \ \ \ \ \ \ \ \ \ \ \ \ \ \ \ \ \ \ \ \ \ \ \ \ \ \ \ \ \ \ \ \ \ \ \ \ \ \ \ \ \ \ \ \   + \frac{1}{2}\int^1_0 \| P_N^{\otimes n} \theta^M (t)\|^2_{L^2(\mathbb{T}^d)}dt \Bigg] \notag \\
&= \mathbb{E} \left[ A_M + B_M  \right] + \mathbb{E} \left[ K\left| C_M   \right|^b \right] + \frac{1}{2}\mathbb{E} \left[ \int^1_0 \| \theta^M (t)\|^2_{L^2(\mathbb{T}^d)}dt  \right],\label{kekka}
\end{align}
where 
\begin{align}
C_M \coloneqq  &\int_{\mathbb{T}^d}  \left\{ :Y_{1}^2: - 2Y_{1}Y_{1,M}(\eta) + 2Y_{1}\alpha_M g_M^3 + (-Y_{1,M}(\eta)+ \alpha_Mg_M^3)^2  \right\}dx    \notag \\ 
& +  \int_{\mathbb{T}^d}  \left\{ :Y_{2}^2: - 2Y_{2}Y_{2,M}(\eta) + 2Y_{2}\beta_M g_M^5 + (-Y_{2,M}(\eta)+ \beta_Mg_M^5)^2  \right\}dx. \label{revise}
\end{align}
From Lemma \ref{gausstoku}, \eqref{tobi} and $\sigma_{\frac{4}{9}d,M}\asymp M^{\frac{d}{9}}$,
\begin{align}
\mathbb{E} \left[ B_M  \right] &= -\lambda \mathbb{E} \left[ \sum_{\beta\in\{ (1,3), (6,0) \}}c_\beta \sum_{0<\gamma \le \beta, \gamma\in\mathbb{N}^2} \binom{\beta}{\gamma} \sum_{0\le \tilde \gamma \le \gamma, \tilde \gamma \in\mathbb{N}^2} \binom{\gamma}{\tilde \gamma} \int_{\mathbb{T}^d} :\mathbb{Y}_M^{\tilde \gamma}:\left( -\mathbb{Y}_M(\eta)^{\gamma- \tilde\gamma}\right) \boldsymbol{a}_M^{\beta -\gamma}dx \right] \notag \\ 
&= -\lambda \mathbb{E} \left[ \sum_{\beta\in\{ (1,3), (6,0) \}}c_\beta \sum_{0<\gamma \le \beta, \gamma\in\mathbb{N}^2} \binom{\beta}{\gamma}  \int_{\mathbb{T}^d} H_\gamma \left( \mathbb{Y}_M - \mathbb{Y}_M(\eta); \sigma_{\frac{4}{9}d,M}  \right)        \boldsymbol{a}_M^{\beta -\gamma}dx \right] \notag \\ 
&\lesssim \mathbb{E} \left[ \sum_{\beta\in\{ (1,3), (6,0) \}} \sum_{0<\gamma \le \beta, \gamma\in\mathbb{N}^2} \binom{\beta}{\gamma}  \int_{\mathbb{T}^d} H_\gamma \left( \frac{\mathbb{Y}_M - \mathbb{Y}_M(\eta)}{\sigma_{\frac{4}{9}d,M}^{\frac{1}{2}}}; 1  \right) \sigma_{\frac{4}{9}d,M}^{\frac{|\gamma|}{2}}       | \boldsymbol{a}_M^{\beta -\gamma}|dx \right] \notag \\ 
&\lesssim \max_{\beta\in\{ (1,3), (6,0)\},0<\gamma \le \beta}  M^{ \frac{|\gamma|}{18}d+\frac{(\beta-\gamma)\cdot (3,5)}{9}d-d   } = M^{\frac{13}{18}d}.\label{kekka1}
\end{align}
Moreover, from \eqref{tobi} and $\alpha_M \asymp \beta_M \asymp 1$ as $M\rightarrow \infty$,
\begin{align}
A_M = -\lambda \int_{\mathbb{T}^d} F(\boldsymbol{a}_M)dx &= -\lambda \left( \alpha_M\beta_M^3 - 100\alpha_M^6   \right) \int_{\mathbb{T}^d} g_M^{18} dx \notag \\
&= -\frac{\lambda}{2}\alpha_M\beta_M^3  \int_{\mathbb{T}^d} g_M^{18} dx \asymp  -\lambda M^d. \label{kekka2}
\end{align}
On the other hand, from \eqref{revise}, we can write
\begin{align}
C_M &= \sum_{i=1,2} \int_{\mathbb{T}^d}  \left\{ :Y_{i}^2: - 2\left( Y_{i}Y_{i,M}(\eta) - \mathbb{E}[ Y_{i}Y_{i,M}(\eta)]\right) + (Y_{i,M}(\eta)^2 - \mathbb{E}[Y_{i,M}(\eta)^2])   \right.    \notag\\
&\ \ \ \ \ \  + \beta_M^2 g_M^{10} + 2\beta_M (Y_{2}-Y_{2,M}(\eta))g_M^5 + \sum_{i=1,2} \left(-2\mathbb{E}[ Y_{i}Y_{i,M}(\eta)] + \mathbb{E}[Y_{i,M}(\eta)^2] \right)  \notag\\
&\ \ \ \ \ \ + \left. \alpha_M^2 g_M^6 + 2\alpha_M (Y_1 - Y_{1,M}(\eta))g_M^3 \right\}dx.
\end{align}
Here, from the definition of $\beta_M$, there holds
\[ \int_{\mathbb{T}^d} \Big\{ \beta_M^2 g_M^{10} + \sum_{i=1,2} \left( -2\mathbb{E}[ Y_{i}Y_{i,M}(\eta)] + \mathbb{E}[Y_{i,M}(\eta)^2]\right) \Big\}dx =  \int_{\mathbb{T}^d} \left\{ \beta_M^2 g_M^{10} -2\mathbb{E}[Y_{1,M}(\eta)^2]\right\}dx =0  \]
and from Lemma \ref{haki}, \eqref{tobi} and $\mbox{supp}(\hat g_M)\subset \{ |l|\le M\}$,
\begin{align*}
\left|\int_{\mathbb{T}^d} (Y_{2}-Y_{2,M}(\eta))g_M^5 dx \right|&\le \| Y_{2}-Y_{2,M}(\eta)  \|_{W^{-\frac{d}{18}-\epsilon,\infty}(\mathbb{T}^d)} \| g_M^5  \|_{W^{\frac{d}{18}+2\epsilon,1}(\mathbb{T}^d)}\\
&\le Q_M  \| g_M  \|_{H^{\frac{d}{18}+2\epsilon}(\mathbb{T}^d)}  \| g_M  \|_{L^{9}(\mathbb{T}^d)}^4 \le Q_M M^{\frac{d}{18}+2\epsilon} \| g_M \|_{L^{2}(\mathbb{T}^d)} \le Q_M.
\end{align*}
Similarly, we get
\begin{align*}
\left|\int_{\mathbb{T}^d} (Y_{1}-Y_{1,M}(\eta))g_M^3 dx \right| \le Q_M.
\end{align*}
Therefore, noting that
\[ \left| \int_{\mathbb{T}^d}\left\{ Y_{i}Y_{i,M}(\eta) - \mathbb{E}[ Y_{i}Y_{i,M}(\eta)] \right\} dx\right| \le Q_M \   \]
and 
\[\left| \int_{\mathbb{T}^d}\left\{ Y_{i,M}(\eta)^2 - \mathbb{E}[Y_{i,M}(\eta)^2]\right\} dx\right|  \le Q_M, \]
which can be proven similarly to Proposition \ref{gono},
and
\[\left| \int_{\mathbb{T}^d}  \alpha_M^2 g_M^{6} dx \right| \lesssim 1 \]
which follows from $\alpha_M \asymp 1$ and \eqref{tobi},
we get
\begin{equation}\label{kekka3}
\mathbb{E} \left[  K \left| C_M   \right|^b           \right] \lesssim 1
\end{equation}
uniformly in $M \in\mathbb{N}$. Moreover, from Lemmas \ref{hq3} and \ref{haki},
\begin{align}
&\frac{1}{2}\mathbb{E} \left[ \int_0^1 \|  \theta^M(t) \|^2_{L^2(\mathbb{T}^d)}dt  \right]  \notag\\
&= \frac{1}{2}\mathbb{E} \left[ \int_0^1  \int_{\mathbb{T}^d} \left\{ \left(\langle \nabla \rangle^{\frac{4}{9}d}Z^\eta_{1,M}(t)  \right) ^2 + \left(\alpha_M \langle \nabla \rangle^{\frac{4}{9}d}(g_M^3)  \right) ^2    + \left(\langle \nabla \rangle^{\frac{4}{9}d}Z^\eta_{2,M}(t)  \right) ^2 + \left(\alpha_M \langle \nabla \rangle^{\frac{4}{9}d}(g_M^5)  \right) ^2   \right\}    dx     dt  \right]  \notag\\
&\lesssim M^d.\label{kekka4}
\end{align}
From \eqref{kekka}, \eqref{kekka1}, \eqref{kekka2}, \eqref{kekka3} and \eqref{kekka4}, we obtain 
\begin{align}
&\lim_{L\rightarrow\infty}\lim_{N\rightarrow\infty} \mathbb{E} \Bigg[ -\min \left( R_N^{\lambda F}(\mathbb{Y}+I(\theta^M)), L \right) +  K\left| \int_{\mathbb{T}^d}  H_2 (P_N Y_2 + P_N I_2(\theta^M);\sigma_{\alpha,N})    \right|^b  \notag \\
& \ \ \ \ \ \ \ \ \ \ \ \ \ \ \ \ \ \ \ \ \ \ \ \ \ \ \ \ \ \ \ \ \ \ \ \ \ \ \ \ \ \ \ \ \ \ \ \ \ \ \ \ \ \ \ \ \ \ \ \ \ \ \ \ \ \ \ \ \ \ \ \ \ \ \ \ \ \ \ \ \ \   + \frac{1}{2}\int^1_0 \| P_N^{\otimes n} \theta^M (t)\|^2_{L^2(\mathbb{T}^d)}dt \Bigg] \notag \\
& \lesssim -c\lambda M^d +M^{\frac{13}{18}d}+ M^d + C \rightarrow -\infty \notag
\end{align}
as $M\rightarrow \infty$ for sufficiently large $\lambda >0$.

\end{proof}

\appendix 
\section{Convergence of Wick products}\label{appen}

In this section, we give a proof of the convergence of Wick products and thus the well-definedness of the renormalized density function of the Gibbs measure.
Recall that we introduced the following notations in Sections \ref{jai} and \ref{varia}. For $\gamma \in \mathbb{N}^n$ and $\mathbb{Y}_N = P_N^{\otimes n}\mathbb{Y} = (Y_{1,N}, \cdots, Y_{n,N})$,
\begin{equation} \label{ssss}
:\mathbb{Y}_N^\gamma: \coloneqq H_\gamma \left( \mathbb{Y}_N;\sigma_{\alpha,N} \right), 
\end{equation}
where
\begin{equation}\label{sssss} 
\sigma_{\alpha,N} \coloneqq \sum_{l\in\mathbb{Z}^d,|l|\le N}\frac{1}{(1+|l|^2)^{\alpha}} = \mathbb{E} \left[ Y_{i,N}^2 \right]. 
\end{equation}

\begin{prop}\label{gono}
Let $\gamma \in \mathbb{N}^n$ and $|\gamma| = \gamma_1 + \cdots + \gamma_n \eqqcolon k \in \mathbb{N}$. Then, if $\alpha> \frac{k-1}{2k}d$, $:\mathbb{Y}_N^\gamma:$ converges in $W^{(\alpha -\frac{d}{2})k-\epsilon, \infty}(\mathbb{T}^d)$ almost surely as $N \rightarrow \infty$ for any $\epsilon>0$.
\end{prop}

\begin{proof}
We only prove 
\begin{equation}\label{gonon}
\mathbb{E} \left[ \| :\mathbb{Y}_N^\gamma:  \|^p_{W^{(\alpha -\frac{d}{2})k-\epsilon, \infty}(\mathbb{T}^d)}  \right] < \infty
\end{equation}
uniformly in $N\in\mathbb{N}$ for $p \gg 1$.
The almost sure convergence of $:\mathbb{Y}_N^\gamma:$ can be proven by considering the difference $:\mathbb{Y}_N^\gamma: - :\mathbb{Y}_M^\gamma:$ and applying a standard argument with Borel-Cantelli Lemma, see \cite[Proposition 2.1]{wave}, for example.
From the equivalence of the moments of Gaussian polynomials, 
\begin{align}
\mathbb{E} \left[ \| :\mathbb{Y}_N^\gamma:  \|^{2p}_{W^{\beta, 2p}(\mathbb{T}^d)}  \right] &= \mathbb{E} \left[ \| \langle \nabla \rangle^\beta  \left(:\mathbb{Y}_N^\gamma:\right)  \|^{2p}_{L^{2p}(\mathbb{T}^d)}  \right] \notag \\
&\lesssim \int_{\mathbb{T}^d} \mathbb{E} \left[ \left| \langle \nabla \rangle^\beta  \left(:\mathbb{Y}_N^\gamma:\right) (x)     \right|^2   \right]^p dx. \label{kakaka}
\end{align}
From Lemma \ref{gauss}, \eqref{ssss} and \eqref{sssss}, 
\begin{align}
&\mathbb{E} \left[ \left| \langle \nabla \rangle^\beta  \left(:\mathbb{Y}_N^\gamma:\right) (x)     \right|^2   \right] \notag\\
&=  \langle \nabla \rangle^\beta_x \langle \nabla \rangle^\beta_y\   \mathbb{E} \left[ :\mathbb{Y}_N^\gamma: (x) :\mathbb{Y}_N^\gamma: (y)        \right]\  \bigg|_{x=y}  \notag\\
&= \langle \nabla \rangle^\beta_x \langle \nabla \rangle^\beta_y\ \prod_{i=1}^n  \mathbb{E} \left[ Y_{i,N}(x)Y_{i,N}(y)       \right]^{\gamma_i}\  \bigg|_{x=y}  .
\end{align}
Because we can calculate as 
\begin{align}
\mathbb{E} \left[ Y_{i,N}(x)Y_{i,N}(y)   \right] &\simeq \mathbb{E} \left[ \sum_{l\in\mathbb{Z}^d,\ |l|\le N}\frac{1}{\langle l \rangle^\alpha}\beta_i^l(t)e_l(x)  \sum_{l\in\mathbb{Z}^d,\ |l|\le N}\frac{1}{\langle l \rangle^\alpha}\beta_i^l(t)e_l(y)   \right] \notag \\
&=  \sum_{l\in\mathbb{Z}^d,\ |l|\le N}\frac{1}{\langle l \rangle^{2\alpha}}e_l(x) e_{-l} (y),\notag
\end{align}
we get
\begin{align}
&\mathbb{E} \left[ \left| \langle \nabla \rangle^\beta  \left(:\mathbb{Y}_N^\gamma:\right) (x)     \right|^2   \right] \notag\\
&\le \sum_{l_1 , \cdots, l_k \in \mathbb{Z}^d}  \frac{1}{\langle l_1 \rangle^{2\alpha} \cdots \langle l_k \rangle^{2\alpha}}\langle l_1 + \cdots + l_k \rangle^{2\beta} < +\infty  
\end{align}
if $\beta < (\alpha -\frac{d}{2})k$ and $\alpha>\frac{k-1}{2k}d$. This can be checked by applying \cite[Lemma 2.3]{arata}, for example. Therefore, from \eqref{kakaka} and Sobolev embedding, we get \eqref{gonon} by taking sufficiently large $p>1$. 

\end{proof}

\end{document}